\documentclass{amsart}
\usepackage{amsmath}
\usepackage{amssymb, color}
\usepackage[margin=1in]{geometry}
\usepackage{graphicx}
\usepackage{subfigure}
\usepackage{float}
\usepackage{hyperref}
\usepackage{epsf}
\usepackage{titletoc}
\usepackage{eufrak}
\usepackage{ulem}
\usepackage{algorithm, algorithmicx}
\usepackage{algpseudocode}
\usepackage{multirow}
\usepackage{bm}
\usepackage{arydshln}
\usepackage{pifont}

\newtheorem{theorem}{Theorem}[section]

\newtheorem{remark}{Remark}[section]
\newtheorem{definition}{Definition}[section]
\newtheorem{proposition}{Proposition}[section]
\newtheorem{lemma}{Lemma}[section]
\algnewcommand{\LineComment}[1]{\State \(\triangleright\) #1}

\graphicspath{{figure/}}

\newcommand{\bd}{\boldsymbol}
\renewcommand{\tilde}{\widetilde}
\renewcommand{\hat}{\widehat}
\newcommand{\vct}[1]{\bd{#1}}

\newcommand{\CalP}{{\mathcal{P}}}

\newcommand{\OO}{\mathbb{O}}
\newcommand{\RR}{\mathbb{R}}

\newcommand{\Or}{\mathcal{O}}

\newcommand{\Cov}{\text{Cov}}
\newcommand{\Cost}{\text{Cost}}

\DeclareMathOperator{\diag}{diag}

\newcommand{\rd}{\mathrm{d}}

\title{A sparse decomposition of low rank symmetric positive semi-definite matrices}
\author{Thomas Y. Hou}
\thanks{Applied and Comput. Math, Caltech, Pasadena, CA 91125. {\it Email: hou@cms.caltech.edu.}}
\author{Qin Li}
\thanks{Math, UW-Madison, Madison, WI 53705. {\it Email: qinli@math.wisc.edu.}}
 \author{Pengchuan Zhang}
\thanks{Applied and Comput. Math, Caltech, Pasadena, CA 91125. {\it Email: pzzhang@cms.caltech.edu.}}

\date{\today}

\begin{document}
\maketitle

\begin{abstract}
Suppose that $A \in \RR^{N \times N}$ is symmetric positive semidefinite with rank $K \le N$. Our goal is to decompose $A$ into $K$ rank-one matrices $\sum_{k=1}^K g_k g_k^T$ where the modes $\{g_{k}\}_{k=1}^K$ are required to be as sparse as possible. In contrast to eigen decomposition, these sparse modes are not required to be orthogonal. Such a problem arises in random field parametrization where $A$ is the covariance function and is intractable to solve in general. In this paper, we partition the indices from 1 to $N$ into several patches and propose to quantify the sparseness of a vector by the number of patches on which it is nonzero, which is called patch-wise sparseness. Our aim is to find the decomposition which minimizes the total patch-wise sparseness of the decomposed modes. We propose a domain-decomposition type method, called intrinsic sparse mode decomposition (ISMD), which follows the ``local-modes-construction + patching-up" procedure. The key step in the ISMD is to construct local pieces of the intrinsic sparse modes by a joint diagonalization problem. Thereafter a pivoted Cholesky decomposition is utilized to glue these local pieces together. Optimal sparse decomposition, consistency with different domain decomposition and robustness to small perturbation are proved under the so called regular-sparse assumption (see Definition~\ref{def:regular_sparse}). We provide simulation results to show the efficiency and robustness of the ISMD. We also compare the ISMD to other existing methods, e.g., eigen decomposition, pivoted Cholesky decomposition and convex relaxation of sparse principal component analysis~\cite{LaiLuOsher_2014, vu2013fantope}.
\end{abstract}

\section{Introduction}
Many problems in science and engineering lead to huge \textit{symmetric} and \textit{positive semi-definite} (PSD) matrices. Often they arise from the discretization of self-adjoint PSD operators or their kernels, especially in the context of data science and partial differential equations.

Consider a symmetric PSD matrix of size $N \times N$, denoted as $A$. Since $N$ is typically large, this causes serious obstructions when dealing numerically with such problems. Fortunately in many applications the discretization $A$ is low-rank or approximately low-rank, i.e., there exists $\{\psi_1, \dots, \psi_K\} \subset \RR^N$ for $K \ll N$ such that
\begin{equation*}
	A = \sum_{k=1}^K \psi_k \psi_k^T \quad \text{or} \quad \|A - \sum_{k=1}^K \psi_k \psi_k^T\|_2 \le \epsilon,
\end{equation*}
respectively. Here, $\epsilon>0$ is some small number and $\|A\|_2 = \lambda_{max}(A)$ is the largest eigenvalue of $A$. To obtain such a low-rank decomposition/approximation of $A$, the most natural method is perhaps the eigen decomposition with $\{\psi_k\}_{k=1}^K$ as the eigenvectors corresponding to the largest $K$ eigenvalues of $A$. An additional advantage of the eigen decomposition is the fact that eigenvectors are orthogonal to each other. However, eigenvectors are typically dense vectors, i.e., every entry is typically nonzero. 

For a symmetric PSD matrix $A$ with rank $K \ll N$, the aim of this paper is to find an alternative decomposition
\begin{equation} \label{eqn:sparsedecom}
	A = \sum_{k=1}^K g_k g_k^T.
\end{equation}
Here the number of components is still its rank $K$, which is optimal, and the modes $\{g_k\}_{k=1}^K$ are required to be as sparse as possible. In this paper, we work on the symmetric PSD matrices, which are typically the discretized self-adjoint PSD operators or their kernels. We could have just as well worked on the self-adjoint PSD operators. This would correspond to the case when $N=\infty$. Much of what will be discussed below applies equally well to this case.

Symmetric PSD matrices/operators/kernels appear in many science and engineering branches and various efforts have been made to seek sparse modes. In statistics, sparse Principal Component Analysis (PCA) and its convex relaxations~\cite{Jolliffe_2003, Zou_PCA_06, dAspremont_sparsePCA, vu2013fantope} are designed to sparsify the eigenvectors of data covariance matrices. In quantum chemistry, Wannier functions~\cite{wannier1937structure, kohn1959image} and other methods~\cite{marzari1997maximally, weinan2010localized, marzari2012maximally, ozolins_compressed_2013, LaiLuOsher_2014} have been developed to obtain a set of functions that approximately span the eigenspace of the Hamitonian, but are spatially localized or sparse. In numerical homogenization of elliptic equations with rough coefficients~\cite{hou_multiscale_1997, hou_PetrovGalerkin_2004, efendiev2011multiscale, owhadi_polyharmonic_2014, owhadi2015multi}, a set of multiscale basis functions is constructed to approximate the eigenspace of the elliptic operator and is used as the finite element basis to solve the equation. In most cases, sparse modes reduce the computational cost for further scientific experiments. Moreover, in some cases sparse modes have a better physical interpretation compared to the global eigen-modes. Therefore, it is of practical importance to obtain sparse (localized) modes.

\subsection{Our results} \label{sec:ourresults}
The number of nonzero entries of a vector $\psi \in \RR^N$ is called its $l^0$ norm, denoted by $\|\psi\|_0$. Since the modes in \eqref{eqn:sparsedecom} are required to be as sparse as possible, the sparse decomposition problem is naturally formulated as the following optimization problem
\begin{equation}\label{opt:minsupport}
	\boxed{\min_{\psi_1, \dots, \psi_K \in \RR^N}\quad \sum_{k=1}^{K}\|\psi_k\|_0 \quad \mathrm{s.t.} \quad A = \sum_{k=1}^K \psi_k \psi_k^T\, .}
\end{equation}
However, this problem is rather difficult to solve because: first, minimizing $l^0$ norm results in a combinatorial problem and is computationally intractable in general; second, the number of unknown variables is $K\times N$ where $N$ is typically a huge number. Therefore, we introduce the following patch-wise sparseness as a surrogate of $\|\psi_k\|_0$ and make the problem computationally tractable.

\begin{definition}[Patch-wise sparseness] \label{def:gsparseness}
Suppose that $\CalP = \{P_m\}_{m=1}^M$ is a disjoint partition of the $N$ nodes, i.e., $[N] \equiv \{1, 2, 3, \dots, N\} = \sqcup_{m=1}^M P_m$. The patch-wise sparseness of $\psi \in \RR^N$ with respect to the partition $\CalP$, denoted by $s(\psi;\CalP)$, is defined as
\begin{equation*}
s(\psi;\CalP) = \#\{P \in \CalP : \psi|_{_{P}} \neq \vct{0}\}.
\end{equation*}
\end{definition}

Throughout this paper, $[N]$ denotes the index set $\{1, 2, 3, \dots, N\}$; $\vct{0}$ denotes the vectors with all entries equal to 0; $|P|$ denotes the cardinality of a set $P$; $\psi|_{_{P}} \in \RR^{|P|}$ denotes the restriction of $\psi \in \RR^N$ on patch $P$. Once the partition $\CalP$ is fixed, smaller $s(\psi;\CalP)$ means that $\psi$ is nonzero on fewer patches, which implies a sparser vector. With patch-wise sparseness as a surrogate of the $l^0$ norm, the sparse decomposition problem \eqref{opt:minsupport} is relaxed to
\begin{equation}\label{opt:minsparseness}
	\boxed{\min_{\psi_1, \dots, \psi_K \in \RR^N}\quad \sum_{k=1}^{K} s(\psi_k;\CalP) \quad \mathrm{s.t.} \quad A = \sum_{k=1}^K \psi_k \psi_k^T\, .}
\end{equation}
If $\{g_k\}_{k=1}^K$ is an optimizer for \eqref{opt:minsparseness}, we call them a set of {\it intrinsic sparse modes} for $A$ under partition $\CalP$. Since the objective function of problem \eqref{opt:minsparseness} only takes nonnegative integer values, we know that for a symmetric PSD matrix $A$ with rank $K$, there exists at least one set of intrinsic sparse modes.

It is obvious that the intrinsic sparse modes depend on the domain partition $\CalP$. Two extreme cases would be $M = N$ and $M=1$. For $M = N$, $s(\psi;\CalP)$ recovers $\|\psi\|_0$ and the patch-wise sparseness minimization problem \eqref{opt:minsparseness} recovers the original $l^0$ minimization problem \eqref{opt:minsupport}. Unfortunately, it is computationally intractable. For $M = 1$, every non-zero vector has sparseness one, and thus the number of nonzero entries makes no difference. However, in this case the problem \eqref{opt:minsparseness} is computationally tractable. For instance, a set of (unnormalized) eigenvectors is one of the optimizers. We are interested in the sparseness defined in between, namely, a partition with a meso-scale patch size. Compared to $\|\psi\|_0$, the meso-scale partition sacrifices some resolution when measuring the support, but makes the optimization \eqref{opt:minsparseness} efficiently solvable. Specifically, Problem~\eqref{opt:minsparseness} with the following \textit{regular-sparse} partitions enjoys many good properties. These properties enable us to design a very efficient algorithm to solve Problem~\eqref{opt:minsparseness}.
\vspace{-1.1mm}
\begin{definition}[regular-sparse partition]\label{def:regular_sparse}
The partition $\CalP$ is regular-sparse with respect to $A$ if there exists a decomposition $A = \sum_{k=1}^K g_k g_k^T$ such that all nonzero modes on each patch $P_m$ are linearly independent.
\end{definition}
\vspace{-1.1mm}
If two intrinsic sparse modes are non-zero on exactly the same set of patches, which are called unidentifiable modes in Definition~\ref{def:identifiable}, it is easy to see that any rotation of these unidentifiable modes forms another set of intrinsic sparse modes. From a theoretical point of view, if a partition is regular-sparse with respect to A, the intrinsic sparse modes are unique up to rotations of unidentifiable modes, see Theorem~\ref{thm:ISMD}. Moreover, as the partition gets refined, the original identifiable intrinsic sparse modes remain unchanged, while the original unidentifiable modes become identifiable and become sparser (in the sense of $l^0$ norm), see Theorem~\ref{thm:ISMDconsistency}. In this sense, the intrinsic sparse modes are independent of the partition that we use. From a computational point of view, a regular-sparse partition ensures that the restrictions of the intrinsic sparse modes on each patch $P_m$ can be constructed from rotations of local eigenvectors. Following this idea, we propose the intrinsic sparse mode decomposition (ISMD), see Algorithm~\ref{alg:ISMD}. In Theorem~\ref{thm:ISMD}, we have proved that the ISMD solves problem~\eqref{opt:minsparseness} exactly on regular-sparse partitions. We point out that, even when the partition is not regular-sparse, numerical experiments show that the ISMD still generates a sparse decomposition of $A$.

The ISMD consists of three steps. In the first step, we perform eigen decomposition of $A$ restricted on local patches $\{P_m\}_{m=1}^M$, denoted as $\{A_{mm}\}_{m=1}^M$, to get $A_{mm} = H_m H_m^T$. Here, columns of $H_m$ are the unnormalized local eigenvectors of $A$ on patch $P_m$. In the second step, we recover the local pieces of intrinsic sparse modes, denoted by $G_m$, by rotating the local eigenvectors $G_m = H_m D_m$. The method to find the right local rotations $\{D_m\}_{m=1}^M$ is the core of the ISMD. All the local rotations are coupled by the decomposition constraint $A = \sum_{k=1}^K g_k g_k^T$ and it seems impossible to solve $\{D_m\}_{m=1}^M$ from this big coupled system. Surprisingly, when the partition is regular-sparse, this coupled system can be decoupled and every local rotation $D_m$ can be solved independently by a joint diagonalization problem~\eqref{opt:jointDm}. In the last ``patch-up" step, we identify correlated local pieces across different patches by the pivoted Cholesky decomposition of a symmetric PSD matrix $\Omega$ and then glue them into a single intrinsic sparse mode. Here, $\Omega$ is the projection of $A$ onto the subspace spanned by all the local pieces $\{G_m\}_{m=1}^M$, see Eqn.~\eqref{def:omega}. This step is necessary to reduce the number of decomposed modes to the optimal $K$, i.e., the rank of $A$. The last step also equips the ISMD the power to identify long range correlation and to honor the intrinsic correlation structure hidden in $A$. The popular $l^1$ approach typically does not have this property. 

The ISMD has very low computational complexity.  There are two reasons for its efficiency: first of all, instead of computing the expensive global eigen decomposition, we compute only the local eigen decompositions of $\{A_{mm}\}_{m=1}^M$; second, there is an efficient algorithm to solve the joint diagonalization problems for the local rotations $\{D_m\}_{m=1}^M$. Moreover, because both performing the local eigen decompositions and solving the joint diagonalization problems can be done independently on each patch, the ISMD is embarrassingly parallelizable.

The stability of the ISMD is also explored when the input data $A$ is mixed with noises. We study the small perturbation case, i.e., $\hat{A} = A + \epsilon \tilde{A}$. Here, $A$ is the noiseless rank-$K$ symmetric PSD matrix, $\tilde{A}$ is the symmetric additive perturbation and $\epsilon > 0$ quantifies the noise level. A simple thresholding step is introduced in the ISMD to achieve our aim: \textit{to clean up the noise $\epsilon \tilde{A}$ and to recover the intrinsic sparse modes of $A$}. Under some assumptions, we can prove that sparse modes $\{\hat{g}_k\}_{k=1}^K$, produced by the ISMD with thresholding, exactly capture the supports of $A$'s intrinsic sparse modes $\{g_k\}_{k=1}^K$ and the error $\|\hat{g}_k - g_k\|$ is small. See Section~\ref{sec:perturbation} for a precise description. 

We have verified all the theoretical predictions with numerical experiments on several synthetic covariance matrices of high dimensional random vectors. Without parallel execution, for partitions with a large range of patch sizes, the computational cost of the ISMD is comparable to that of the partial eigen decomposition~\cite{sorensen1992implicit, lehoucq1996deflation}. For certain partitions, the ISMD could be 10 times faster than the partial eigen decomposition. We have also implemented the convex relaxation of sparse PCA~\cite{LaiLuOsher_2014, vu2013fantope} and compared these two methods. It turns out that the convex relaxation of sparse PCA fails to capture the long range correlation, needs to perform (partial) eigen decomposition on matrices repeatedly for many times and is thus much slower than the ISMD. Moreover, we demonstrate the robustness of the ISMD on partitions which are not regular-sparse and on inputs which are polluted with small noises.

\subsection{Applications}\label{sec:applications}
The ISMD leads to a sparse-orthogonal matrix factorization for any matrix. Given a matrix $X \in \RR^{N \times M}$ of rank $K$ and a partition $\CalP$ of the index set $[N]$, the ISMD tries to solve the following optimization problem:
\begin{equation}\label{eqn:ismd_svd}
	\boxed{\min_{\stackrel{g_1, \dots, g_K \in \RR^N}{u_1, \dots, u_K \in \RR^M}}\quad \sum_{k=1}^{K} s(g_k;\CalP) \quad \mathrm{s.t.} \quad X = \sum_{k=1}^K g_k u_k^T\,, \quad u_k^T u_{k'} = \delta_{k,k'} \,~~~~ \forall 1 \le k, k' \le K,}
\end{equation}
where $s(g_k;\CalP)$ is the patch-wise sparseness defined in Definition~\eqref{def:gsparseness}. Compared to the bi-orthogonal property of SVD, the ISMD requires orthogonality only in one dimension and requires sparsity in the other dimension. The method to obtain the decomposition~\eqref{eqn:ismd_svd} consists of three steps: first, compute $A = X X^T$; second, apply the ISMD to $A$ to get $\{g_k\}_{k=1}^K$; third, project $X$ on to $\{g_k\}_{k=1}^K$ to obtain $\{u_k\}_{k=1}^K$.

The sparse-orthogonal matrix factorization~\eqref{eqn:ismd_svd} has potential applications in statistics, machine learning and uncertainty quantification. In statistics and machine learning, latent factor models with sparse loadings have found many applications ranging from DNA microarray analysis~\cite{gao2012sparse}, facial and object recognition~\cite{wagner2012toward}, web search models~\cite{agarwal2009regression} and etc. Specifically, latent factor models decompose a data matrix $X \in \RR^{N\times M}$ by product of the loading matrix $G \in \RR^{N\times K}$ and the factor value matrix $U \in \RR^{M \times K}$, with possibly small noise $E \in \RR^{N\times M}$, i.e.,
\begin{equation} \label{eqn:latentfactormodel}
	X = G U^T + E.
\end{equation} 
The sparse-orthogonal matrix factorization~\eqref{eqn:ismd_svd} tries to find the optimal sparse loadings $G$ under the condition that latent factors are {\it normalized} and {\it uncorrelated}, i.e., columns in $U$ are orthonormal. In practice, the uncorrelated latent factors make lots of sense, but is not guaranteed by many existing matrix factorization methods, e.g., non-negative matrix factorization (NMF)~\cite{lee1999learning}, sparse PCA~\cite{Jolliffe_2003, Zou_PCA_06, dAspremont_sparsePCA}, structured sparse PCA~\cite{jenatton2010structured}.

In uncertainty quantification (UQ), we often need to parametrize a random field, denoted as $\kappa(x,\omega)$, with a finite number of random variables. Applying the ISMD to its covariance function, denoted by $\Cov(x,y)$, we can get a parametrization with $K$ random variables:
\begin{equation} \label{eqn:sparseKL}
	\kappa(x,\omega) = \bar{\kappa}(x) + \sum_{k=1}^K g_k(x) \eta_k(\omega), 
\end{equation} 
where $\bar{\kappa}(x)$ is the mean field, the physical modes $\{g_k\}_{k=1}^K$ are sparse/localized, and the random variables $\{\eta_k\}_{k=1}^K$ are {\it centered}, {\it uncorrelated}, and have {\it unit variance}. The parametrization~\eqref{eqn:sparseKL} has a form similar to the widely used Karhenen-Lo\`eve (KL) expansion~\cite{karhunen_uber_1947,loeve_probability_1977}, but in the KL expansion the physical modes $\{g_k\}_{k=1}^K$ are eigenfunctions of the covariance function and are typically nonzero everywhere. Obtaining a sparse parametrization is important to uncover the intrinsic sparse feature in a random field and to achieve computational efficiency for further scientific experiments. In~\cite{hou_LocalgPC_2016}, such sparse parametrization methods are used to design efficient algorithms to solve partial differential equations with random inputs.

\subsection{Connection with the sparse matrix factorization problem}\label{subsec:connections}
Given a matrix $X \in \RR^{N\times M}$ of $M$ columns corresponding to $M$ observations in $\RR^N$, a sparse matrix factorization problem is to find a matrix $G=[g_1, \dots, g_r] \in \RR^{N \times r}$, called {\it dictionary}, and a matrix $U=[u_1, \dots, u_r] \in \RR^{M \times r}$, called {\it decomposition coefficients}, such that $G U^T$ approximates $X$ well and the columns in $G$ are sparse.

In~\cite{lee2006efficient, witten2009penalized, mairal2010online}, the authors formulated this problem as an optimization problem by penalizing the l1 norm of G, i.e. $\|G\|_1 := \sum_{k=1}^r \|g_k\|_1$, to enforce the sparsity of the dictionary. This can be written as
\begin{equation}\label{eqn:matrix_factorize}
	\boxed{\min_{G \in \RR^{N \times r}, U \in \RR^{M \times r}}\quad \|X - G U^T\|_F^2 + \lambda \|G\|_1  \quad \mathrm{s.t.} \quad \|u_k\|_2 \le 1 \,~~~~ \forall 1 \le k \le r,}
\end{equation}
where the parameter $\lambda>0$ controls to what extent the dictionary $G$ is regularized. We point out that the l1 penalty can be replaced by other penalties. For example,  the structured sparse PCA~\cite{jenatton2010structured} uses certain l1/l2 norm of $G$ to enforce sparsity with specific structures, e.g. rectangular structure on a grid. Problem~\eqref{eqn:matrix_factorize} is not jointly convex in $(G, U)$. Certain specially designed algorithms have been developed to solve this optimization problem. We will discuss one of these methods in Section~\ref{sec:othermethods}.

There are two major differences between the optimization problem~\eqref{eqn:ismd_svd} and the optimization problem~\eqref{eqn:matrix_factorize}. First, the ISMD, which is designed to solve~\eqref{eqn:ismd_svd}, requires that the decomposition coefficients $U$ be orthonormal, while many other methods, including sparse PCA and structured sparse PCA, which are designed to solve~\eqref{eqn:matrix_factorize}, only normalize every columns in $U$. One needs to decide whether the orthogonality in $U$ is necessary in her application and choose the appropriate method. Second, the number of modes $K$ in the ISMD must be the rank of the matrix, while the number of modes $r$ in problem~\eqref{eqn:matrix_factorize} is picked by users and can be any number. In other words, the ISMD is seeking an {\it exact matrix decomposition}, while other methods make a trade-off between the accuracy $\|X - G U^T\|_F$ and the sparsity $\|G\|_1$ by recovering the matrix approximately instead of obtaining an exact recovery. Although the ISMD can be modified to do matrix approximation (with the orthogonality constraint on $U$), see Algorithm~\ref{alg:ISMDm2}, the optimal sparsity of the dictionary $G$ is not guaranteed anymore. {\it Based on these two differences, we recommend the ISMD for sparse matrix factorization problems where the orthogonality in decomposition coefficients $U$ is required and an exact (or nearly exact) decomposition is desired.} In our upcoming paper~\cite{hou_sparseOC1_2016,hou_sparseOC2_2016}, we will present our recent results on solving Problem~\eqref{eqn:matrix_factorize}.

\subsection{Outlines}
In Section~\ref{sec:algorithm} we present our ISMD algorithm for low rank matrices, analyze its computational complexity and talk about its relation with other methods for sparse decomposition or approximation. In Section~\ref{sec:exactdecom} we present our main theoretical results, i.e., Theorem~\ref{thm:ISMD} and Theorem~\ref{thm:ISMD}. In Section~\ref{sec:perturbmodify}, we discuss the stability of the ISMD by performing perturbation analysis. We also provide two modified ISMD algorithms: Algorithm~\ref{alg:ISMDm1} for low rank matrix with small noise, and Algorithm~\ref{alg:ISMDm2} for sparse matrix approximation. Finally, we present a few numerical examples in Section~\ref{sec:numericalExamples} to demonstrate the efficiency of the ISMD and compare its performance with other existing methods.

\section{Intrinsic Sparse Mode Decomposition}\label{sec:algorithm}
In this section, we present the algorithm of the ISMD and analyze its computational complexity. Its relation with other matrix decomposition methods is discussed in the end of this section. In the rest of the paper, $\OO(n)$ denotes the set of real unitary matrices of size $n\times n$; $\mathbb{I}_{n}$ denotes the identity matrix with size $n \times n$.

\subsection{ISMD}\label{sec:ISMDalg}
Suppose that we have one symmetric positive symmetric matrix, denoted as $A \in \RR^{N\times N}$, and a partition of the index set $[N]$, denoted as $\CalP = \{P_m\}_{m=1}^M$. The partition typically originates from the physical meaning of the matrix $A$. For example, if $A$ is the discretized covariance function of a random field on domain $D \subset \RR^d$, $\CalP$ is constructed from certain domain partition of $D$. The submatrix of $A$, with row index in $P_m$ and column index in $P_n$, is denoted as $A_{mn}$. To simplify our notations, we assume that indices in $[N]$ are rearranged such that $A$ is written as below:
\begin{align}\label{def:Ablock}
A = \left[\begin{array}{c:c:c:c}
A_{11} & A_{12} & \cdots & A_{1M}\\
\hdashline
A_{21}& A_{22} & \cdots & A_{2M}\\
\hdashline
\vdots & \vdots & \ddots & \vdots \\
\hdashline
A_{M1} & A_{M2} & \cdots & A_{MM} \end{array}\right] \,.
\end{align}
Notice that when implementing the ISMD, there is no need to rearrange the indices as above. The ISMD tries to find the optimal sparse decomposition of $A$ w.r.t. partition $\CalP$, defined as the minimizer of problem~\eqref{opt:minsparseness}. The ISMD consists of three steps: local decomposition, local rotation, and global patch-up.

In the first step, we perform eigen decomposition 
\begin{equation} \label{eqn:localeig}
	A_{mm} = \sum_{i=1}^{K_m} \gamma_{m,i} h_{m,i} h_{m,i}^T \equiv H_m H_m^T,
\end{equation}
where $K_m$ is the rank of $A_{mm}$ and $H_m = [\gamma_{m,1}^{1/2} h_{m,i}\,, \gamma_{m,2}^{1/2} h_{m,2}\,, \dots\, \gamma_{m,K_m}^{1/2} h_{m,K_m}]$. If $A_{mm}$ is ill-conditioned, we truncate the small eigenvalues and a truncated eigen decomposition is used as follows:
 \begin{equation} \label{eqn:localeigtruncate}
	A_{mm} \approx \sum_{i=1}^{K_m} \gamma_{m,i} h_{m,i} h_{m,i}^T \equiv H_m H_m^T.
\end{equation}

Let $K_{(t)} \equiv \sum_{m=1}^M K_m$ be the total local rank of $A$. We extend columns of $H_m$ into $\RR^N$ by adding zeros, and get the block diagonal matrix 
\begin{equation*}
	H_{ext} = \text{diag}\{H_1, H_2, \cdots, H_M\}.
\end{equation*}
The correlation matrix with basis $H_{ext}$, denoted by $\Lambda \in \RR^{K_{(t)} \times K_{(t)}}$, is the matrix such that 
\begin{equation} \label{def:lambda}
	A = H_{ext} \Lambda H_{ext}^T.
\end{equation}
Since columns of $H_{ext}$ are orthogonal and span a space that contains $\text{range}(A)$, $\Lambda$ exists and can be computed block-wisely as follows:
\begin{align}\label{def:Lambda}
\Lambda = \left[\begin{array}{c:c:c:c}
\Lambda_{11} & \Lambda_{12} & \cdots & \Lambda_{1M}\\
\hdashline
\Lambda_{21}& \Lambda_{22} & \cdots & \Lambda_{2M}\\
\hdashline
\vdots & \vdots & \ddots & \vdots \\
\hdashline
\Lambda_{M1} & \Lambda_{M2} & \cdots & \Lambda_{MM} \end{array}\right] \, ,  \quad 
\Lambda_{mn} = H_m^{\dagger} A_{mn} \left(H_n^{\dagger}\right)^T \in \RR^{K_m \times K_n}\,.
\end{align}
where $H_m^{\dagger} \equiv (H_m^T H_m)^{-1} H_m^T$ is the (Moore-Penrose) pseudo-inverse of $H_m$.

In the second step, on every patch $P_m$, we solve the following joint diagonaliziation problem to find a local rotation $D_m$:
\begin{equation}\label{opt:jointDm}
\boxed{\begin{split}
	\min_{V\in \OO(K_m)}\quad \sum_{n=1}^{M} \sum_{i \neq j} |(V^T \Sigma_{n;m} V)_{i,j}|^2 \,,
\end{split}}
\end{equation}
in which 
\begin{equation}\label{def:Sigma}
	\Sigma_{n;m} \equiv \Lambda_{mn}\Lambda_{mn}^T.
\end{equation}
We rotate the local eigenvectors with $D_m$ and get $G_m = H_m D_m$. Again, we extend columns of $G_m$ into $\RR^N$ by adding zeros, and get the block diagonal matrix 
\begin{equation*}
	G_{ext} = \text{diag}\{G_1, G_2, \cdots, G_M\}.
\end{equation*}
The correlation matrix with basis $G$, denoted by $\Omega \in \RR^{K_{(t)} \times K_{(t)}}$, is the matrix such that 
\begin{equation} \label{def:omega}
	A = G_{ext} \Omega G_{ext}^T.
\end{equation}
With $\Lambda$ in hand, $\Omega$ can be obtained as follows:
\begin{equation}\label{eqn:Omega}
	\Omega = D^T \Lambda D\,, \qquad D = \text{diag}\{D_1, D_2, \cdots, D_M\}.
\end{equation}
Joint diagonalization has been well studied in the blind source separation (BSS) community. We present some relevant theoretical results in Appendix~\ref{appendix:jointDiagonalization}. A Jacobi-like algorithm~\cite{cardoso_blind_1993,bunse-gerstner_numerical_1993}, see Algorithm~\ref{alg:jointD}, is used in our paper to solve problem~\eqref{opt:jointDm}. For most cases, we may want to normalize the columns of $G_{ext}$ and put all the magnitude information in $\Omega$, i.e.,
\begin{equation} \label{eqn:renormalize}
	G_{ext} = \bar{G}_{ext} E, \quad \bar{\Omega} = E \Omega E^T,
\end{equation}
where $E$ is a diagonal matrix with $E_{ii}$ being the $l^2$ norm of the $i$-th column of $G_{ext}$, $\bar{G}_{ext}$ and $\bar{\Omega}$ will substitute the roles of $G$ and $\Omega$ in the rest of the algorithm.

In the third step, we use the pivoted Cholesky decomposition to patch up the local pieces $G_m$. Specifically, suppose the pivoted Cholesky decomposition of $\Omega$ is given as
\begin{equation}\label{eqn:OmegaPivotChol}
	\Omega = PLL^T P^T\,,
\end{equation}
where $P\in \RR^{K_{(t)} \times K_{(t)}}$ is a permutation matrix and $L \in \RR^{K_{(t)}\times K}$ is a lower triangular matrix with positive diagonal entries. Since $A$ has rank $K$, both $\Lambda$ and $\Omega$ have rank $K$. This is why $L$ only has $K$ nonzero columns. However, we point out that the rank $K$ is automatically identified in the algorithm instead of given as an input parameter. Finally, $A$ is decomposed as
\begin{equation}\label{eqn:ISMD}
	A = G G^T \equiv G_{ext}PL(G_{ext}PL)^T\,.
\end{equation}
The columns in $G$ ($G_{ext}PL$) are our decomposed sparse modes. 

The full algorithm is summarized in Algorithm~\ref{alg:ISMD}. We point out that there are two extreme cases for the ISMD:
\begin{itemize}
\item The coarsest partition $\CalP = \{[N]\}$. In this case, the ISMD is equivalent to the standard eigen decomposition.
\item The finest partition $\CalP = \left\{\{i\} : i \in [N] \right\}$. In this case, the ISMD is equivalent to the pivoted Cholesky factorization on $\bar{A}$ where $\bar{A}_{ij} = \frac{A_{ij}}{\sqrt{A_{ii}A_{jj}}}$. If the normalization~\eqref{eqn:renormalize} is applied, the ISMD is equivalent to the pivoted Cholesky factorization of $A$ in this case. 
\end{itemize}
In these two extreme cases, there is no need to use the joint diagonalization step and it is known that in general neither the ISMD nor the pivoted Cholesky decomposition generates sparse decomposition. When $\CalP$ is neither of these two extreme cases, the joint diagonalization is applied to rotate the local eigenvectors and thereafter the generated modes are patch-wise sparse. Specifically, when the partition is regular-sparse, the ISMD generates the optimal patch-wise sparse decomposition as stated in Theorem~\ref{thm:ISMD}.

\begin{algorithm}
\caption{Intrinsic sparse mode decomposition}\label{alg:ISMD}
\begin{algorithmic}[1]
\Require{$A \in \RR^{N\times N}$: symmetric and PSD; $\CalP=\{P_m\}_{m=1}^M$: partition of index set $[N]$}
\Ensure{$G = [g_1, g_2, \cdots, g_K]$: $K$ is the rank of $A$, $A = G G^T$}
\LineComment{Local eigen decomposition} \label{line:start}
\For{$m=1,2,\cdots,M$}
	\State Local eigen decomposition: $A_{mm} = H_m H_m^T$
\EndFor
\LineComment{Assemble correlation matrix $\Lambda$}
\State Assemble $\Lambda = H_{ext}^{\dagger} A \left(H_{ext}^{\dagger}\right)^T$ block-wisely as in Eqn.~\eqref{def:Lambda} 
\LineComment{Joint Diagonalization}
\For{$m=1,2,\cdots,M$}
	\For{$n=1,2,\cdots,M$}
		\State $\Sigma_{n;m} = \Lambda_{mn}\Lambda_{mn}^T$
	\EndFor
	\State Solve the joint diagonalization problem~\eqref{opt:jointDm} for $D_m$ \Comment{Use Algorithm~\ref{alg:jointD}}
\EndFor  \label{line:end}
\LineComment{Assemble correlation matrix $\Omega$ and its pivoted Cholesky decomposition}
\State $\Omega = D^T \Lambda D$  \label{line:assembleOmega}
\State $\Omega = P L L^T P^T$  \label{line:pivotChol}
\LineComment{Assemble the intrinsic sparse modes $G$}
\State $G = H_{ext} D P  L$
\end{algorithmic}
\end{algorithm}

\begin{remark}
One can interpret $H_m$ as the patch-wise amplitude and $D_m$ as the patch-wise phase. The patch-wise amplitude is easy to obtain using a local eigen decomposition~\eqref{eqn:localeig}, while the patch-wise phase is obtained by the joint diagonalization~\eqref{opt:jointDm}.

In fact, the ISMD solves the following optimization problem where we jointly diagonalize $A_{mn}$:
\begin{equation}\label{opt:jointGm0}
\boxed{\begin{split}
	\min_{G_m \in \RR^{|P_m| \times K_m}} \quad & \sum_{n=1}^{M} \sum_{i \neq j} |B_{n;m}(i,j)|^2 \, \\
	\text{s.t.}   \quad & G_m G_m^T = A_{mm}\,, \\
			& G_m B_{n;m} G_m^T = A_{mn} A_{nn}^{\dagger} A_{mn}^T,
\end{split}}
\end{equation}
in which $A_{nn}^{\dagger} = \sum_{i=1}^{K_n} \gamma_{n,i}^{-1} h_{n,i} h_{n,i}^T$ is the (Moore-Penrose) pseudo-inverse of $A_{nn}$. Eqn.~\eqref{opt:jointGm0} is not a unitary joint diagonalization problem, i.e., the variable $G_m$ is not unitary. The ISMD solves this non-unitary joint diagonalization problem in two steps:
\begin{enumerate}
\item Perform a local eigen decomposition $A_{mm} = H_m H_m^T$. Then the feasible $G_m$ can be written as $H_m D_m$ with a unitary matrix $D_m$.
\item Find the rotation $D_m$ that solves the unitary joint diagonalization problem~\eqref{opt:jointDm}.
\end{enumerate}
\end{remark}

\subsection{Computational complexity} \label{sec:complexity}
The main computational cost of the ISMD comes from the local KL expansion, the joint diagonalization, and the pivoted Cholesky decomposition. To simplify the analysis, we assume that the partition $\CalP$ is uniform, i.e., each group has $\frac{N}{M}$ nodes. On each patch, we perform eigen decomposition of $A_{mm}$ of size $N/M$ and rank $K_m$. Then, the cost of the local eigen decomposition step is
\begin{equation*}
	\Cost_1 = \sum_{m=1}^M \Or\left( (N/M)^2 K_m \right) = (N/M)^2 \Or(\sum_{m=1}^M K_m).
\end{equation*}
For the joint diagonalization, the computational cost of Algorithm~\ref{alg:jointD} is
\begin{equation*}
	\sum_{m=1}^M N_{corr, m} K_m^3 N_{iter,m}\,.
\end{equation*}
Here, $N_{corr, m}$ is the number of nonzero matrices in $\{\Sigma_{n;m}\}_{n=1}^M$. Notice that $\Sigma_{n;m} \equiv \Lambda_{mn}\Lambda_{mn}^T = 0$ if and only if $A_{mn} = 0$. Therefore, $N_{corr,m}$ may be much smaller than $M$ if $A$ is sparse. Nevertheless, we take an upper bound $M$ to estimate the cost. $N_{corr, m} K_m^3$ is the computational cost for each sweeping in Algorithm~\ref{alg:jointD} and $N_{iter,m}$ is the number of iterations needed for the convergence. The asymptotic convergence rate is shown to be quadratic~\cite{bunse-gerstner_numerical_1993}, and we see no more than 6 iterations needed in our numerical examples. Therefore, we can take $N_{iter, m} = \Or(1)$ and in total we have
\begin{equation*}
	\Cost_2 = \sum_{m=1}^M M \Or(K_m^3) = M \Or(\sum_{m=1}^M K_m^3).
\end{equation*}
Finally, the pivoted Cholesky decomposition of $\Omega$, which is of size $\sum_{k=1}^M K_m$, has cost
\begin{equation*}
	\Cost_3 = \Or\left( (\sum_{k=1}^M K_m) K^2 \right) = K^2 \Or(\sum_{m=1}^M K_m).
\end{equation*}
Combining the computational costs in all three steps, we conclude that the total computational cost of the ISMD is 
\begin{equation} \label{eqn:cost1}
	\Cost_{\text{ISMD}} =  \left( (N/M)^2 + K^2 \right) \Or(\sum_{m=1}^M K_m) + M \Or(\sum_{m=1}^M K_m^3)\,.
\end{equation}
Making use of $K_m \le K$, we have an upper bound for $\Cost_{ISMD}$
\begin{equation} \label{eqn:cost2}
	\Cost_{\text{ISMD}} \le  \Or(N^2 K/M) + \Or(M^2 K^3)\,.
\end{equation}
When $M = \Or((N/K)^{2/3})$, $\Cost_{\text{ISMD}} \le \Or(N^{4/3} K^{5/3})$. Comparing to the cost of partial eigen decomposition~\cite{sorensen1992implicit, lehoucq1996deflation}, which is about $\Or(N^2 K)$ \footnote{The cost can be reduced to $\Or(N^2\log(K))$ if a randomized SVD with some specific technique is applied.}, the ISMD is more efficient for low-rank matrices. 

For matrix $A$ which has a sparse decomposition, the local ranks $K_m$ are much smaller than its global rank $K$. An extreme case is $K_m = \Or(1)$, which is in fact true for many random fields, see~\cite{xiu2015, hou_LocalgPC_2016}. In this case, 
\begin{equation} \label{eqn:cost3}
	\Cost_{\text{ISMD}} =  \Or(N^2 / M) + \Or(M^2) + \Or(M K^2)\,.
\end{equation}
When the partition gets finer ($M$ increases), the computational cost first decreases due to the saving in local eigen decompositions. The computational cost achieves its minimum around $M=\Or(N^{2/3})$ and then increases due to the increasing cost for the joint diagonalization. This trend is observed in our numerical examples, see Figure~\ref{fig:2d_compare2}.

We point out that the $M$ local eigen decompositions~\eqref{eqn:localeig} and the joint diagonalization problems~\eqref{opt:jointDm} are solved independently on different patches. Therefore, our algorithm is embarrassingly parallelizable. This will save the computational cost in the first two steps by a factor of $M$, which makes the ISMD even faster.

\subsection{Connection with other matrix decomposition methods}\label{sec:othermethods}
Sparse decompositions of symmetric PSD matrices have been studied in different fields for a long time. There are in general two approaches to achieve sparsity: rotation or $L^1$ minimization.

The rotation approach begins with eigenvectors. Suppose that we have decided to retain and rotate $K$ eigenvectors. Define $H = [h_1, h_2, \dots, h_K]$ with $h_k$ being the $k$-th eigenvector. We post-multiply $H$ by a matrix $T \in \RR^{K\times K}$ to obtain the rotated modes $G = [g_1, g_2, \dots, g_K] = H T$. The choice of $T$ is determined by the rotation criterion we use. In data science, for the commonly-used varimax rotation criterion~\cite{krzanowski1995multivariate, jolliffe2002simplified}, $T$ is an orthogonal matrix chosen to maximize the variance of squared modes within each column of $G$. This drives entries in $G$ towards 0 or $\pm 1$. In quantum chemistry, every column in $H$ and $G$ corresponds to a function over a physical domain $D$ and certain specialized sparse modes -- localized modes -- are sought after. The most widely used criterion to achieve maximally localized modes is the one proposed in \cite{marzari1997maximally}. This criterion requires $T$ to be unitary, and then minimizes the second moment: 
\begin{equation} \label{eqn:wannier}
	\sum_{k=1}^K \int_D (x - x_k)^2 |g_k(x)|^2 \rd x\,,
\end{equation}
where $x_k = \int_D x |g_k(x)|^2 \rd x$. More recently, a method weighted by higher degree polynomials is discussed in \cite{weinan2010localized}. While these criteria work reasonably well for simple symmetric PSD functions/operators, they all suffer from non-convex optimization -- which requires a good starting point to converge to the global minimum. In addition, these methods only care about the eigenspace spanned by $H$ instead of the specific matrix decomposition, and thus they cannot be directly applied to solve our problem~\eqref{opt:minsparseness}.

The ISMD proposed in this paper follows the rotation approach. The ISMD implicitly finds a unitary matrix $T \in \RR^{K\times K}$ to construct the intrinsic sparse modes
\begin{equation}\label{eqn:rotationview}
	[g_1, g_2, \dots, g_K] = [\sqrt{\lambda_1} h_1, \sqrt{\lambda_2} h_2, \dots, \sqrt{\lambda_K} h_K ] ~T.
\end{equation}
Notice that we rotate the unnormalized eigenvector $\sqrt{\lambda_k} h_k$ to satisfy the decomposition constraint $A = \sum_{k=1}^K g_k g_k^T$. The criterion of the ISMD is to minimize the total patch-wise sparseness as in \eqref{opt:minsparseness}. The success of the ISMD lies in the fact that as long as the domain partition is regular-sparse, the optimization problem~\eqref{opt:minsparseness} can be exactly and efficiently solved by Algorithm~\ref{alg:ISMD}. Moreover, the intrinsic sparse modes produced by the ISMD are optimally localized because we are directly minimizing the total patch-wise sparseness of $\{g_k\}_{k=1}^K$. 

The $L^1$ minimization approach, pioneered by ScotLass~\cite{Jolliffe_2003}, has a rich literature in solving the sparse matrix factorization problem~\eqref{eqn:matrix_factorize}, see~
\cite{Zou_PCA_06, dAspremont_sparsePCA, zhang2012sparse, vu2013fantope, ozolins_compressed_2013, LaiLuOsher_2014}. Problem~\eqref{eqn:matrix_factorize} is highly non-convex in $(G, U)$, and there has been a lot of efforts (see e.g.~\cite{dAspremont_sparsePCA, vu2013fantope, LaiLuOsher_2014}) in relaxing it to a convex optimization. First of all, since there are no essential constraints on $U$, one can get rid of $U$ by considering the variational form~\cite{Jolliffe_2003, Zou_PCA_06, ozolins_compressed_2013}:
\begin{equation}\label{eqn:eigvariationalL1_matrix}
	\boxed{\min_{G \in \RR^{N\times K}}\quad -\mathrm{Tr}( G^T A G ) + \mu \|G\|_1  \quad \mathrm{s.t.} \quad G^T G = \mathbb{I}_{K},}
\end{equation}
where $A = X X^T$ is the covariance matrix as in the ISMD~\eqref{opt:minsparseness} and $\mathrm{Tr}$ is the trace operator on square matrices. Notice that the problem is still non-convex due to the orthogonality constraint $G^T G = \mathbb{I}_{K}$. In the second step, the authors in~\cite{vu2013fantope} proposed the following semi-definite programming to obtain the sparse density matrix $W \in \RR^{n\times n}$, which plays the same role as $G G^T$ in \eqref{eqn:eigvariationalL1_matrix}:
\begin{equation} \label{eqn:eigvariationalL1_relax}
\boxed{\min_{W \in \RR^{N\times N}}\quad -\mathrm{Tr}( A W ) + \mu \|W\|_1  \quad \mathrm{s.t.} \quad 0 \preceq W \preceq \mathbb{I}_{N},\, \mathrm{Tr}(W) = K.}
\end{equation}
Here, $0 \preceq W \preceq \mathbb{I}_{N}$ means that both $W$ and $\mathbb{I}_{N} - W$ are symmetric and positive semi-definite. Finally, the first $K$ eigenvectors of $W$ are used as the sparse modes $G$. An equivalent formulation was proposed in \cite{LaiLuOsher_2014}, and the authors proposed to pick $K$ columns of $W$ as the sparse modes $G$.

We will compare the advantages and disadvantages of the ISMD and the convex relaxation of sparse PCA in Section~\ref{subsec:compareL1ISMD} and Section~\ref{sec:exponential}.

\section{Theoretical results with regular-sparse partitions}\label{sec:exactdecom}
In this section, we present the main theoretical results of the ISMD, i.e., Theorem~\ref{thm:ISMD}, Theorem~\ref{thm:ISMDconsistency} and its perturbation analysis. We first introduce a domain-decomposition type presentation of any feasible decomposition $A = \sum_{k=1}^K \psi_k \psi_k^T$. Then we discuss the regular-sparse property and use it to prove our main results. When no ambiguity arises, we denote patch-wise sparseness $s(g_k; \CalP)$ as $s_k$.

\subsection{A domain-decomposition type presentation}\label{sec:sparseness}
For an arbitrary decomposition $A = \sum_{k=1}^K \psi_k \psi_k^T$, denote $\Psi \equiv [\psi_1, \dots, \psi_K]$ and $\Psi|_{_{P_m}} \equiv [\psi_1|_{_{P_m}}, \dots, \psi_K|_{_{P_m}}]$. For a sparse decomposition, we expect that most columns in $\Psi|_{_{P_m}}$ are zero, and thus we define the local dimension on patch $P_m$ as follows. 
\begin{definition}[Local dimension]\label{def:localdim}
The local dimension of a decomposition $A = \sum_{k=1}^K \psi_k \psi_k^T$ on patch $P_m$ is the number of nonzero modes when restricted to this patch, i.e.,
\begin{equation*}
	d(P_m; \Psi) = |S_m|, \qquad S_m = \{ k ~:~  \psi_k|_{_{P_m}} \neq 0 \}.
\end{equation*}
\end{definition}
When no ambiguity arises, $d(P_m; \Psi)$ is written as $d_m$. We enumerate all the elements in $S_m$ as $\{k_{i}^{m}\}_{i=1}^{d_m}$, and group together all the nonzero local pieces on patch $P_m$ and obtain
\begin{equation} \label{eqn:localmodes}
	\Psi_m \equiv [\psi_{m, 1}, \dots, \psi_{m, d_m}]\,, \qquad \psi_{k_{i}^{m}}|_{_{P_m}} = \psi_{m, i}\,.
\end{equation}
Therefore, we have
\begin{equation} \label{eqn:defLm}
\Psi|_{_{P_m}}= \Psi_m L_m^{(\psi)}\,,
\end{equation}
where $L_m^{(\psi)}$ is a matrix of size $d_m\times K$ with the $k_{i}^{m}$-th column being $\vct{e}_i$ for $i \in [d_m]$ and other columns being $\vct{0}$. Here, $\vct{e}_i$ is the $i$-th column of $\mathbb{I}_{d_m}$. $L_m^{(\psi)}$ is called the \textit{local indicator matrix} of $\Psi$ on patch $P_m$. Restricting the decomposition constraint $A = \Psi \Psi^T$ to patch $P_m$, we have $A_{mm} = \Psi|_{_{P_m}} \left(\Psi|_{_{P_m}}\right)^T$ where $A_{mm}$ is the restriction of $A$ on patch $P_m$, as in \eqref{def:Ablock}. Since $\Psi_m$ is obtained from $\Psi|_{_{P_m}}$ by deleting zero columns, we have
\begin{equation}\label{eqn:localConstraint}
	A_{mm} = \Psi_m \Psi_m^T.
\end{equation}

We stack up $\Psi_m$ and $L_m^{(\psi)}$ as follows,
\begin{equation*}
\Psi_{ext} \equiv \text{diag}\{\Psi_1, \Psi_2,\cdots, \Psi_M\}\,,\quad L^{(\psi)} \equiv \left[L_1^{(\psi)} ; L_2^{(\psi)} ; \cdots ; L_M^{(\psi)}\right]\,,
\end{equation*}
and then we have:
\begin{equation} \label{eqn:gdecom}
\Psi = [\Psi|_{_{P_1}}; \dots; \Psi|_{_{P_M}}] = \Psi_{ext} L^{(\psi)}\,.
\end{equation}
The intuition in Eqn.~\eqref{eqn:gdecom} is that the local pieces $\Psi_m$ are linked together by the indicator matrix $L^{(\psi)}$ and the modes $\Psi$ on the entire domain $[N]$ can be recovered from $\Psi_{ext}$ and $ L^{(\psi)}$. We call $L^{(\psi)}$ the \textit{indicator matrix} of $\Psi$.

We use a simple example to illustrate the patch-wise sparseness, the local dimension and Eqn.~\eqref{eqn:gdecom}. In this case, $\Psi \in \RR^{N\times K}$ ($N=100, K=2$) is the discretized version of two functions on $[0,1]$ and $\CalP$ partitions $[0,1]$ uniformly into four intervals as shown in Figure~\ref{fig:gdecom}. $\psi_1$, the red starred mode, is nonzero on the left two patches and $\psi_2$, the blue circled mode, is nonzero on the right three patches. The sparseness of $\psi_1$ is 2, the sparseness of $\psi_2$ is 3, and the local dimensions of the four patches are 1, 2, 1, and 1 respectively, as we comment in Figure~\ref{fig:gdecom}.
\begin{figure}\label{fig:gdecom}
\centering
\includegraphics[height = 0.2\textheight, width = 0.5\textwidth]{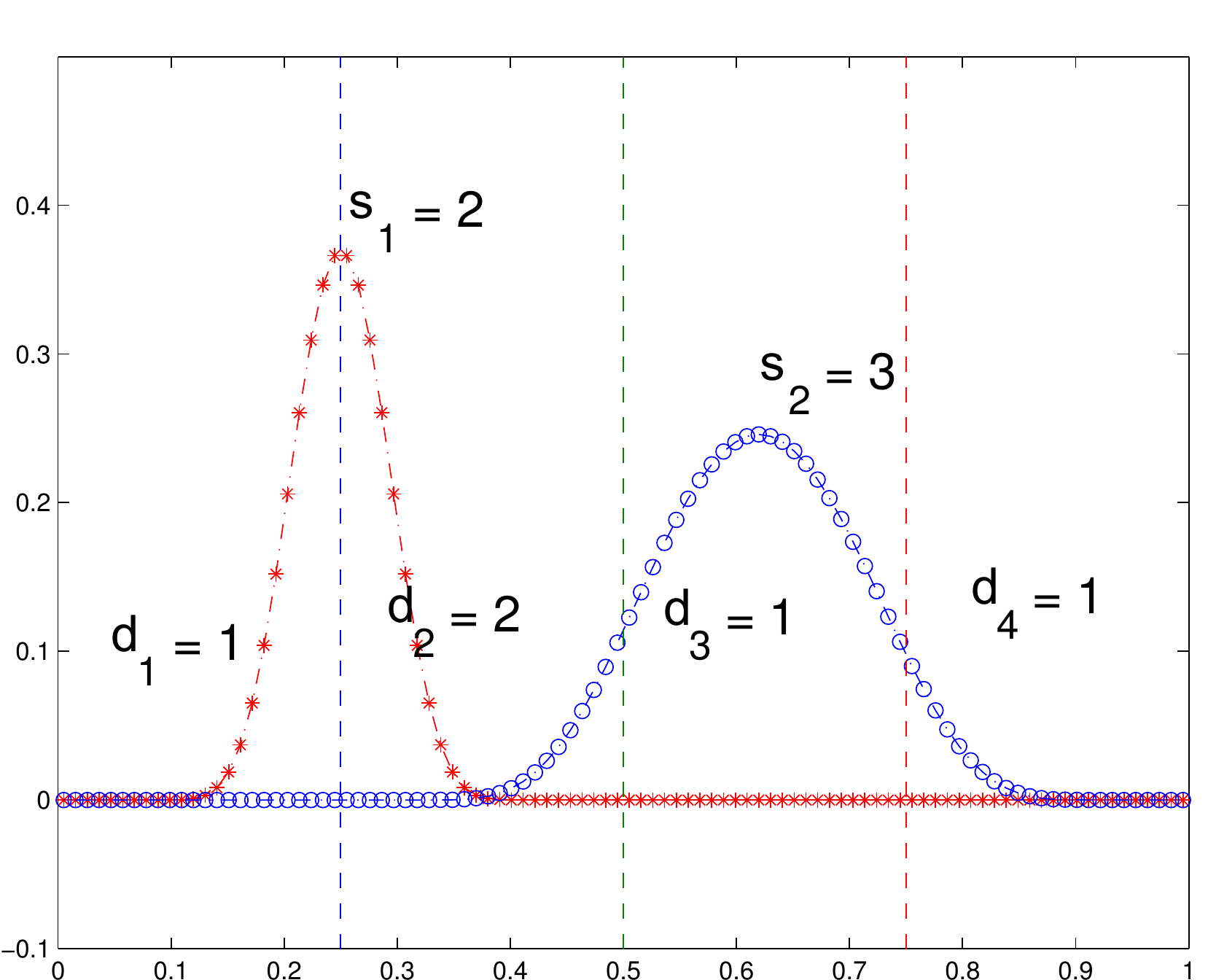}
\caption{Illustration of sparseness, local dimension and $\Psi = \Psi_{ext} L^{(\psi)}$.}
\end{figure}
Following the definitions above, we have $\Psi_1 = \psi_1|_{_{P_1}}$, $L_1^{(\psi)} = [1, 0]$, $\Psi_2 = [\psi_1|_{_{P_2}}, \psi_2|_{_{P_2}}]$, $L_2^{(\psi)} = [1, 0; 0, 1]$, $\Psi_3 = \psi_2|_{_{P_3}}$, $L_3^{(\psi)}  = [0, 1]$, $\Psi_4 = \psi_2|_{_{P_4}}$ and $L_4^{(\psi)} = [0, 1]$. Finally, we get
\begin{align*}
\left[\psi_1,\psi_2\right] = \Psi_{ext}L^{(\psi)} \equiv \left[\begin{array}{c:cc:c:c}
\psi_{1,1} & 0 & 0 & 0 & 0\\
\hdashline
0& \psi_{1,2} & \psi_{2,2} & 0 & 0\\
\hdashline
0& 0 & 0 & \psi_{2,3} & 0\\
\hdashline
0& 0 & 0 & 0 & \psi_{2,4} \end{array}\right] \left[\begin{array}{cc}
1 & 0\\
\hdashline
1 & 0\\
0 & 1\\ 
\hdashline
0 & 1\\ 
\hdashline
0 & 1\end{array}\right].
\end{align*}

With this domain-decomposition type representation of $\Psi$, the decomposition constraint is rewritten as:
\begin{equation}\label{eqn:constraint2}
	A = \Psi \Psi^T = \Psi_{ext} \Omega^{(\psi)} \Psi_{ext}^T\,, \qquad \Omega^{(\psi)} \equiv L^{(\psi)} \left( L^{(\psi)} \right)^T \,.
\end{equation}
Here, $\Omega^{(\psi)}$ has a role similar to that of $\Omega$ in the ISMD. It can be viewed as the correlation matrix of $A$ under basis $\Psi_{ext}$, just like how $\Lambda$ and $\Omega$ are defined.

Finally, we provide two useful properties of the local indicator matrices $L_m^{(\psi)}$, which are direct consequences of their definitions. Its proof is elementary and can be found in Appendix~\ref{appendix:proofs}.
\begin{proposition}\label{prop:localGL}
For an arbitrary decomposition $A = \Psi \Psi^T$, 
\begin{enumerate}
\item 
The $k$-th column of $L^{(\psi)}$, denoted as $l_k^{(\psi)}$, satisfies $\|l_k^{(\psi)}\|_1 = s_k$ where $s_k$ is the patch-wise sparseness of $\psi_k$, as in Definition~\ref{def:gsparseness}. Moreover, different columns in $L^{(\psi)}$ have disjoint supports.
\item Define 
\begin{equation}\label{def:B}
B_{n;m}^{(\psi)} \equiv \Omega^{(\psi)}_{mn} \left(\Omega^{(\psi)}_{mn}\right)^T \,,
\end{equation} 
where $\Omega^{(\psi)}_{mn} \equiv L_m^{(\psi)}(L^{(\psi)}_n)^T$ is the $(m,n)$-th block of $\Omega^{(\psi)}$. $B_{n;m}^{(\psi)}$ is diagonal with diagonal entries either 1 or 0. Moreover, $B^{(\psi)}_{n;m}(i,i) = 1$ if and only if there exists $k \in [K]$ such that $\psi_k|_{_{P_m}} = \psi_{m,i}$ and $\psi_k|_{_{P_n}} \neq \vct{0}$.
\end{enumerate}
\end{proposition}

Since different columns in $L^{(\psi)}$ have disjoint supports, $\Omega^{(\psi)} \equiv L^{(\psi)} \left( L^{(\psi)} \right)^T$ has a block-diagonal structure with $K$ blocks. The $k$-th diagonal block is the one contributed by $l_k^{(\psi)}\left(l_k^{(\psi)}\right)^T$. Therefore, as long as we obtain $\Omega^{(\psi)}$, we can use the pivoted Cholesky decomposition to efficiently recover $L^{(\psi)}$. The ISMD follows this rationale: we first construct local pieces $\Psi_{ext} \equiv \text{diag}\{\Psi_1, \Psi_2,\cdots, \Psi_M\}$ for certain set of intrinsic sparse modes $\Psi$; then from the decomposition constraint~\eqref{eqn:constraint2} we are able to compute $\Omega^{(\psi)}$; finally, the pivoted Cholesky decomposition is applied to obtain $L^{(\psi)}$ and the modes are assembled by $\Psi = \Psi_{ext} L^{(\psi)}$. Obviously, the key step is to construct $\Psi_{ext}$, which are local pieces of a set of intrinsic sparse modes -- this is exactly where the regular-sparse property and the joint diagonalization come into play.

\subsection{regular-sparse property and local modes construction}\label{sec:LocalSpace}
In this and the next subsections (Section~\ref{sec:LocalSpace} - Section~\ref{sec:consistency}), we assume that the submatrices $A_{mm}$ are well conditioned and thus the exact local eigen decomposition~\eqref{eqn:localeig} is used in the ISMD.

Combining the local eigen decomposition~\eqref{eqn:localeig} and local decomposition constraint~\eqref{eqn:localConstraint}, there exists $D_m^{(\psi)} \in \RR^{K_m\times d_m}$ such that
\begin{equation}\label{eqn:lineartrans}
	\Psi_m = H_mD_m^{(\psi)}.
\end{equation}
Moreover, since the local eigenvectors are linearly independent, we have
\begin{equation}\label{eqn:dmgeKm}
	d_m \ge K_m\,, \qquad D_m^{(\psi)}\left(D_m^{(\psi)}\right)^T = \mathbb{I}_{K_m}.
\end{equation}
We see that $d_m = K_m$ if and only if columns in $\Psi_m$ is also linearly independent. In this case, $D_m^{(\psi)}$ is unitary, i.e., $D_m^{(\psi)} \in \OO(K_m)$. This is exactly what is required by the regular-sparse property, see Definition~\ref{def:regular_sparse}. It is easy to see that we have the following equivalent definitions of regular-sparse property.
\begin{proposition}\label{prop:regularsparse1}
The following assertions are equivalent.
\begin{enumerate}
\item The partition $\CalP$ is regular-sparse with respect to $A$.
\item There exists a decomposition $A = \sum_{k=1}^K \psi_k \psi_k^T$ such that on every patch $P_m$ its local dimension $d_m$ is equal to the local rank $K_m$, i.e., $d_m = K_m$.
\item The minimum of problem~\eqref{opt:minsparseness} is $\sum_{m=1}^M K_m$.
\end{enumerate}
\end{proposition}

The proof is elementary and is omitted here. By Proposition~\ref{prop:regularsparse1}, for regular-sparse partitions local pieces of a set of intrinsic sparse modes can be constructed from rotating local eigenvectors, i.e., $\Psi_m = H_m D_m^{(\psi)}$. All the local rotations $\{D_m^{(\psi)}\}_{m=1}^M$ are coupled by the decomposition constraint $A = \Psi \Psi^T$. At first glance, it seems impossible to find such $D_m$ from this big coupled system. However, the following lemma gives a necessary condition that $D_m^{(\psi)}$ must satisfy so that $H_m D_m^{(\psi)}$ are local pieces of a set of intrinsic sparse modes. More importantly, this necessary condition turns out to be sufficient, and thus provides us a criterion to find the local rotations.

\begin{lemma}\label{lem:necessaryD}
Suppose that $\CalP$ is regular-sparse w.r.t. $A$ and that $\{\psi_k\}_{k=1}^K$ is an arbitrary set of intrinsic sparse modes. Denote the transformation from $H_m$ to $\Psi_m$ as $D_m^{(\psi)}$, i.e., $\Psi_m = H_m D_m^{(\psi)}$. Then $D_m^{(\psi)}$ is unitary and jointly diagonalizes $\{\Sigma_{n;m}\}_{n=1}^M$, which are defined in~\eqref{def:Sigma}. Specifically, we have
\begin{equation} \label{eqn:BDSigma}
	B_{n;m}^{(\psi)} = \left(D_m^{(\psi)}\right)^T \Sigma_{n;m} D_m^{(\psi)}, \quad m = 1, 2, \dots, M,
\end{equation}
where $B_{n;m}^{(\psi)} \equiv \Omega^{(\psi)}_{mn} \left(\Omega^{(\psi)}_{mn}\right)^T$, defined in~\eqref{def:B}, is diagonal with diagonal entries either 0 or 1.
\end{lemma}
\begin{proof}
From item 3 in Proposition~\ref{prop:regularsparse1}, any set of intrinsic sparse modes must have local dimension $d_m = K_m$ on patch $P_m$. Therefore, the transformation $D_m^{(\psi)}$ from $H_m$ to $\Psi_m$ must be unitary. Combining $\Psi_m = H_m D_m^{(\psi)}$ with the decomposition constraint~\eqref{eqn:constraint2}, we get
\begin{equation*}
A = H_{ext} D^{(\psi)} \Omega^{(\psi)} \left(D^{(\psi)}\right)^T H_{ext},
\end{equation*}
where $D^{(\psi)} = \text{diag}\{D_1^{(\psi)}, D_2^{(\psi)}, \dots, D_M^{(\psi)}\}$. Recall that $A = H_{ext} \Lambda H_{ext}$ and that $H_{ext}$ has linearly independent columns, we obtain
\begin{equation}\label{eqn:defLLT}
\Lambda = D^{(\psi)} \Omega^{(\psi)} \left(D^{(\psi)}\right)^T,
\end{equation}
or block-wisely,
\begin{equation}\label{eqn:OmegaPsi}
\Lambda_{mn} = D_m^{(\psi)} \Omega^{(\psi)}_{mn} \left(D_n^{(\psi)}\right)^T.
\end{equation}
Since $D_n^{(\psi)}$ is unitary, Eqn.~\eqref{eqn:BDSigma} naturally follows the definitions of $B_{n;m}^{(\psi)}$ and $\Sigma_{n;m}$. By item 2 in Proposition~\ref{prop:localGL}, we know that $B_{n;m}^{(\psi)}$ is diagonal with diagonal entries either 0 or 1.
\end{proof}

Lemma~\ref{lem:necessaryD} guarantees that $D_m^{(\psi)}$ for an arbitrary set of intrinsic sparse modes is the minimizer of the joint diagonalization problem~\eqref{opt:jointDm}. In the other direction, the following lemma guarantees that any minimizer of the joint diagonalization problem~\eqref{opt:jointDm}, denoted as $D_m$, transforms local eigenvectors $H_m$ to $G_m$, which are the local pieces of certain intrinsic sparse modes. 

\begin{lemma}\label{lem:sufficientD}
Suppose that $\CalP$ is regular-sparse w.r.t. $A$ and that $D_m$ is a minimizer of the joint diagonalization problem~\eqref{opt:jointDm}. As in the ISMD, define $G_m = H_m D_m$. Then there exists a set of intrinsic sparse modes such that its local pieces on patch $P_m$ are equal to $G_m$.
\end{lemma}

Before we prove this lemma, we examine the uniqueness property of intrinsic sparse modes. It is easy to see that permutations and sign flips of a set of intrinsic sparse modes are still a set of intrinsic sparse modes. Specifically, if $\{\psi_k\}_{k=1}^K$ is a set of intrinsic sparse modes and $\sigma : [K] \to [K]$ is a permutation, $\{\pm \psi_{\sigma(k)}\}_{k=1}^K$ is another set of intrinsic sparse modes. Another kind of non-uniqueness comes from the following concept -- identifiability. 

\begin{definition}[Identifiability]\label{def:identifiable}
 For two modes $g_1, g_2 \in \RR^N$, they are unidentifiable on partition $\CalP$ if they are supported on the same patches, i.e., $\{P \in \CalP : g_1|_{_{P}} \neq \vct{0}\} = \{P \in \CalP : g_2|_{_{P}} \neq \vct{0}\}$. Otherwise, they are identifiable. For a collection of modes $\{g_i\}_{i=1}^k \subset \RR^N$, they are unidentifiable iff any pair of them are unidentifiable. They are pair-wisely identifiable iff any pair of them are identifiable.
\end{definition}

It is important to point out that the identifiability above is based on the resolution of partition $\CalP$. Unidentifiable modes for partition $\CalP$ may have different supports and become identifiable on a refined partition. Unidentifiable intrinsic sparse modes lead to another kind of non-uniqueness for intrinsic sparse modes. For instance, when two intrinsic sparse modes $\psi_m$ and $\psi_n$ are unidentifiable, then any rotation of $[\psi_m, \psi_n]$ while keeping other intrinsic sparse modes unchanged is still a set of intrinsic sparse modes. 

Local pieces of intrinsic sparse modes inherit this kind of non-uniqueness. Suppose $\Psi_m\equiv [\psi_{m, 1}, \dots, \psi_{m, d_m}]$ are the local pieces of a set of intrinsic sparse modes $\Psi$ on patch $P_m$. First, if $\sigma : [d_m] \to [d_m]$ is a permutation, $\{\pm \psi_{m,\sigma(i)}\}_{i=1}^{d_m}$ are local pieces of another set of intrinsic sparse modes. Second, if $\psi_{m, i}$ and $\psi_{m, j}$ are the local pieces of two unidentifiable intrinsic sparse modes, then any rotation of $[\psi_{m, i}, \psi_{m, j}]$ while keeping other local pieces unchanged are local pieces of another set of intrinsic sparse modes. It turns out that this kind of non-uniqueness has a one-to-one correspondence with the non-uniqueness of joint diagonalizers for problem~\eqref{opt:jointDm}, which is characterized in Theorem~\ref{thm:jointDunique}. Keeping this correspondence in mind, the proof of Lemma~\ref{lem:sufficientD} is quite intuitive. 

\begin{proof}\textbf{[Proof of Lemma~\ref{lem:sufficientD}]}
Let $\Psi \equiv [\psi_1, \dots, \psi_K]$ be an arbitrary set of intrinsic sparse modes. We order columns in $\Psi$ such that unidentifiable modes are grouped together, denoted as $\Psi = [\Psi_{1}, \dots, \Psi_{Q}]$, where $Q$ is the number of unidentifiable groups. Accordingly on patch $P_m$, $\Psi_m = [\Psi_{m,1}, \dots, \Psi_{m,Q_m}]$ where $Q_m$ is the number of nonzero unidentifiable groups. Denote the number of columns in each group as $n_{m,i}$, i.e., there are $n_{m,i}$ modes in $\{\psi_k\}_{k=1}^K$ that are nonzero and unidentifiable on patch $P_m$.

Making use of item 2 in Proposition~\ref{prop:localGL}, one can check that $\psi_{m,i}$ and $\psi_{m,j}$ are unidentifiable if and only if $B_{n;m}^{(\psi)}(i,i) = B_{n;m}^{(\psi)}(j,j)$ for all $n \in [M]$. Since unidentifiable pieces in $\Psi_m$ are grouped together, the same diagonal entries in $\{B_{n;m}^{(\psi)}\}_{n=1}^M$ are grouped together as required in Theorem~\ref{thm:jointDunique}. Now we apply Theorem~\ref{thm:jointDunique} with $M_k$ replaced by $\Sigma_{n;m}$, $\Lambda_k$ replaced by $B_{n;m}^{(\psi)}$, $D$ replaced by $D_m^{(\psi)}$, the number of distinct eigenvalues $m$ replaced by $Q_m$, eigenvalue's multiplicity $q_i$ replaced by $n_{m,i}$ and the diagonalizer $V$ replaced by $D_m$. Therefore, there exists a permutation matrix $\Pi_m$ and a block diagonal matrix $V_m$ such that
\begin{equation}\label{eqn:sufficientD1}
D_m \Pi_m = D_m^{(\psi)} V_m \,, \qquad V_m = \text{diag}\{V_{m,1}, \dots, V_{m,Q_m}\}\,.
\end{equation}
Recall that $G_m = H_m D_m$ and $\Psi_m = H_m D_m^{(\psi)}$, we obtain that
\begin{equation}\label{eqn:sufficientD2}
	G_m \Pi_m = \Psi_m V_m = [\Psi_{m,1} V_{m,1}\,, \dots\,, \Psi_{m,Q_m} V_{m,Q_m}]\,.
\end{equation}
From Eqn.~\eqref{eqn:sufficientD2}, we can see that identifiable pieces are completely separated and the small rotation matrices, $V_{m,i}$, only mix unidentifiable pieces $\Psi_{m,i}$. $\Pi_m$ merely permutes the columns in $G_m$. From the non-uniqueness of local pieces of intrinsic sparse modes, we conclude that $G_m$ are local pieces of another set of intrinsic sparse modes.
\end{proof}

We point out that the local pieces $\{G_m\}_{m=1}^M$ constructed by the ISMD on different patches may correspond to different sets of intrinsic sparse modes. Therefore, the final ``patch-up" step should further modify and connect them to build a set of intrinsic sparse modes. Fortunately, the pivoted Cholesky decomposition elegantly solves this problem.

\subsection{Optimal sparse recovery and consistency of the ISMD}\label{sec:consistency}
As defined in the ISMD, $\Omega$ is the correlation matrix of $A$ with basis $G_{ext}$, see~\eqref{def:omega}. If $\Omega$ enjoys a block diagonal structure with each block corresponding to a single intrinsic sparse mode, just like $\Omega^{(\psi)} \equiv L^{(\psi)} \left( L^{(\psi)} \right)^T$, the pivoted Cholesky decomposition can be utilized to recover the intrinsic sparse modes. 

It is fairly easy to see that $\Omega$ indeed enjoys such a block diagonal structure when there is one set of intrinsic sparse modes that are pair-wisely identifiable. Denoting this identifiable set as $\{\psi_k\}_{k=1}^K$ (only its existence is needed), by Eqn.~\eqref{eqn:sufficientD1}, we know that on patch $P_m$ there is a permutation matrix $\Pi_m$ and a diagonal matrix $V_m$ with diagonal entries either 1 or -1 such that $D_m \Pi_m = D_m^{(\psi)} V_m$. Recall that $\Lambda = D \Omega D^T = D^{(\psi)} \Omega^{(\psi)} \left(D^{(\psi)}\right)^T$, see \eqref{eqn:Omega} and \eqref{eqn:OmegaPsi}, we have 
\begin{equation}\label{eqn:OmegaRel}
\Omega = D^T D^{(\psi)} \Omega^{(\psi)} \left(D^{(\psi)}\right)^TD = \Pi V^T \Omega^{(\psi)} V \Pi^T,
\end{equation}
in which $V = \diag\{V_1, \dots, V_m\}$ is diagonal with diagonal entries either 1 or -1 and $\Pi = \diag\{\Pi_1, \dots, \Pi_m\}$ is a permutation matrix. Since the action of $\Pi V^T$ does not change the block diagonal structure of $\Omega^{(\psi)}$, $\Omega$ still has such a structure and the pivoted Cholesky decomposition can be readily applied. In fact, the action of $\Pi V^T$ exactly corresponds to the column permutation and sign flips of intrinsic sparse modes, which is the only kind of non-uniqueness of problem~\eqref{opt:minsparseness} when the intrinsic sparse modes are pair-wisely identifiable. For the general case when there are unidentifiable intrinsic sparse modes, $\Omega$ still has the block diagonal structure with each block corresponding to a group of unidentifiable modes, resulting in the following theorem.
\begin{theorem}\label{thm:ISMD}
Suppose the domain partition $\CalP$ is regular-sparse with respect to $A$. Let $A = G G^T$ be the decomposition given by the ISMD~\eqref{eqn:ISMD} and $\Psi \equiv [\psi_1, \dots, \psi_K]$ be an arbitrary set of intrinsic sparse modes. Let columns in $\Psi$ be ordered such that unidentifiable modes are grouped together, denoted as $\Psi = [\Psi_{1}, \dots, \Psi_{Q}]$, where $Q$ is the number of unidentifiable groups and $n_q$ is the number of modes in $\Psi_q$. Then there exists $Q$ rotation matrices $U_q \in \RR^{n_q\times n_q}$ ($1\le q \le Q$) such that
\begin{equation}\label{eqn:uniqueness}
	G = [\Psi_{1} U_1, \dots, \Psi_{Q} U_Q],
\end{equation} 
with reordering of columns in $G$ if necessary. It immediately follows that
\begin{itemize}
\item the ISMD generates one set of intrinsic sparse modes. 
\item the intrinsic sparse modes are unique up to permutations and rotations within unidentifiable modes.
\end{itemize}
\end{theorem}

\begin{proof} By Eqn.~\eqref{eqn:sufficientD1}, Eqn.~\eqref{eqn:OmegaRel} still holds true with block diagonal $V_m$ for $m\in[M]$. Without loss of generality, we assume that $\Pi = I$ since permutation does not change the block diagonal structure that we desire. Then from Eqn.~\eqref{eqn:OmegaRel} we have
\begin{equation}\label{eqn:OmegaRel2}
\Omega = V^T \Omega^{(\psi)} V = V^T L^{(\psi)} \left( L^{(\psi)} \right)^T V.
\end{equation}
In terms of block-wise formulation, we get
\begin{equation}\label{eqn:OmegaRel3}
\Omega_{mn} = V_m^T \Omega_{mn}^{(\psi)} V_n = V_m^T L_m^{(\psi)} \left( L_n^{(\psi)} \right)^T V_n.
\end{equation}
Correspondingly, by \eqref{eqn:sufficientD2} the local pieces satisfy
\begin{equation*}
	G_m = [G_{m,1}\,, \dots\,, G_{m, Q_m} ]=  [\Psi_{m,1} V_{m,1}\,, \dots\,, \Psi_{m,Q_m} V_{m,Q_m}]\,.
\end{equation*}

Now, we prove that $\Omega$ has the block diagonal structure in which each block corresponds to a group of unidentifiable modes. Specifically, $G_{m,i}= \Psi_{m,i} V_{m,i}$ and $G_{n,j}= \Psi_{n,j} V_{n,j}$ are two identifiable groups, i.e., $\Psi_{m,i}$ and $\Psi_{n,j}$ are from two identifiable groups, and we want to prove that the corresponding block in $\Omega$, denoted as $\Omega_{m,i; n,j}$, is zero. From Eqn.~\eqref{eqn:OmegaRel3}, one gets $\Omega_{m,i; n,j} = V_{m,i}^T L_{m,i}^{(\psi)} \left(L_{n,j}^{(\psi)}\right)^T V_{n,j}$, where $L_{m,i}^{(\psi)}$ are the rows in $L_m^{(\psi)}$ corresponding to $\Psi_{m,i}$. $L_{n,j}^{(\psi)}$ is defined similarly. Due to identifiability between $\Psi_{m,i}$ and $\Psi_{n,j}$, we know $L_{m,i}^{(\psi)} \left(L_{n,j}^{(\psi)}\right)^T = 0$ and thus we obtain the block diagonal structure of $\Omega$.

In \eqref{eqn:OmegaPivotChol}, the ISMD performs the pivoted Cholesky decomposition $\Omega = P L L^T P^T$ and generates sparse modes $G = G_{ext} P L$. Due to the block diagonal structure in $\Omega$, every column in $P L$ can only have nonzero entries on local pieces that are not identifiable. Therefore, columns in $G$ have identifiable intrinsic sparse modes completely separated and unidentifiable intrinsic sparse modes rotated (including sign flip) by certain unitary matrices. Therefore, $G$ is a set of intrinsic sparse modes.

Due to the arbitrary choice of $\Psi$, we know that the intrinsic sparse modes are unique to permutations and rotations within unidentifiable modes.
\end{proof}

\begin{remark}\label{rem:pivotCholRem}
From the proof above, we can see that it is the block diagonal structure of $\Omega$ that leads to the recovery of intrinsic sparse modes. The pivoted Cholesky decomposition is one way to explore this structure. In fact, the pivoted Cholesky decomposition can be replaced by any other matrix decomposition that preserves this block diagonal structure, for instance, the eigen decomposition if there is no degeneracy. 
\end{remark}

Despite the fact that the intrinsic sparse modes depend on the partition $\CalP$, the following theorem guarantees that the solutions to problem \eqref{opt:minsparseness} give consistent results as long as the partition is regular-sparse. 
\begin{theorem}\label{thm:ISMDconsistency}
Suppose that $\CalP_c$ is a partition,  $\CalP_f$ is a refinement of $\CalP_c$ and that $\CalP_f$ is regular-sparse. Suppose $\{g^{(c)}_{k}\}_{k=1}^K$ and $\{g^{(f)}_{k}\}_{k=1}^K$ (with reordering if necessary) are the intrinsic sparse modes produced by the ISMD on $\CalP_c$ and $\CalP_f$, respectively. Then for every $k \in \{1,2, \dots, K\}$, 
in the coarse partition $\CalP_c$ $g^{(c)}_{k}$ and $g^{(f)}_{k}$ are supported on the same patches, while in the fine partition $\CalP_f$ the support patches of $g^{(f)}_{k}$ are contained in the support patches of $g^{(c)}_{k}$, i.e.,
\begin{equation*}
\begin{split}
	\{P \in \CalP_c : g^{(f)}_{k}|_{_{P}} \neq \vct{0}\} &= \{P \in \CalP_c : g^{(c)}_{k}|_{_{P}} \neq \vct{0}\}, \\
	\{P \in \CalP_f : g^{(f)}_{k}|_{_{P}} \neq \vct{0}\} &\subset \{P \in \CalP_f : g^{(c)}_{k}|_{_{P}} \neq \vct{0}\}.
\end{split}
\end{equation*}
Moreover, if $g^{(c)}_{k}$ is identifiable on the coarse patch $\CalP_c$, it remains unchanged when the ISMD is performed on the refined partition $\CalP_f$, i.e., $g^{(f)}_{k} = \pm g^{(c)}_{k}$.
\end{theorem}
\begin{proof}
Given the finer partition $\CalP_f$ is regular-sparse, it is easy to prove the coarser partition $\CalP_c$ is also regular-sparse.\footnote{We provide the proof in supplementary materials, see Lemma~\ref{lem:refinepartition}.} Notice that if two modes are identifiable on the coarse partition $\CalP_c$, they must be identifiable on the fine partition $\CalP_f$. However, the other direction is not true, i.e., unidentifiable modes may become identifiable if the partition is refined. Based on this observation, Theorem~\ref{thm:ISMDconsistency} is a simple corollary of Theorem~\ref{thm:ISMD}. 
\end{proof}

Finally, we provide a necessary condition for a partition to be regular-sparse as follows.
\begin{proposition}\label{prop:necessaryRS}
If $\CalP$ is regular-sparse w.r.t. $A$, all eigenvalues of $\Lambda$ are integers. Here, $\Lambda$ is computed in the ISMD by Eqn.~\eqref{def:Lambda}.
\end{proposition}
\begin{proof}
Let $\{\psi_k\}_{k=1}^K$ be a set of intrinsic sparse modes. Since $\CalP$ is regular-sparse, $D^{(\psi)}$ in Eqn.~\eqref{eqn:defLLT} is unitary. Therefore, $\Lambda$ and $\Omega^{(\psi)} \equiv L^{(\psi)} \left( L^{(\psi)} \right)^T$ share the same eigenvalues. Due to the block-diagonal structure of $\Omega^{(\psi)}$, one can see that 
\begin{equation*}
\Omega^{(\psi)} \equiv L^{(\psi)} \left( L^{(\psi)} \right)^T = \sum_{k=1}^K l_k^{(\psi)} \left( l_k^{(\psi)} \right)^T
\end{equation*}
is in fact the eigen decomposition of $\Omega^{(\psi)}$. The eigenvalue corresponding to the eigenvector $l_k^{(\psi)}$ is $\|l_k^{(\psi)}\|_2^2$, which is also equal to $\|l_k^{(\psi)}\|_1$ because $L^{(\psi)}$ only elements 0 or 1. From item 1 in Proposition~\ref{prop:localGL}, $\|l_k^{(\psi)}\|_1 = s_k$, which is the patch-wise sparseness of $\psi_k$.
\end{proof}

Combining Theorem~\ref{thm:ISMD}, Theorem~\ref{thm:ISMDconsistency} and Proposition~\ref{prop:necessaryRS}, we can develop a hierarchical process that gradually finds the finest regular-sparse partition and thus obtains the sparsest decomposition using the ISMD. This sparsest decomposition can be viewed as another definition of intrinsic sparse modes, which are independent of partitions. In our numerical examples, our partitions are all uniform but with different patch sizes. We see that even when the partition is not regular-sparse, the ISMD still produces a nearly optimal sparse decomposition.

\section{Perturbation analysis and two modifications}\label{sec:perturbmodify}
In real applications, data are often contaminated by noises. For example, when measuring the covariance function of a random field, sample noise is inevitable if a Monte Carlo type sampling method is utilized. A basic requirement for a numerical algorithm is its stability with respect to small noises. In Section~\ref{sec:perturbation}, under several assumptions, we are able to prove that the ISMD is stable with respect to small perturbations in the input $A$. In Section~\ref{sec:stability}, we provide two modified ISMD algorithms that effectively handle noises in different situations.  

\subsection{Perturbation analysis of the ISMD}\label{sec:perturbation}
We consider the additive perturbation here, i.e., $\hat{A}$ is an approximately low rank symmetric PSD matrix that satisfies
\begin{equation}\label{eqn:noisyCov}
	\hat{A} = A + \epsilon \tilde{A}, \qquad \|\tilde{A}\|_2 \le 1.
\end{equation}
Here, $A$ is the noiseless rank-$K$ symmetric PSD matrix and $\tilde{A}$ is the symmetric additive perturbation and $\epsilon > 0$ quantifies the noise level. We divide $\tilde{A}$ into blocks that are conformal with blocks of $A$ in \eqref{def:Ablock} and thus $\hat{A}_{mn} = A_{mn} + \epsilon \tilde{A}_{mn}$. In this case, we need to apply the truncated local eigen decomposition~\eqref{eqn:localeigtruncate} to capture the correct local rank $K_m$. Suppose the eigen decomposition of $\hat{A}_{mm}$ is 
\begin{equation*}
	\hat{A}_{mm} = \sum_{i=1}^{K_m} \hat{\gamma}_{m,i} \hat{h}_{n,i} \hat{h}_{n,i}^T + \sum_{i > K_m} \hat{\gamma}_{m,i} \hat{h}_{n,i} \hat{h}_{n,i}^T.
\end{equation*}
In this subsection, we assume that the noise level is very small with  $\epsilon \ll 1$ such that there is an energy gap between $\hat{\gamma}_{m,K_m}$ and $\hat{\gamma}_{m,K_m+1}$. Therefore, the truncation~\eqref{eqn:localeigtruncate} captures the correct local rank $K_m$, i.e.,
\begin{equation}\label{eqn:truncatedAmm}
	\hat{A}_{mm} \approx \hat{A}_{mm}^{(t)} \equiv \sum_{i=1}^{K_m} \hat{\gamma}_{m,i} \hat{h}_{n,i} \hat{h}_{n,i}^T \equiv \hat{H}_m \hat{H}_m^T.
\end{equation}
In the rest of the ISMD, the perturbed local eigenvectors $\hat{H}_m$ is used as $H_m$ in the noiseless case. We expect that our ISMD is stable with respect to this small perturbation and generates slightly perturbed intrinsic sparse modes of $A$.

To carry out this perturbation analysis, we will restrict ourselves to the case when intrinsic sparse modes of $A$ are pair-wisely identifiable and thus it is possible to compare the error between the noisy output $\hat{g}_k$ with $A$'s intrinsic sparse mode $g_k$. When there are unidentifiable intrinsic sparse modes of $A$, it only makes sense to consider the perturbation of the subspace spanned by those unidentifiable modes and we will not consider this case in this paper. The following lemma is a preliminary result on the perturbation analysis of local pieces $G_m$.
\begin{lemma} \label{lem:Gmperturb}
Suppose that partition $\CalP$ is regular-sparse with respect to $A$ and all intrinsic modes are identifiable with each other. Furthermore, we assume that for all $m \in [M]$ there exists $E_m^{(eig)}$ such that 
\begin{equation}\label{eqn:Annperturb}
	\hat{A}_{mm}^{(t)} = (I + \epsilon E_m^{(eig)}) A_{mm} \left(I + \epsilon (E_m^{(eig)})^T \right) \quad \text{and} \quad \|E_m^{(eig)}\|_2 \le C_{eig}.
\end{equation}
Here $C_{eig}$ is a constant depending on $A$ but not on $\epsilon$ or $\tilde{A}$. Then there exists $E_m^{(jd)} \in \RR^{K_m \times K_m}$ such that
\begin{equation}\label{eqn:Gmperturb}
	\hat{G}_m = (I + \epsilon E_m^{(eig)}) G_m (I + \epsilon E_m^{(jd)} + \Or(\epsilon^2)) J_m \quad \text{and} \quad \|E_m^{(jd)}\|_F \le C_{jd},
\end{equation} 
where $G_m$ and $\hat{G}_m$ are local pieces constructed by the ISMD with input $A$ and $\hat{A}$ respectively, $J_m$ is the product of a permutation matrix with a diagonal matrix having only $\pm 1$ on its diagonal, and $C_{jd}$ is a constant depending on $A$ but not on $\epsilon$ or $\tilde{A}$. Here, $\|\bullet\|_2$ and $\|\bullet\|_F$ are matrix spectral norm and Frobenius norm, respectively.
\end{lemma}

Lemma~\ref{lem:Gmperturb} ensures that local pieces of intrinsic sparse modes can be constructed with $\Or(\epsilon)$ accuracy up to permutation and sign flips (characterized by $J_m$ in~\eqref{eqn:Gmperturb}) under several assumptions. The identifiability assumption is necessary. Without such assumption, these local pieces are not uniquely determined up to permutations and sign flips. The assumption~\eqref{eqn:Annperturb} holds true when eigen decomposition of $A_{mm}$ is well conditioned, i.e., both eigenvalues and eigenvectors are well conditioned. We expect that a stronger perturbation result is still true without making this assumption. The proof of Lemma~\ref{lem:Gmperturb} is an application of perturbation analysis for the joint diagonalization problem~\cite{cardoso_perturbation_1994}, and is presented in Appendix~\ref{appendix:lemproof}.

Finally, $\hat{\Omega}$ is the correlation matrix of $\hat{A}$ with basis $\hat{G}_{ext} = \text{diag}\{\hat{G}_1, \hat{G}_2, \dots, \hat{G}_M\}$. Specifically, the $(m,n)$-th block of $\hat{\Omega}$ is given by
\begin{equation*}
	\hat{\Omega}_{mn} = \hat{G}_m^{\dagger} \hat{A}_{mn} \left( \hat{G}_n^{\dagger} \right)^T.
\end{equation*}
Without loss of generality, we can assume that $J_m = \mathbb{I}_{K_m}$ in~\eqref{eqn:Gmperturb}.\footnote{One can check that $\{J_m\}_{m=1}^M$ only affect the sign of recovered intrinsic sparse modes $[\hat{g}_1, \hat{g}_2, \dots, \hat{g}_K]$ if pivoted Cholesky decomposition is applied on $\hat{\Omega}$.} Based on the perturbation analysis of $G_m$ in Lemma~\ref{lem:Gmperturb} and the standard perturbation analysis of pseudo-inverse, for instance see Theorem 3.4 in~\cite{stewart1977perturbation}, it is straightforward to get a bound of the perturbations in $\hat{\Omega}$, i.e.,
\begin{equation}\label{eqn:omegaperturb}
	\|\hat{\Omega} - \Omega \|_2 \le C_{ismd} \epsilon.
\end{equation}
Here, $C_{ismd}$ depends on the smallest singular value of $G_m$ and the constants $C_{eig}$ and $C_{jd}$ in Lemma~\ref{lem:Gmperturb}. Notice that when all intrinsic modes are identifiable with each other, the entries of $\Omega$ are either 0 or $\pm1$. Therefore, when $C_{ismd} \epsilon$ is small enough, we can exactly recover $\Omega$ from $\hat{\Omega}$ as below:
\begin{equation}\label{eqn:threshold1}
	\Omega_{ij} = \begin{cases}
-1, & \text{for } \hat{\Omega}_{ij} < -0.5, \\
0, & \text{for } \hat{\Omega}_{ij} \in [-0.5, 0.5],\\
1, & \text{for } \hat{\Omega}_{ij} > 0.5.
\end{cases}
\end{equation}
Following Algorithm~\ref{alg:ISMD}, we get the pivoted Cholesky decomposition $\Omega = P L L^T P^T$ and output the perturbed intrinsic sparse modes 
\begin{equation*}
	\hat{G} = \hat{G}_{ext} P L.
\end{equation*}
Notice that the patch-wise sparseness information is all coded in $L$ and we can reconstruct $L$ exactly due to the thresholding step~\eqref{eqn:threshold1}, $\hat{G}$ has the same patch-wise sparse structure as $G$. Moreover, because the local pieces $\hat{G}_{ext}$ are constructed with $\Or(\epsilon)$ error, we have
\begin{equation}\label{eqn:Gperturb}
	\|\hat{G} - G \|_2 \le C_{g} \epsilon,
\end{equation}
where the constant $C_{g}$ only depends the constants $C_{eig}$ and $C_{jd}$ in Lemma~\ref{lem:Gmperturb}.

\subsection{Two modified ISMD algorithms}\label{sec:stability}
In Section~\ref{sec:perturbation}, we have shown that the ISMD is robust to small noises under the assumption of regular sparsity and identifiability. In this section, we provide two modified versions of the ISMD to deal with the cases when these two assumptions fail. The first modification aims at constructing  intrinsic sparse modes from noisy input $\hat{A}$ in the small noise region, as in Section~\eqref{sec:perturbation}, but it does not require the regular sparsity and identifiability. The second modification aims at constructing a simultaneous low-rank and sparse approximation of $\hat{A}$ when the noise is big. Our numerical experiments demonstrate that these modified algorithms are quite effective in practice.

\subsubsection{ISMD with thresholding}
In the general case when unidentifiable pairs of intrinsic sparse modes exist, the thresholding idea~\eqref{eqn:threshold1} is still applicable but the threshold $\epsilon_{th}$ should be learnt from the data, i.e., the entries in $\hat{\Omega}$. Specifically, there are $\Or(1)$ entries in $\hat{\Omega}$ corresponding to the slightly perturbed nonzero entries in $\Omega$; there are also many $\Or(\epsilon)$ entries that are contributed by the noise $\epsilon \tilde{A}$. If the noise level $\epsilon$ is small enough, we can see a gap between these two group of entries, and a threshold $\epsilon_{th}$ is chosen such that it separates these two groups. A simple 2-cluster algorithm is able to identify the threshold $\epsilon_{th}$. In our numerical examples we draw the histogram of absolute values of entries in $\hat{\Omega}$ and it clearly shows the 2-cluster effect, see Figure~\ref{fig:hatOmega}. Finally, we set all the entries in $\hat{\Omega}$ with absolute value less than $\epsilon_{th}$ to 0. In this approach we do not need to know the noise level $\epsilon$ a priori and we just learn the threshold from the data. To modify Algorithm~\ref{alg:ISMD} with this thresholding technique, we just need to add one line between assembling $\Omega$ (Line~\ref{line:assembleOmega}) and the pivoted Cholesky decomposition (Line~\ref{line:pivotChol}), see Algorithm~\ref{alg:ISMDm1}.

\begin{algorithm}
\caption{Intrinsic sparse mode decomposition with thresholding}\label{alg:ISMDm1}
\begin{algorithmic}[1]
\Require $A \in \RR^{N\times N}$: symmetric and PSD; $\CalP=\{P_m\}_{m=1}^M$: partition of index set $[N]$
\Ensure $G = [g_1, g_2, \cdots, g_K]$: $A \approx G G^T$
\State The same with Algorithm~\ref{alg:ISMD} from Line~\ref{line:start} to Line~\ref{line:end}
\LineComment{Assemble $\Omega$, thresholding and its pivoted Cholesky decomposition}
\State $\Omega = D^T \Lambda D$
\State Learn a threshold $\epsilon_{th}$ from $\Omega$ and set all the entries in $\Omega$ with absolute value less than $\epsilon_{th}$ to 0
\State $\Omega = P L L^T P^T$
\LineComment{Assemble the intrinsic sparse modes $G$}
\State $G = H_{ext} D P  L$
\end{algorithmic}
\end{algorithm}

It is important to point out that when the noise is large, the $\Or(1)$ entries and $\Or(\epsilon)$ entries mix together. In this case, we cannot identify such a threshold $\epsilon_{th}$ to separate them, and the assumption that there is an energy gap between $\hat{\gamma}_{m,K_m}$ and $\hat{\gamma}_{m,K_m+1}$ is invalid. In the next subsection, we will present the second modified version to overcome this difficulty. 

\subsubsection{Low rank approximation with ISMD}
In the case when there is no gap between $\hat{\gamma}_{m,K_m}$ and $\hat{\gamma}_{m,K_m+1}$ (i.e., no well-defined local ranks), or when the noise is so large that the threshold $\epsilon_{th}$ cannot be identified, we modify our ISMD to give a low-rank approximation of $A \approx G G^T$, in which $G$ is observed to be patch-wise sparse from our numerical examples.

In this modification, the normalization~\eqref{eqn:renormalize} is applied and thus we have:
\begin{equation*}
	A \approx \bar{G}_{ext} \bar{\Omega} \bar{G}_{ext}^T.
\end{equation*}
It is important to point out that $\bar{\Omega}$ has the same block diagonal structure as $\Omega$ but has different eigenvalues. Specifically, for the case when there is no noise and the regular-sparse assumption holds true, $\bar{\Omega}$ has eigenvalues $\{\|g_k\|_2^2\}_{k=1}^K$ for a certain set of intrinsic sparse modes $g_k$, while $\Omega$ has eigenvalues $\{s_k\}_{k=1}^K$ (here $s_k$ is the patch-wise sparseness of the intrinsic sparse mode). We first perform eigen decomposition $\bar{\Omega} = \bar{L}\bar{L}^T$ and then assemble the final result by $G = \bar{G}_{ext} \bar{L}$. The modified algorithm is summarized in Algorithm~\ref{alg:ISMDm2}.

\begin{algorithm}
\caption{Intrinsic sparse mode decomposition for low rank approximation}\label{alg:ISMDm2}
\begin{algorithmic}[1]
\Require $A \in \RR^{N\times N}$: symmetric and PSD; $\CalP=\{P_m\}_{m=1}^M$: partition of index set $[N]$
\Ensure $G = [g_1, g_2, \cdots, g_K]$: $A \approx G G^T$
\State The same with Algorithm~\ref{alg:ISMD} from Line~\ref{line:start} to Line~\ref{line:end}
\LineComment{Assemble $\Omega$, normalization and its eigen decomposition}
\State $\Omega = D^T \Lambda D$
\State $G_{ext} = \bar{G}_{ext} E, \quad \bar{\Omega} = E \Omega E^T$ as in~\eqref{eqn:renormalize}
\State $\bar{\Omega} = \bar{L} \bar{L}^T$ 
\LineComment{Assemble the intrinsic sparse modes $G$}
\State $G =  \bar{G}_{ext}  \bar{L}$
\end{algorithmic}
\end{algorithm}

Here we replace the pivoted Cholesky decomposition of $\Omega$ in Algorithm~\ref{alg:ISMD} by eigen decomposition of $\bar{\Omega}$. From Remark~\ref{rem:pivotCholRem}, this modified version generates exactly the same result with Algorithm~\ref{alg:ISMD} if all the intrinsic sparse modes have different $l^2$ norm (there are no repeated eigenvalues in $\bar{\Omega}$). The advantage of the pivoted Cholesky decomposition is its low computational cost and the fact that it always exploits the (unordered) block diagonal structure of $\Omega$. However, it is more sensitive to noise compared to eigen decomposition. In contrast, eigen decomposition is much more robust to noise. Moreover, eigen decomposition gives the optimal low rank approximation of $\bar{\Omega}$. Thus Algorithm~\ref{alg:ISMDm2} gives a more accurate low rank approximation for $A$ compared to Algorithm~\ref{alg:ISMD} and Algorithm~\ref{alg:ISMDm1} that use the pivoted Cholesky decomposition.

\section{Numerical experiments} \label{sec:numericalExamples}
In this section, we demonstrate the robustness of our intrinsic sparse mode decomposition method and compare its performance with that of the eigen decomposition, the pivoted Cholesky decomposition, and the convex relaxation of sparse PCA. All our computations are performed using MATLAB R2015a (64-bit) on an Intel(R) Core(TM) i7-3770 (3.40 GHz). The pivoted Cholesky decomposition is implemented in MATLAB according to Algorithm~3.1 in~\cite{lucas2004lapack}.

We will use synthetic covariance matrices of a random permeability field, which models the underground porous media, as the symmetric PSD input $A$. This random permeability model is adapted from the porous media problem~\cite{ghommem_mode_2013,galvis_domain_2010} where the physical domain $D$ is two dimensional. The basic model has a constant background and several localized features to model the subsurface channels and inclusions, i.e.,
\begin{equation} \label{eqn:example2dsparse}
\kappa(x,\omega) = \kappa_0 + \sum_{k = 1}^{K}\eta_k(\omega) g_k(x), \quad x \in [0,1]^2,
\end{equation}
where $\kappa_0$ is the constant background, $\{g_k\}_{k=1}^K$ are characteristic functions of channels and inclusions and $\eta_k$ are the associated uncorrelated latent variables controlling the permeability of each feature. Here, we have $K=35$, including 16 channels and 18 inclusions. Among these modes, there is one artificial smiling face mode that has disjoint branches. It is used here to demonstrate that the ISMD is able to capture long range correlation. For this random medium, the covariance function is
\begin{equation} \label{eqn:example2dsparseA}
	a(x,y) = \sum_{k = 1}^{K} g_k(x) g_k(y), \quad x, y \in [0,1]^2.
\end{equation}
Since the length scales of channels and inclusions are very small, with width about $1/32$, we need a fine grid to resolve these small features. Such a fine grid is also needed when we do further scientific experiments~\cite{ghommem_mode_2013,galvis_domain_2010, hou_LocalgPC_2016}. In this paper, the physical domain $D=[0,1]^2$ is discretized using a uniform grid with $h_x = h_y = 1/96$, resulting in $A \in \RR^{N\times N}$ with $N = 96^2$. One sample of the random field (and the bird's-eye view) and the covariance matrix are plotted in Figure~\ref{fig:2d}. It can be seen that the covariance matrix is sparse and concentrates along the diagonal since modes in the ground-truth media are all localized functions.
\begin{figure}[h]
\centering
\includegraphics[width = 0.30\textwidth]{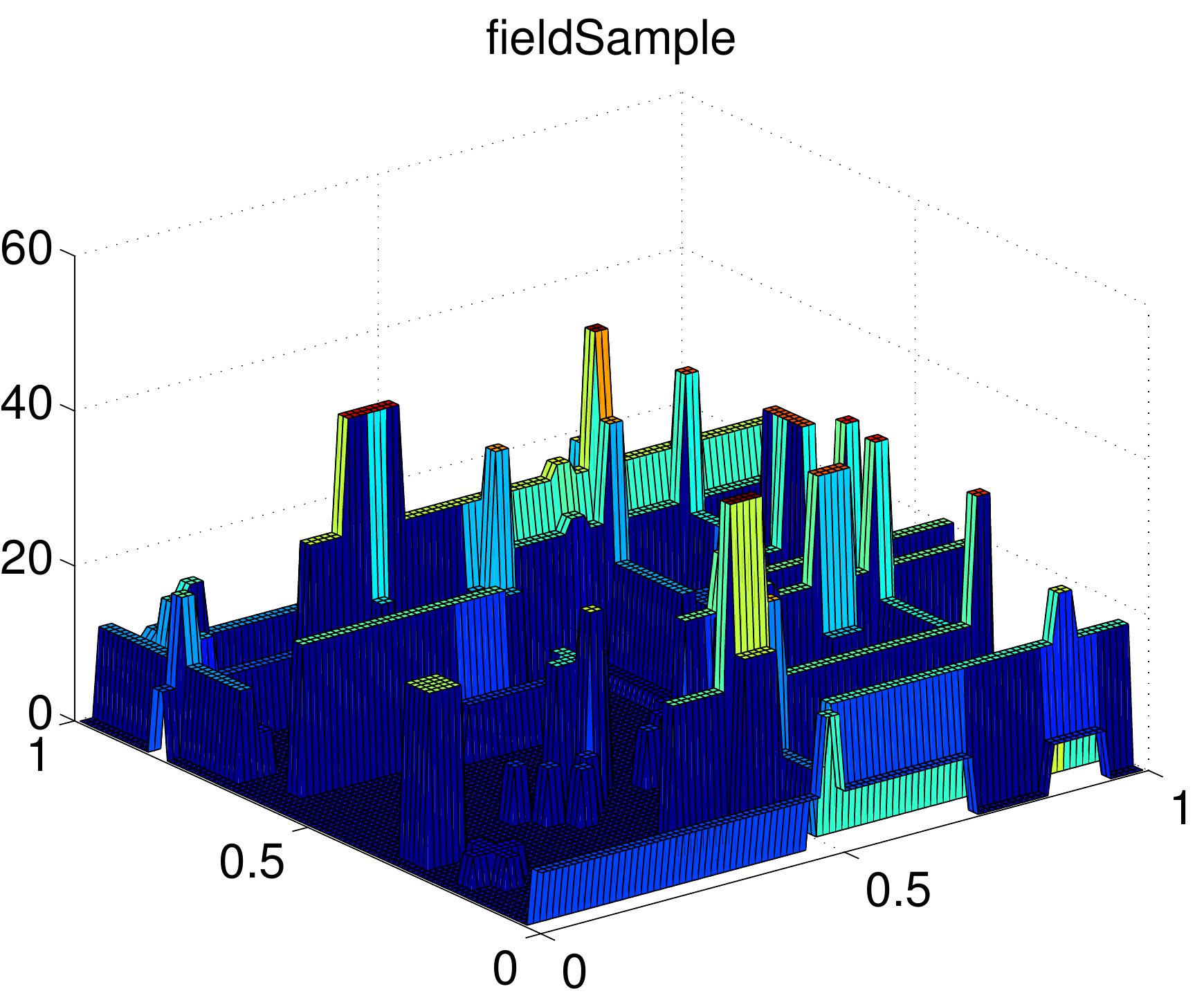}
\includegraphics[width = 0.30\textwidth]{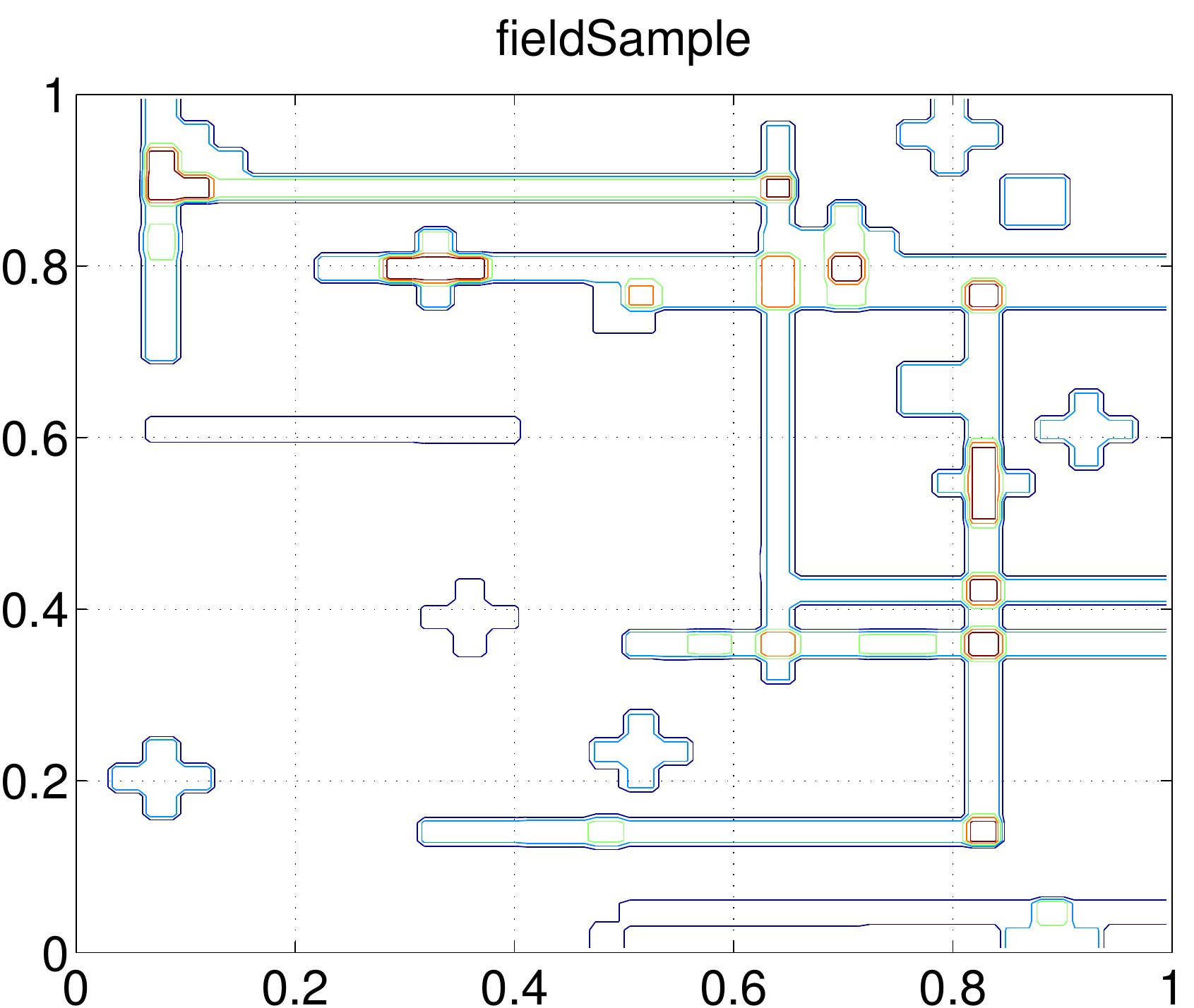}
\includegraphics[width = 0.30\textwidth]{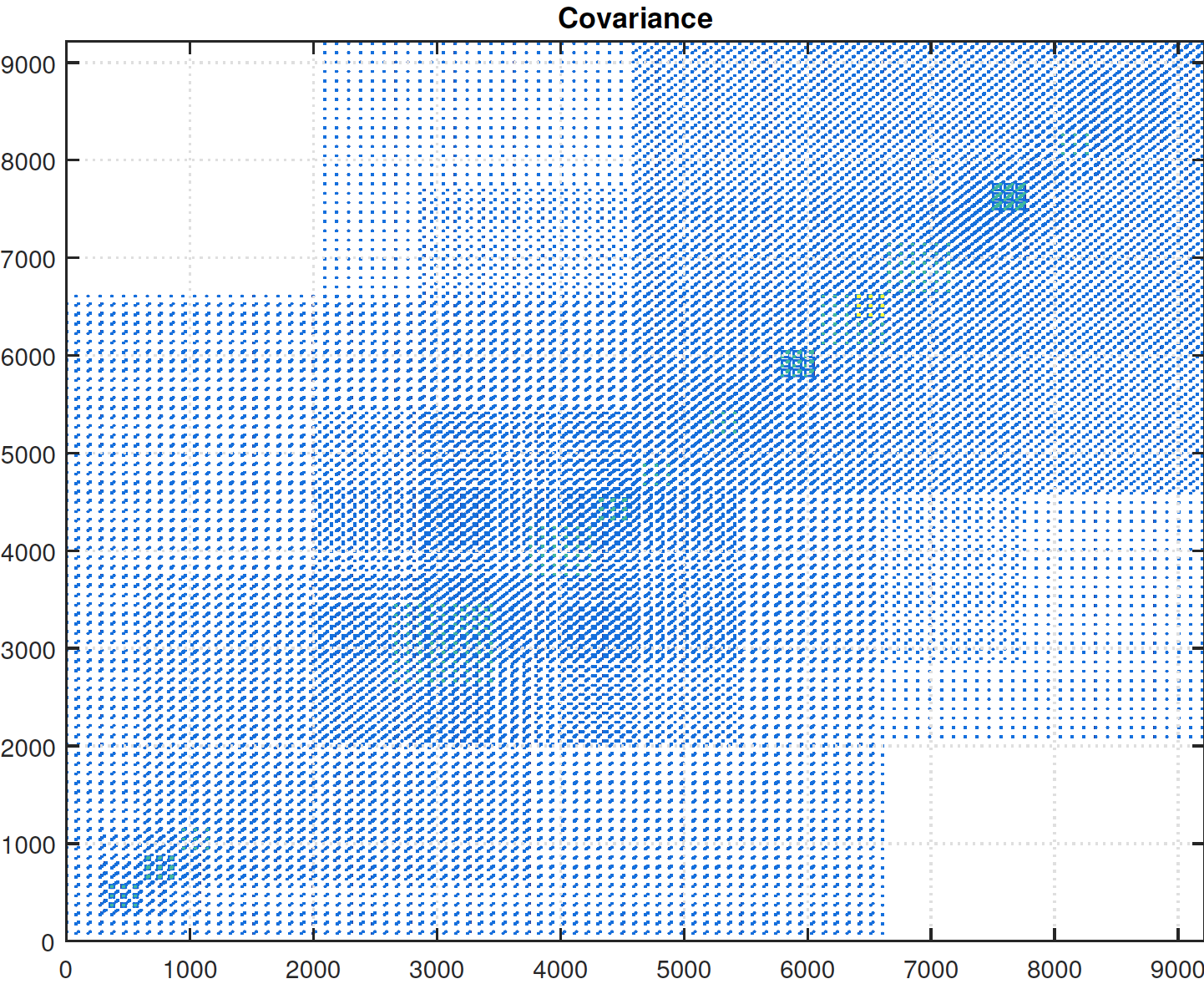}
\caption{One sample and the bird's-eye view. The covariance matrix is plotted on the right.}\label{fig:2d}
\end{figure}

Note that this example is synthetic because we construct $A$ from a sparse decomposition~\eqref{eqn:example2dsparseA}. We would like to test whether different matrix factorization methods, like eigen decomposition, the Cholesky decomposition and the ISMD, are able to recover this sparse decomposition, or even find a sparser decomposition for $A$.

\subsection{ISMD}
The partitions we take for this example are all uniform domain partition with $H_x = H_y = H$. We run the ISMD with patch sizes $H \in \{1, 1/2, 1/3, 1/4, 1/6, 1/8, 1/12, 1/16, 1/24, 1/32, 1/48, 1/96\}$ in this section. For the coarsest partition $H=1$, the ISMD is exactly the eigen decomposition of $A$. For the finest partition $H = 1/96$, the ISMD is equivalent to the pivoted Cholesky factorization on $\bar{A}$ where $\bar{A}_{ij} = \frac{A_{ij}}{\sqrt{A_{ii}A_{jj}}}$. The pivoted Cholesky factorization on $A$ is also implemented. It is no surprise that all the above methods produce 35 modes. The number of modes is exactly the rank of $A$. We plot the first 6 modes for each method in Figure~\ref{fig:2d_compare1}. We can see that both the eigen decomposition (ISMD with $H = 1$) and the pivoted Cholesky factorization on $A$ generate modes which mix different localized feathers together. On the other hand, the ISMD with $H = 1/8$ and $H=1/32$ exactly recover the localized feathers, including the smiling face. 

\begin{figure}
\centering
\includegraphics[width = 0.15\textwidth,height = 0.15\textheight]{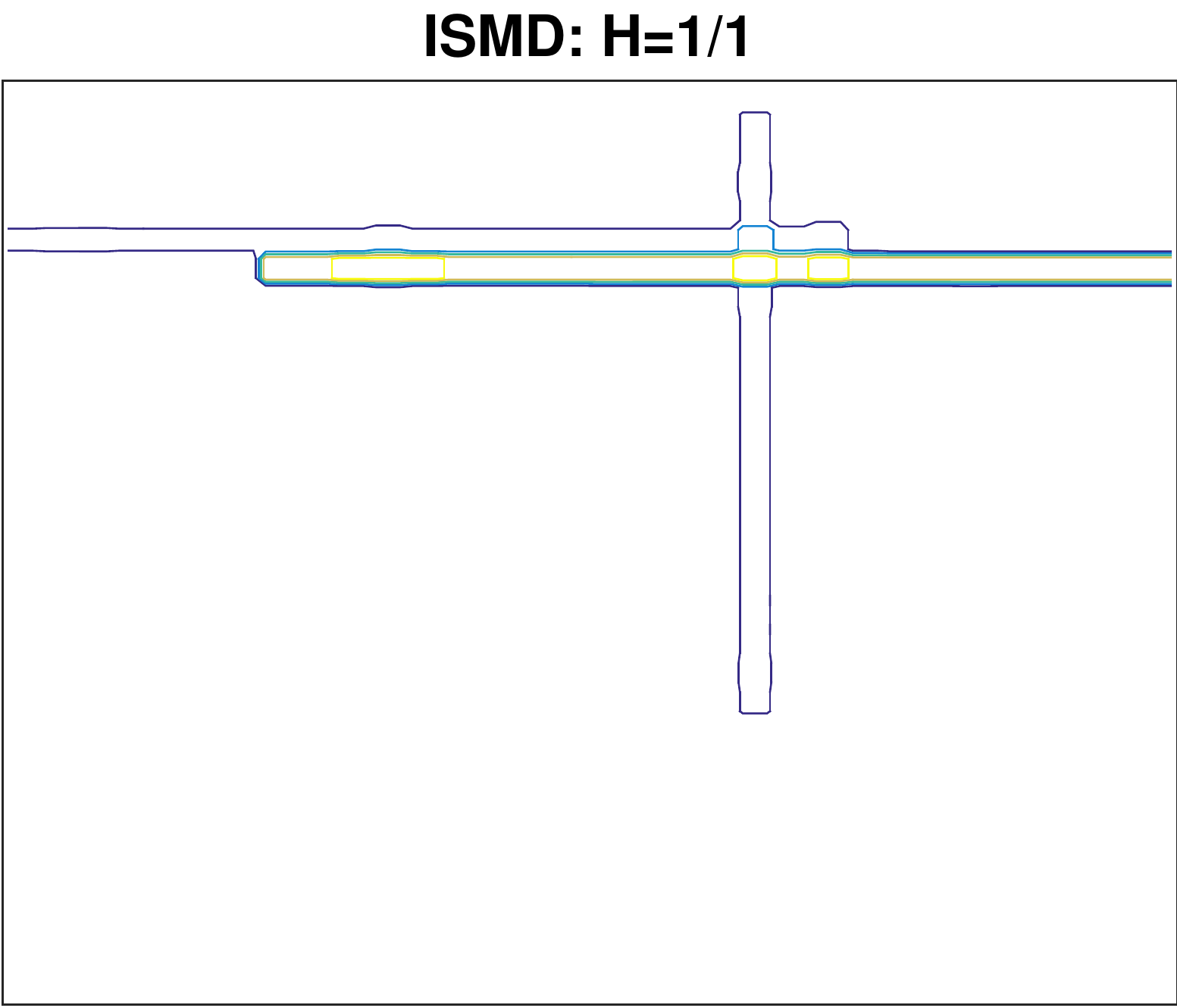}
\includegraphics[width = 0.15\textwidth,height = 0.15\textheight]{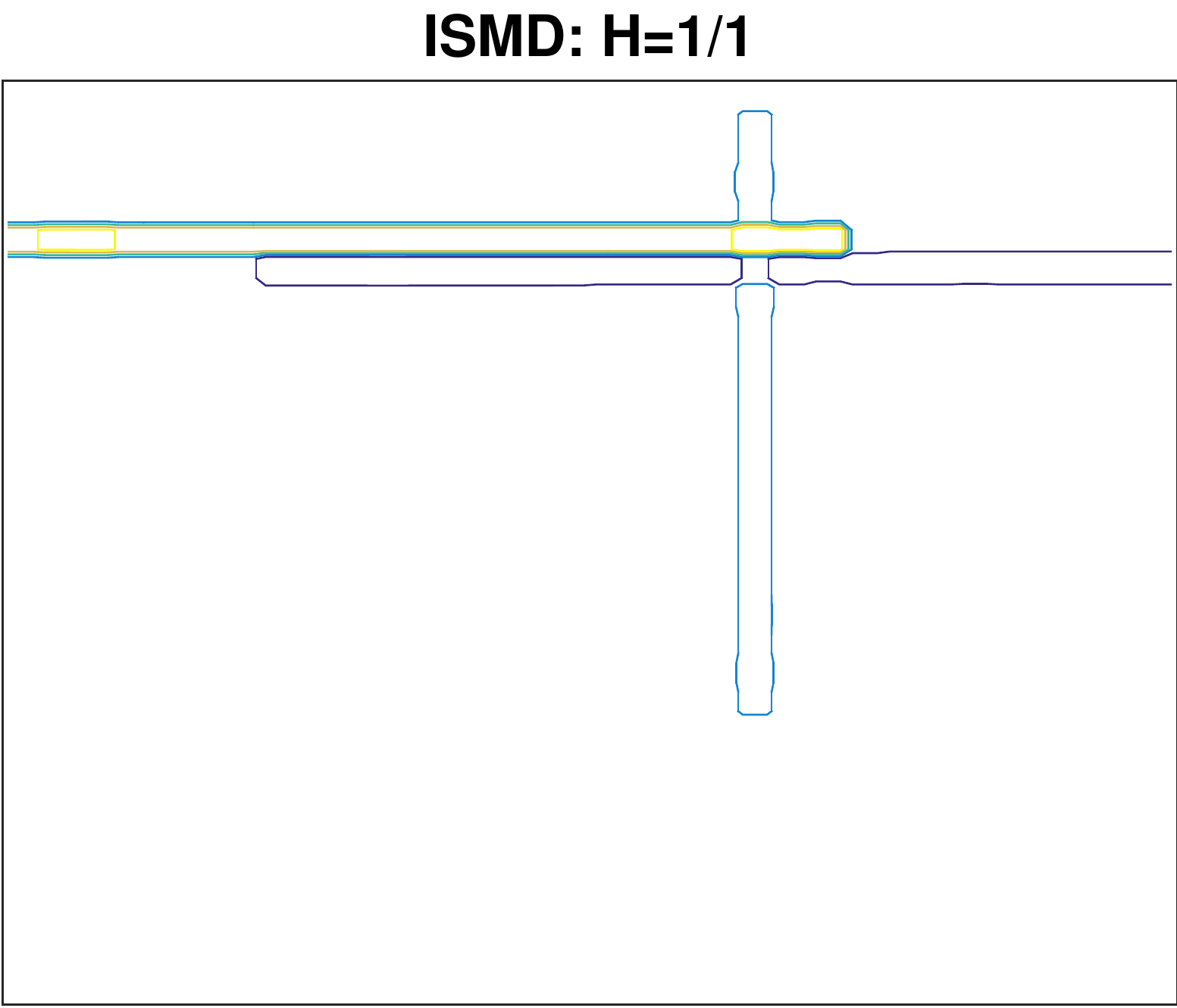}
\includegraphics[width = 0.15\textwidth,height = 0.15\textheight]{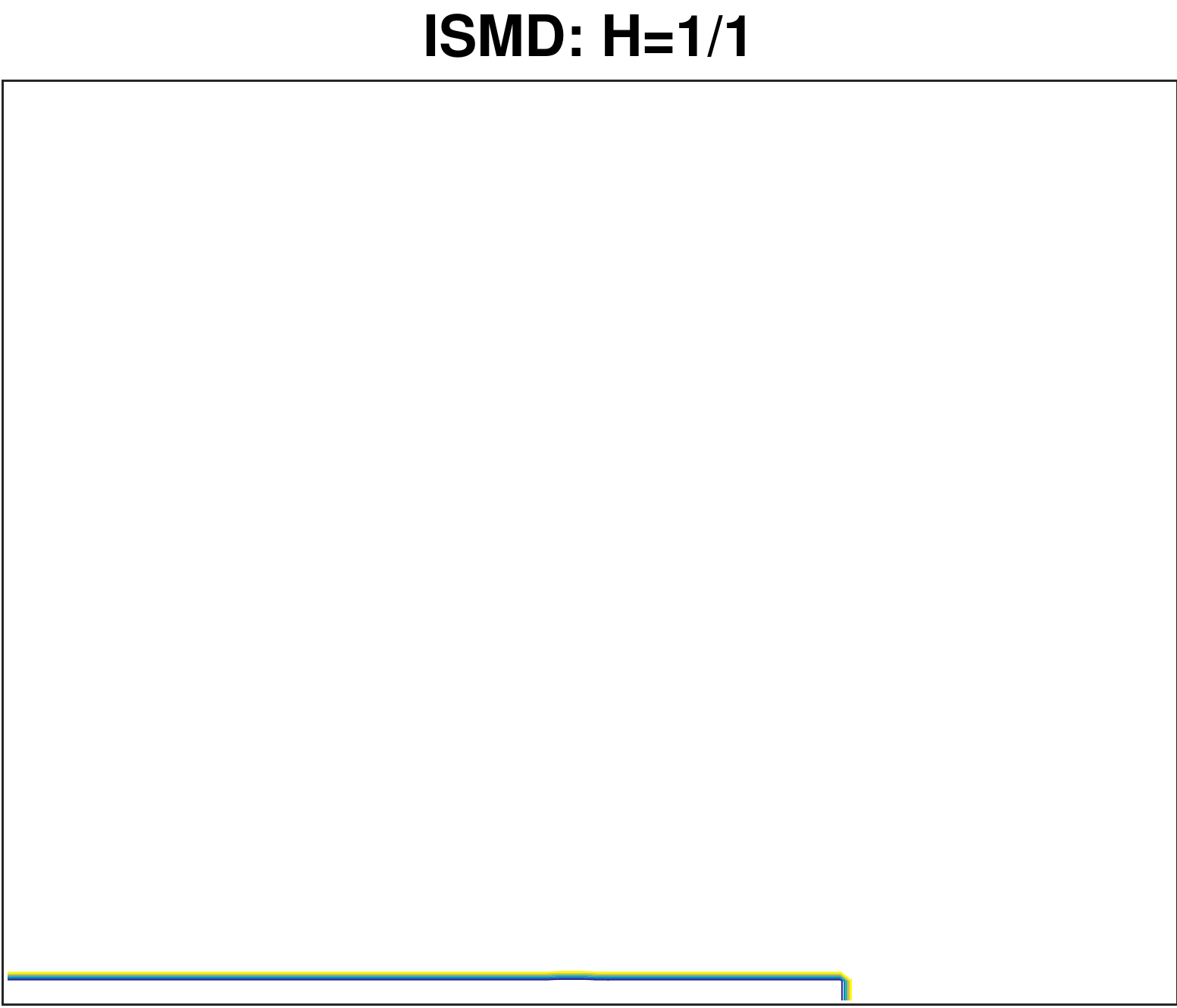}
\includegraphics[width = 0.15\textwidth,height = 0.15\textheight]{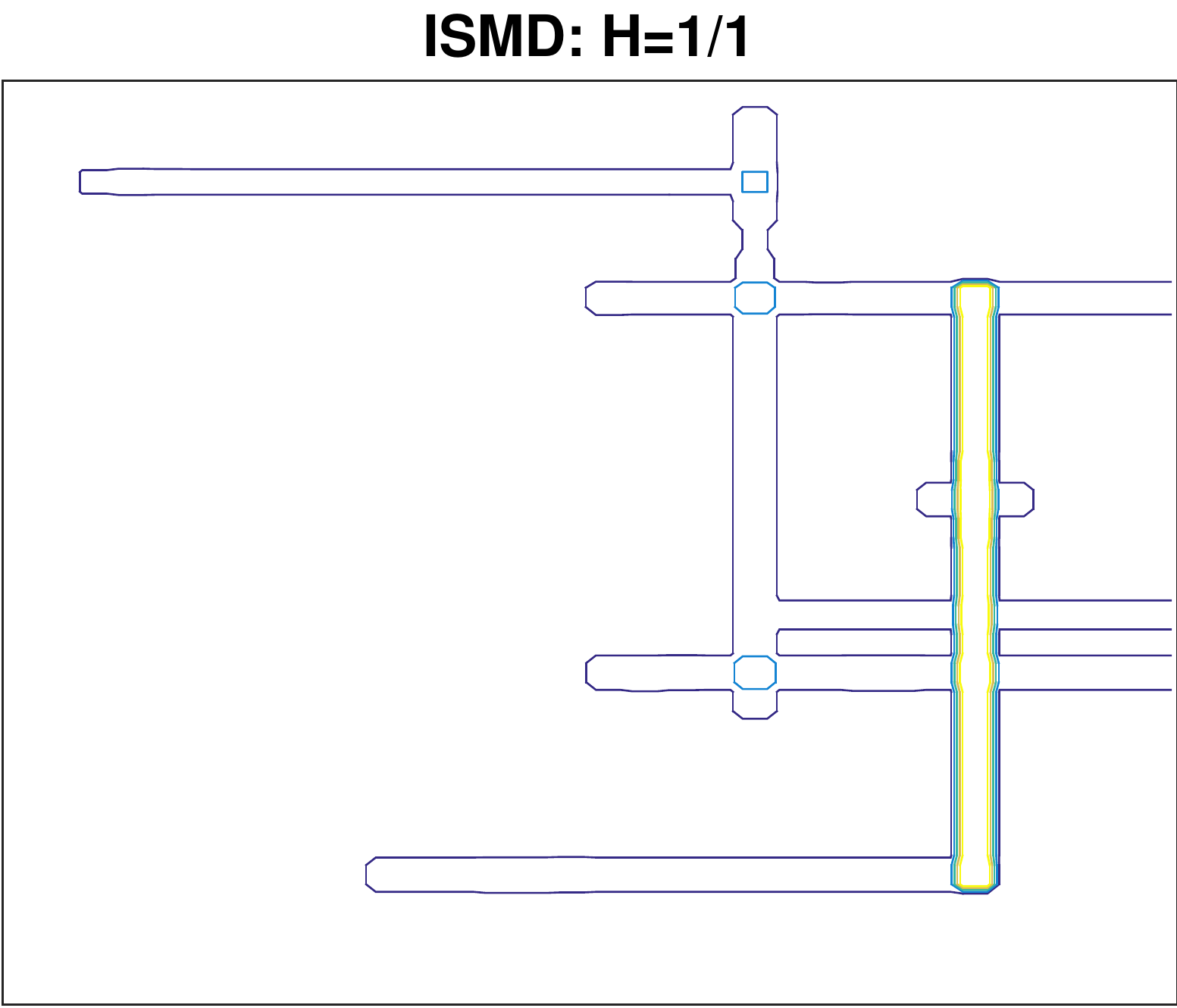} 
\includegraphics[width = 0.15\textwidth,height = 0.15\textheight]{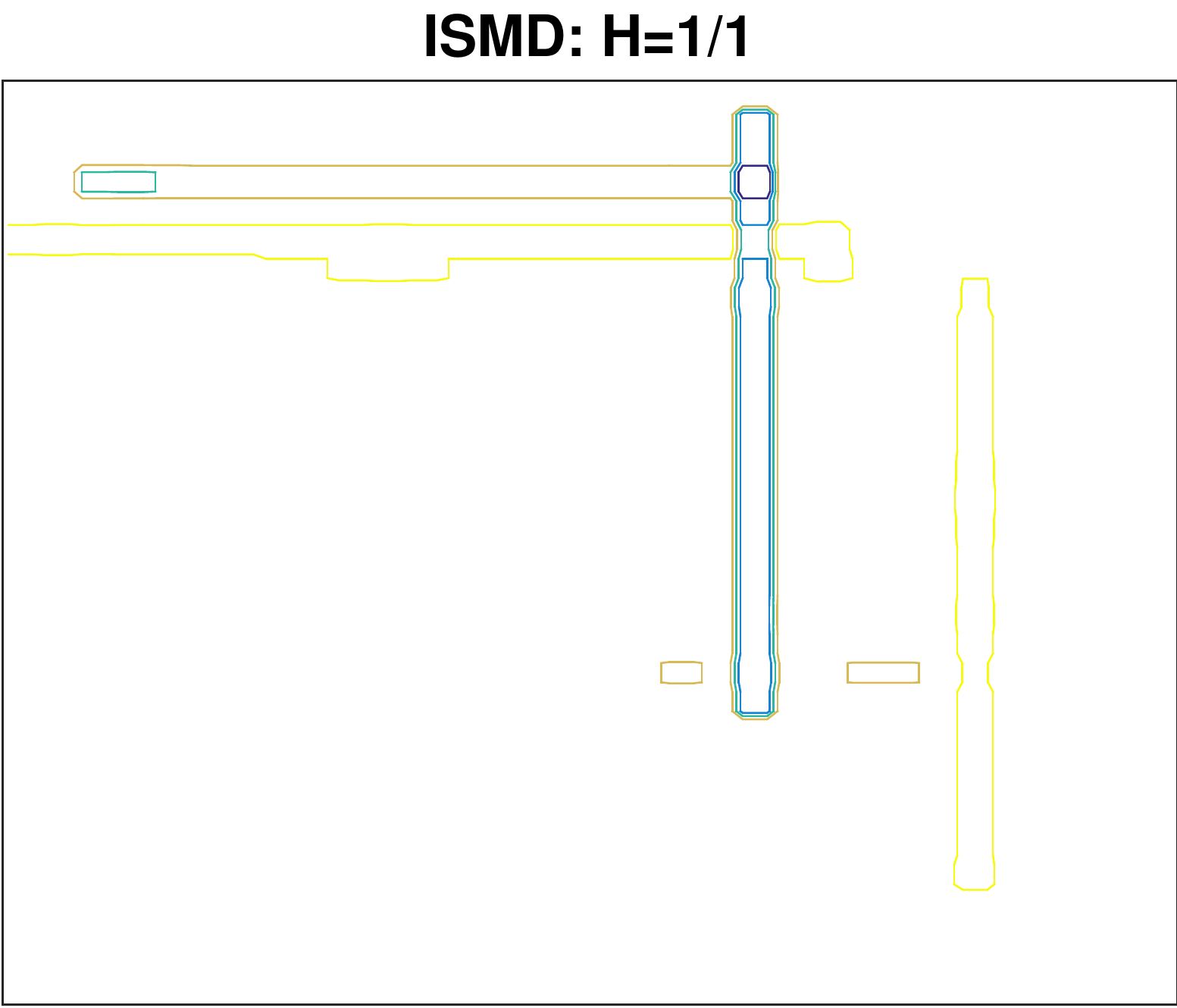}
\includegraphics[width = 0.15\textwidth,height = 0.15\textheight]{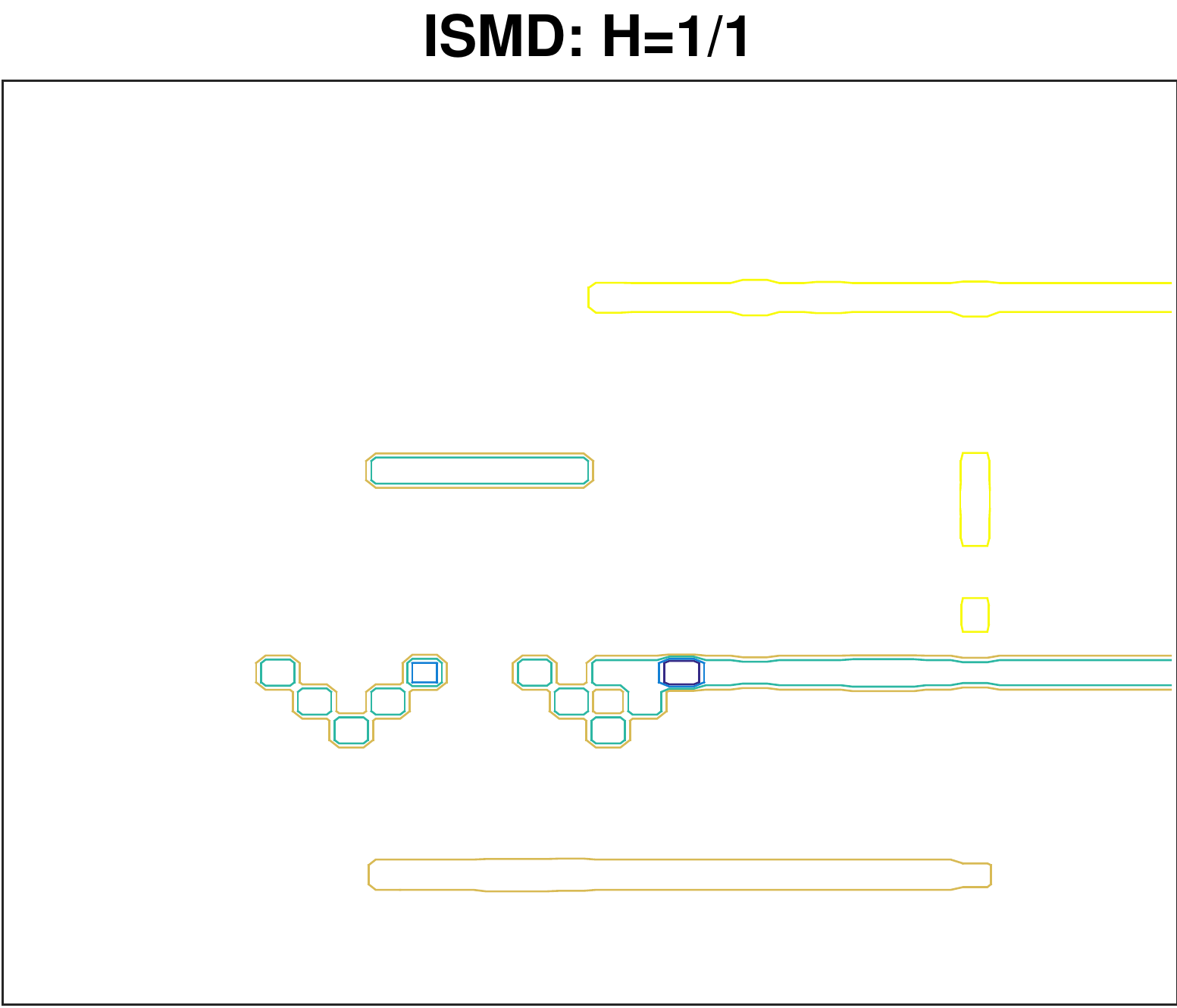}\\
\includegraphics[width = 0.15\textwidth,height = 0.15\textheight]{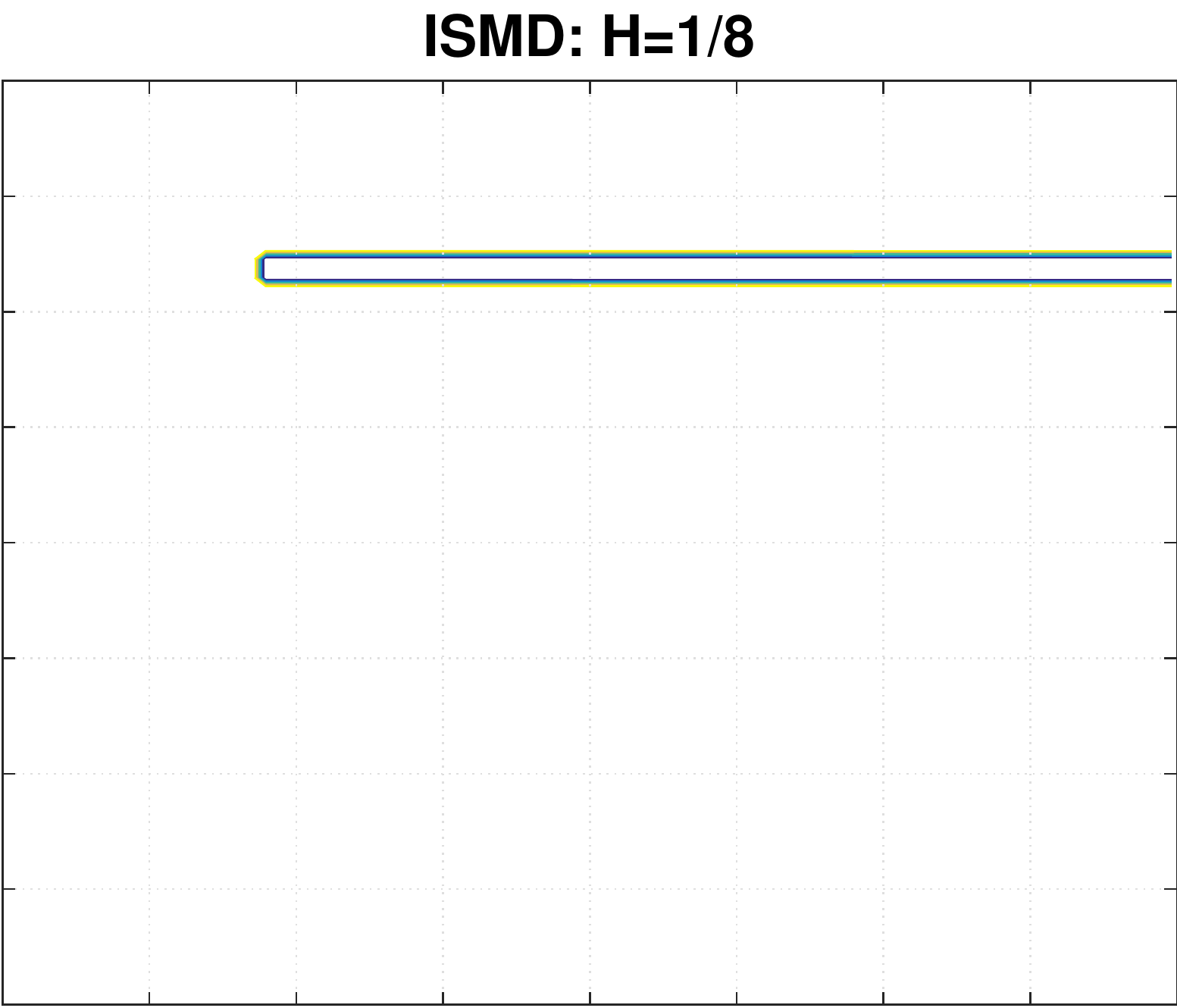}
\includegraphics[width = 0.15\textwidth,height = 0.15\textheight]{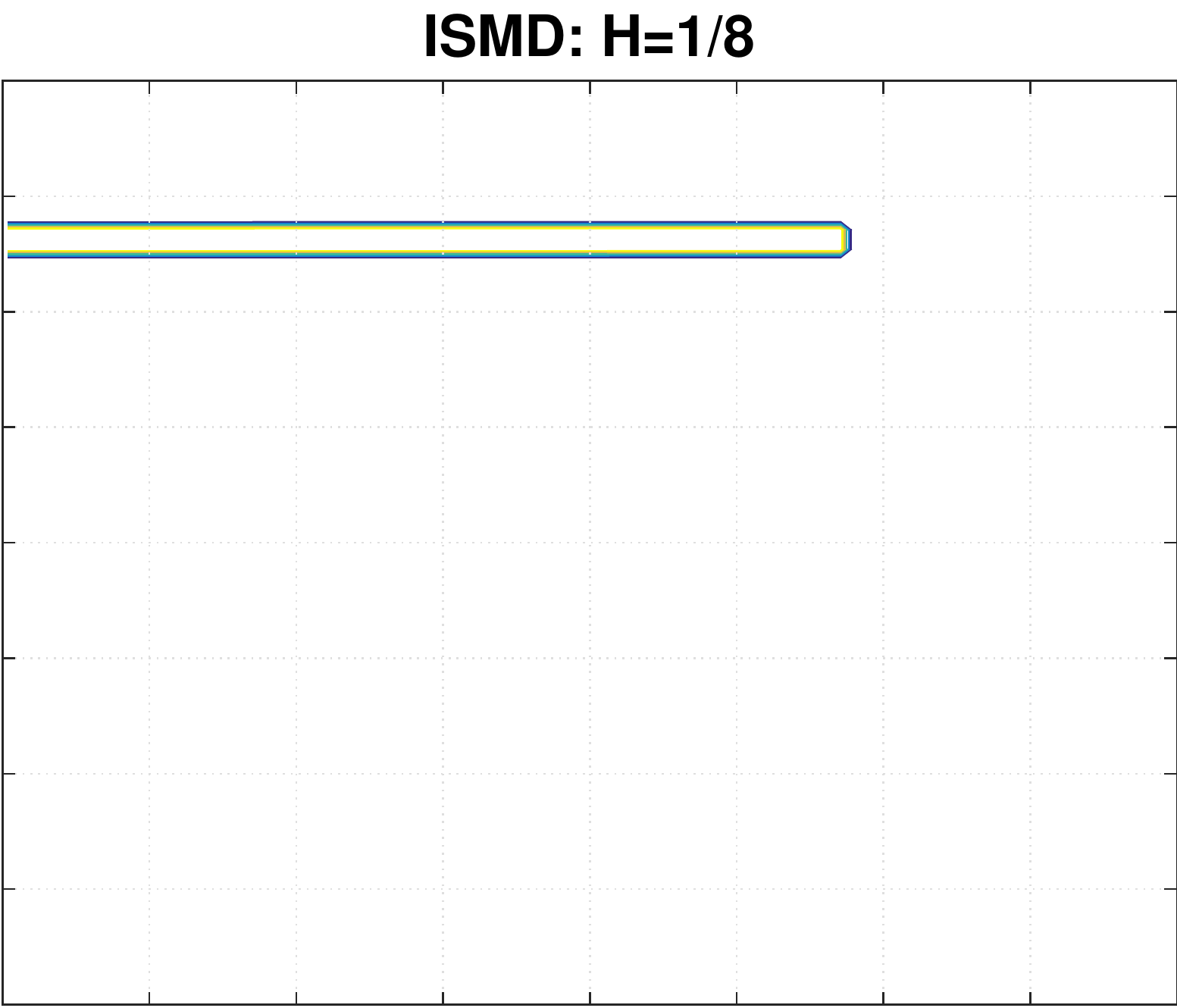}
\includegraphics[width = 0.15\textwidth,height = 0.15\textheight]{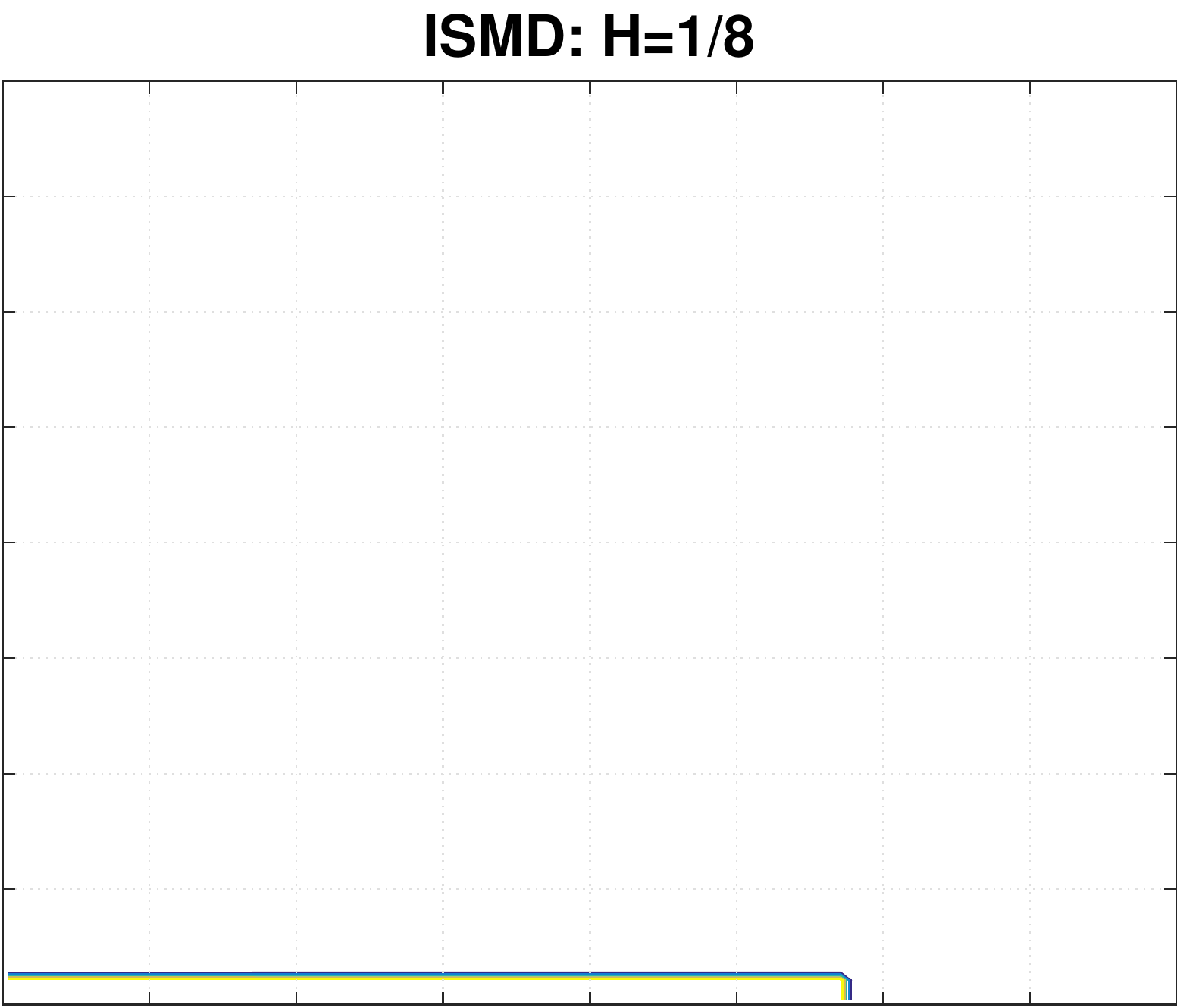}
\includegraphics[width = 0.15\textwidth,height = 0.15\textheight]{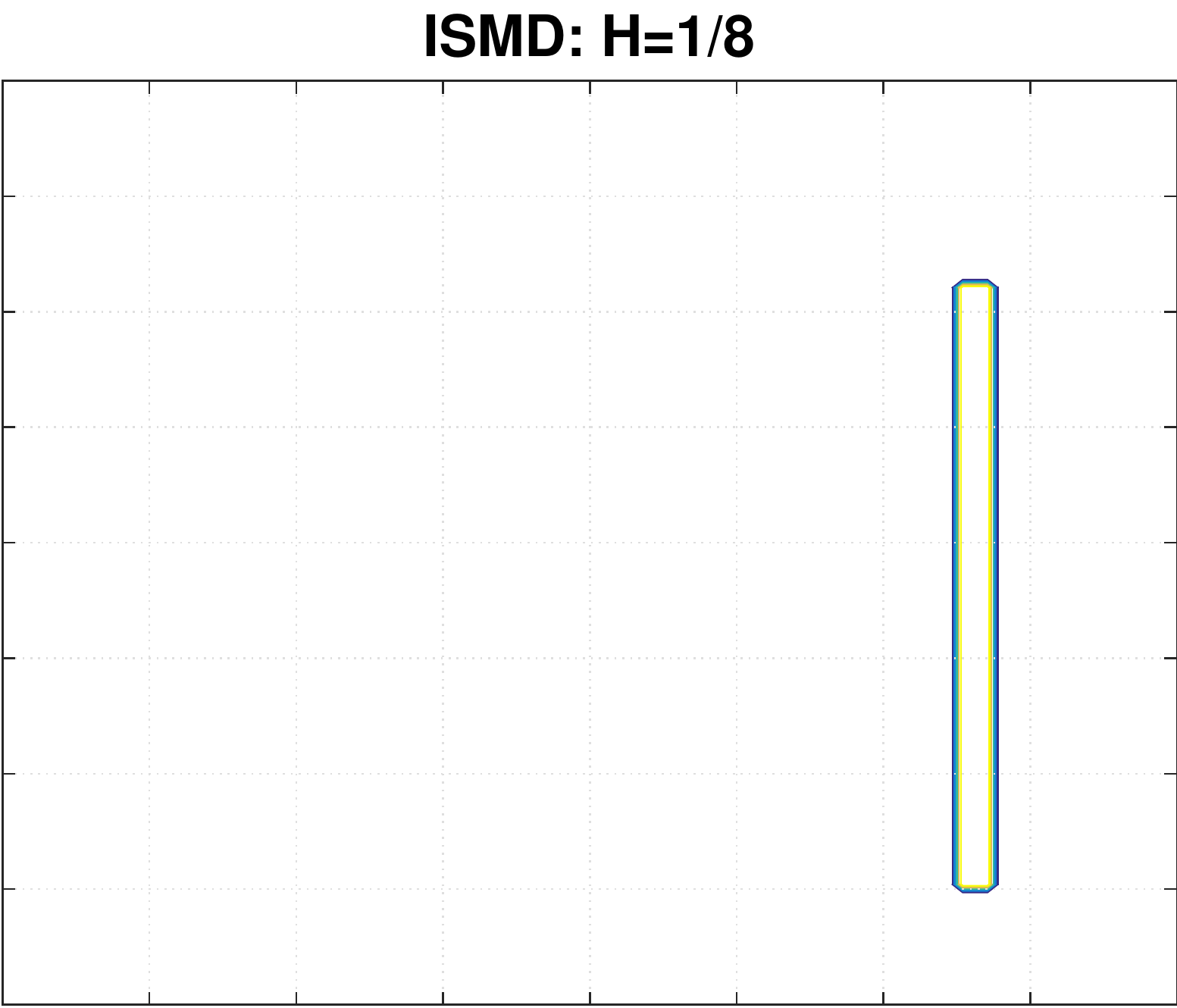}
\includegraphics[width = 0.15\textwidth,height = 0.15\textheight]{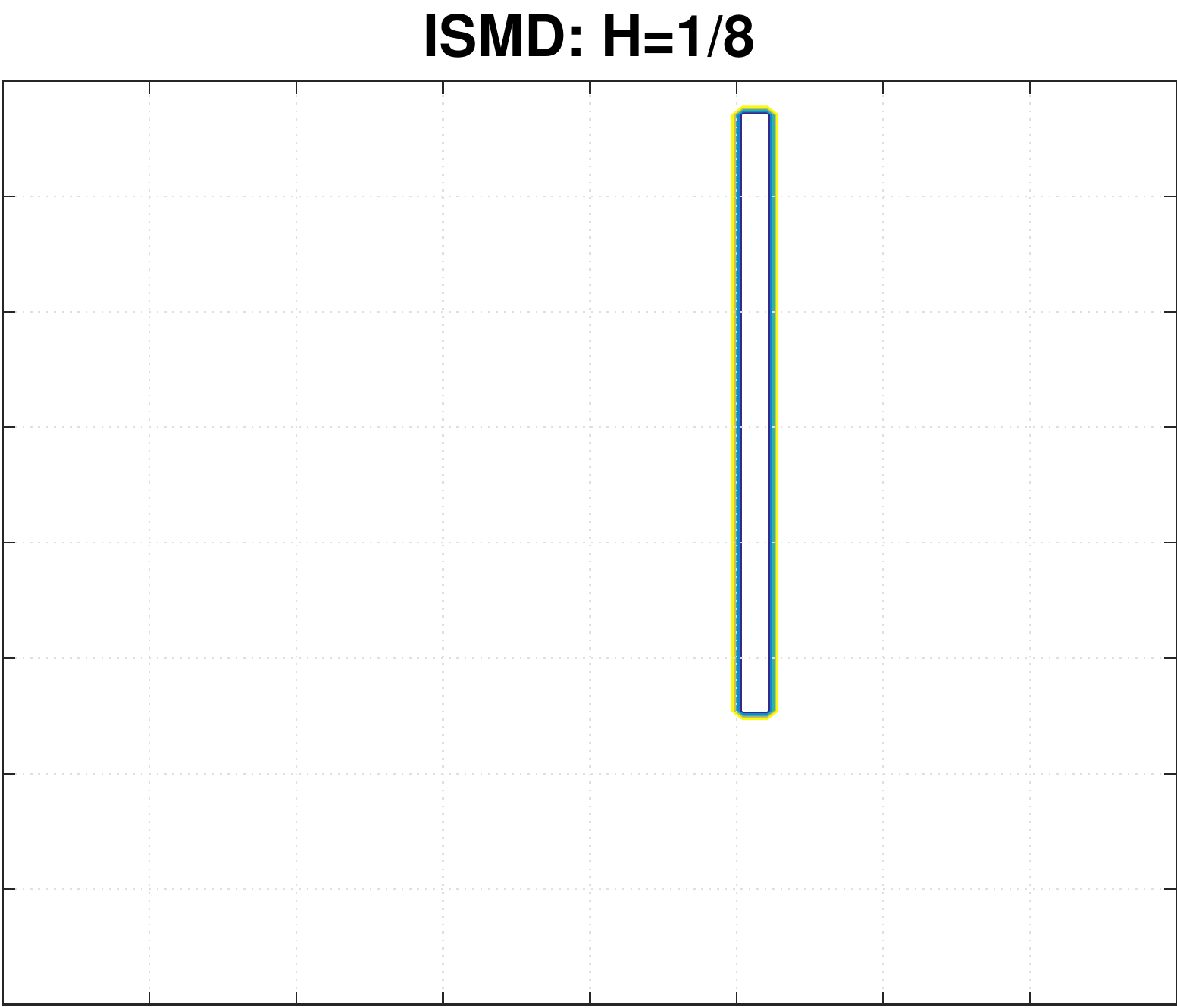}
\includegraphics[width = 0.15\textwidth,height = 0.15\textheight]{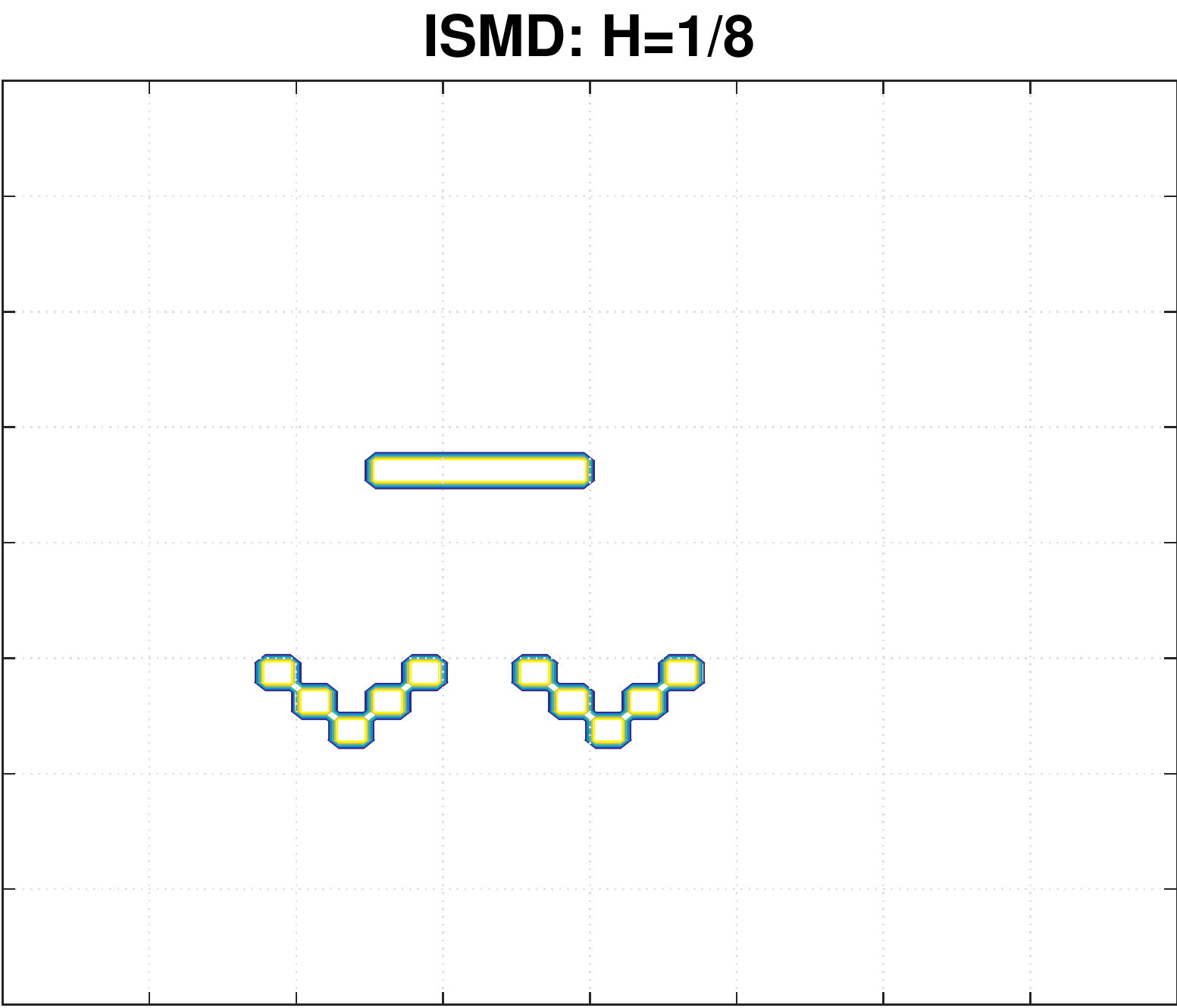}\\
\includegraphics[width = 0.15\textwidth,height = 0.15\textheight]{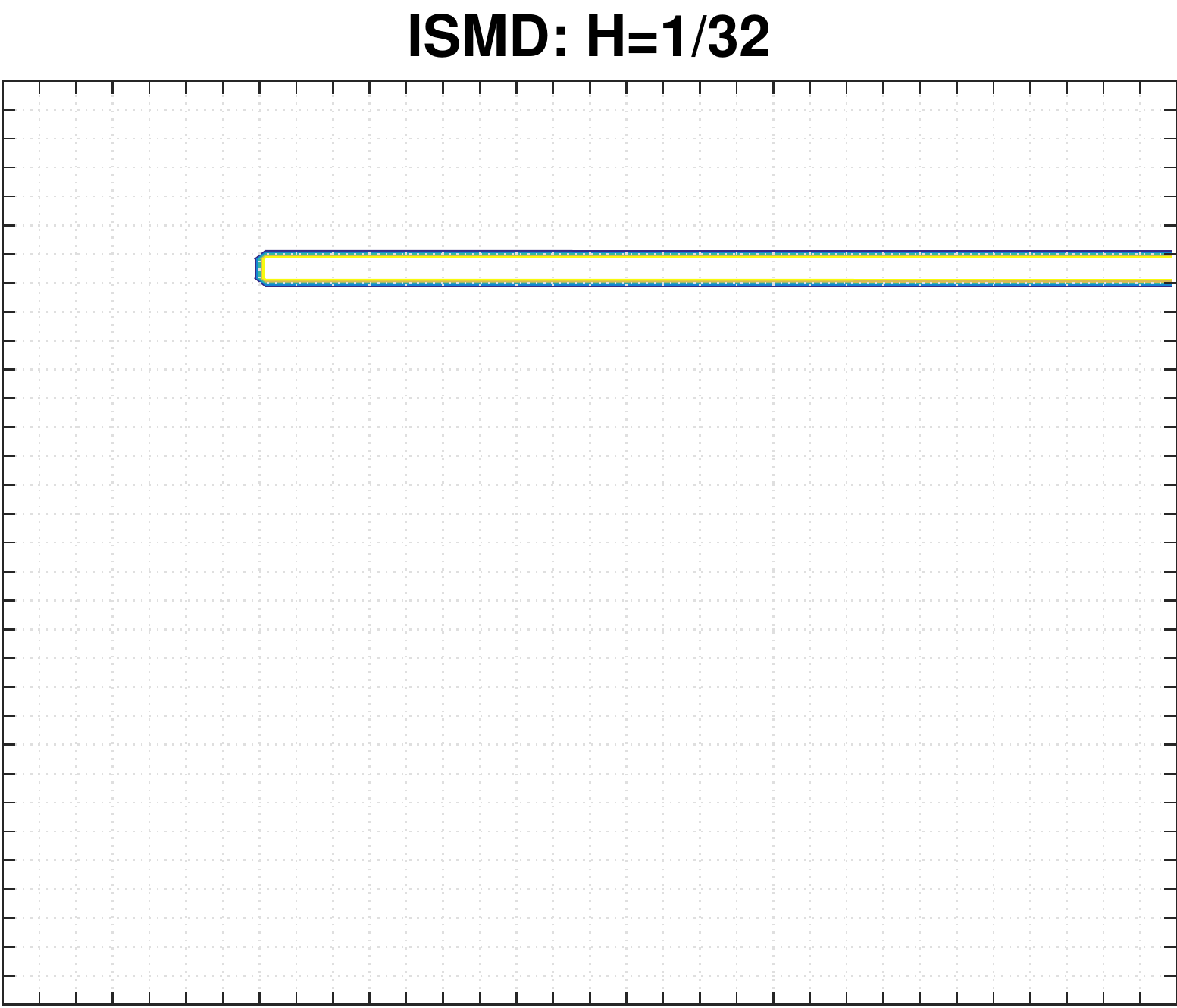}
\includegraphics[width = 0.15\textwidth,height = 0.15\textheight]{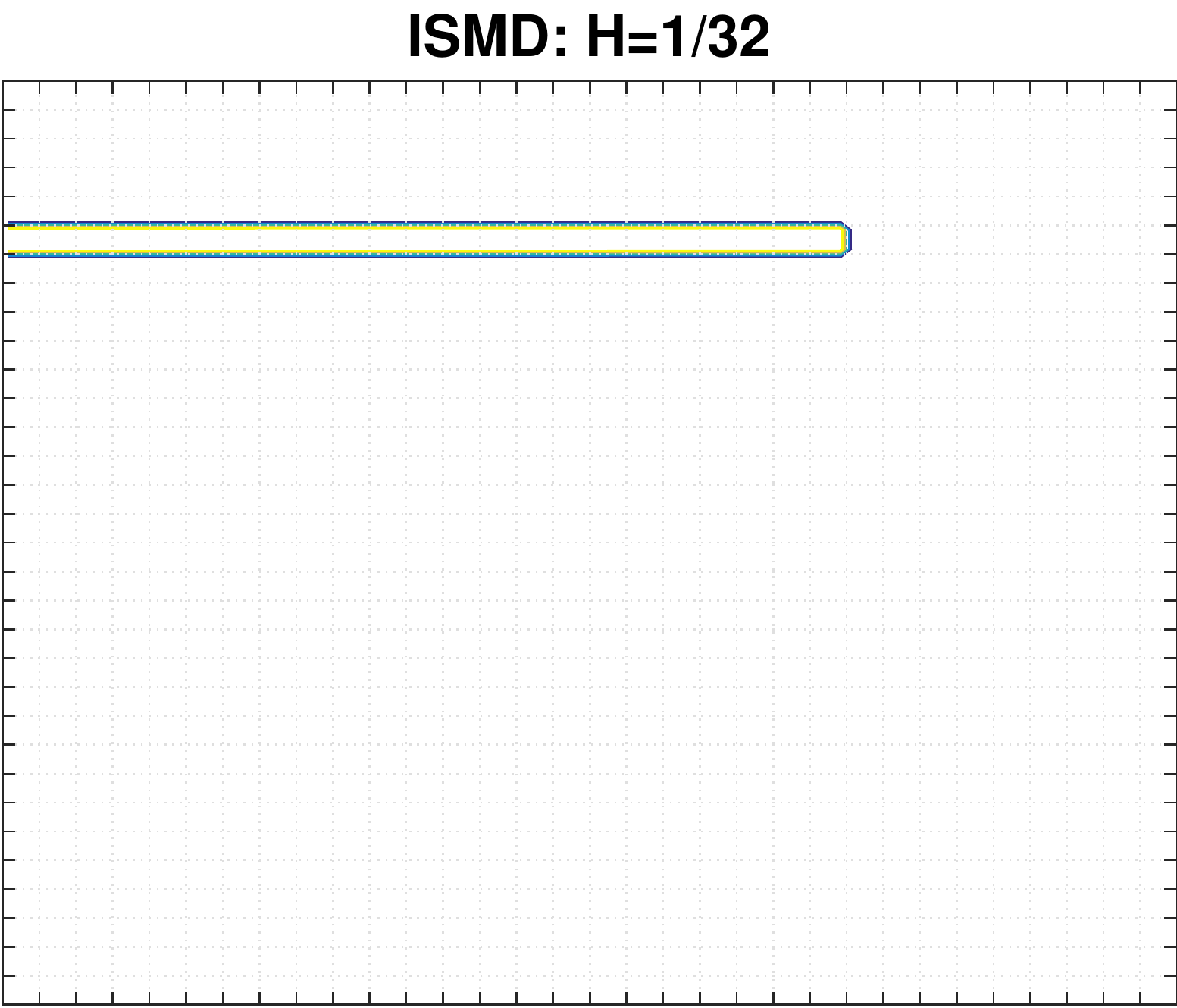}
\includegraphics[width = 0.15\textwidth,height = 0.15\textheight]{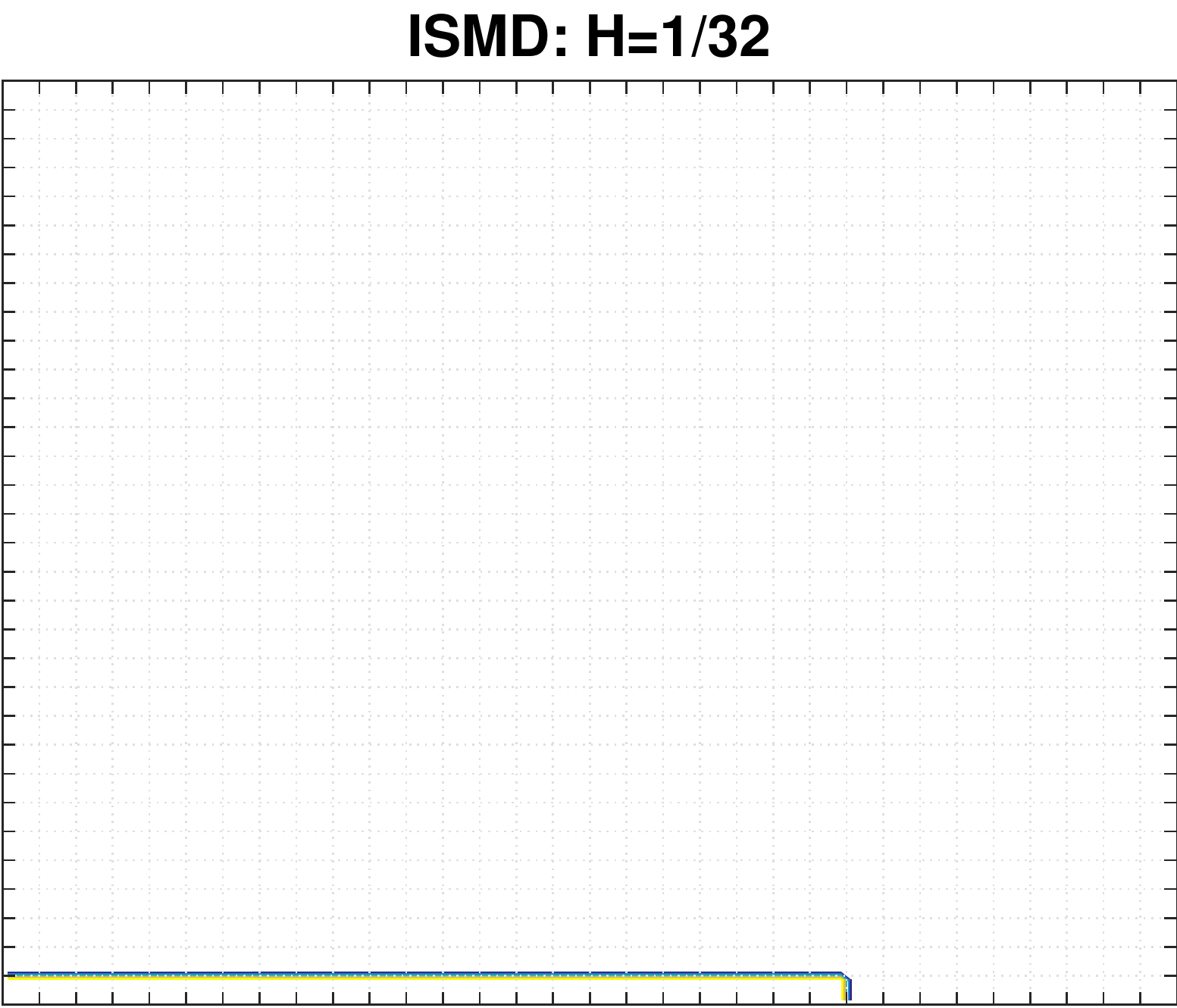}
\includegraphics[width = 0.15\textwidth,height = 0.15\textheight]{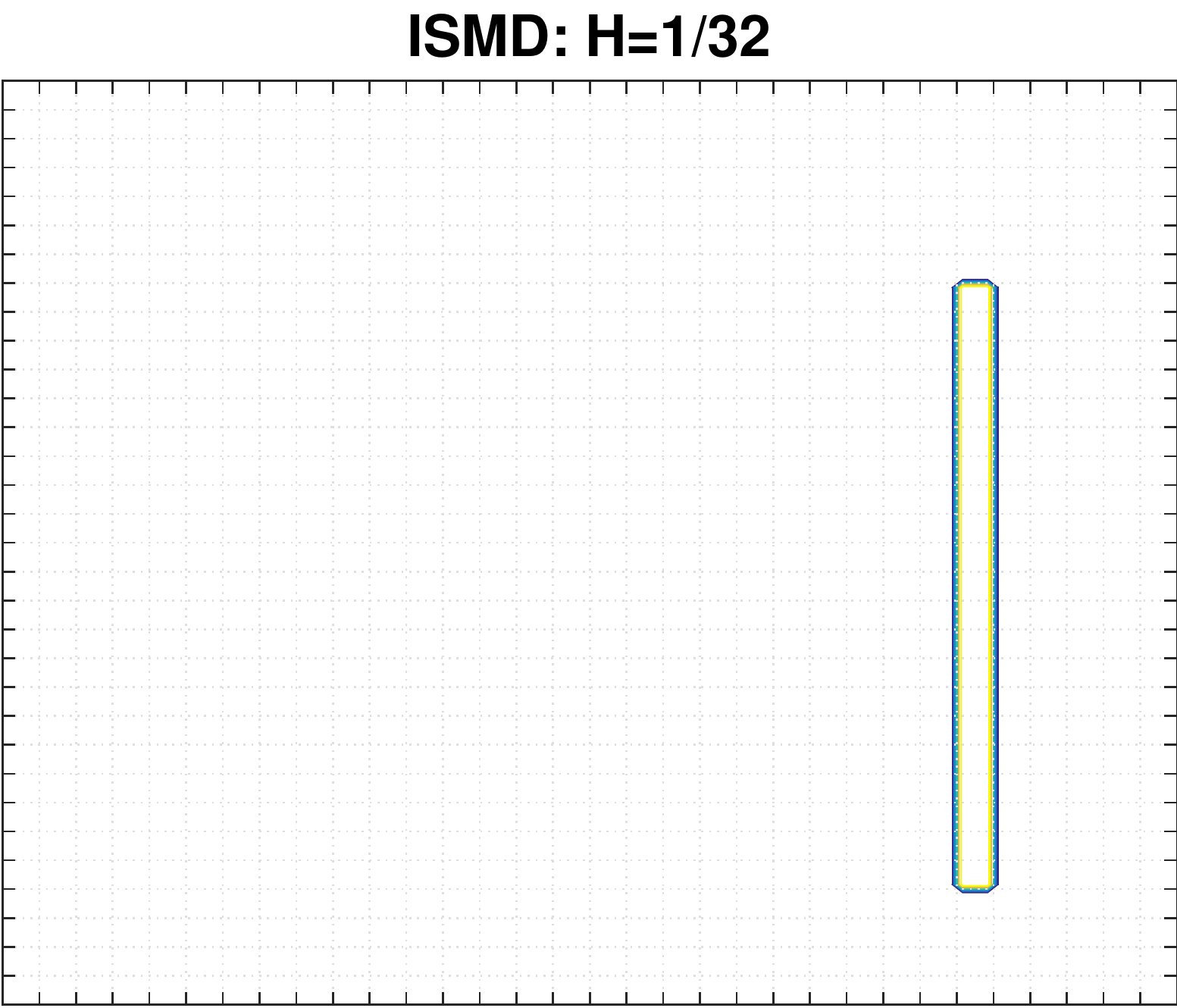}
\includegraphics[width = 0.15\textwidth,height = 0.15\textheight]{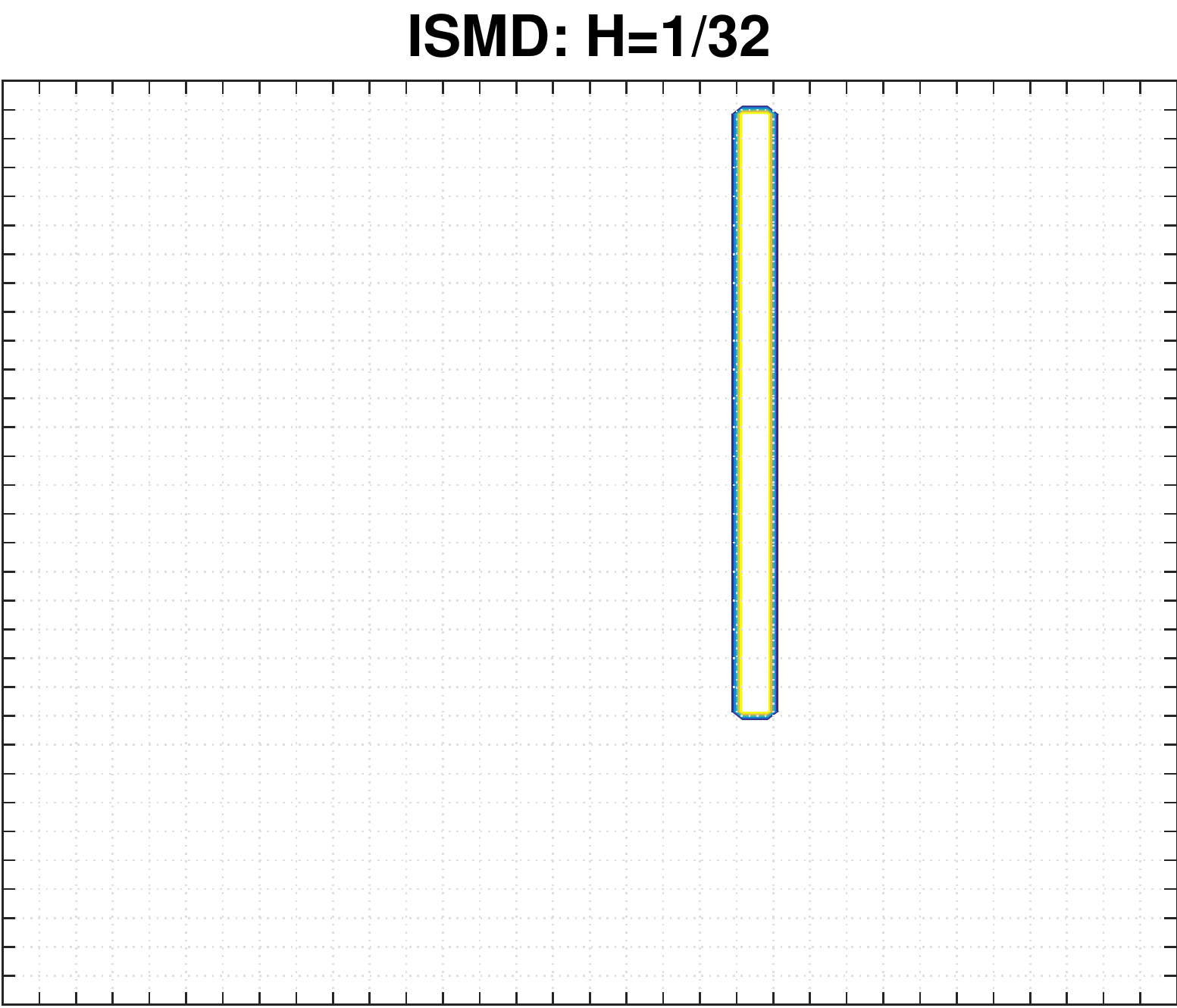}
\includegraphics[width = 0.15\textwidth,height = 0.15\textheight]{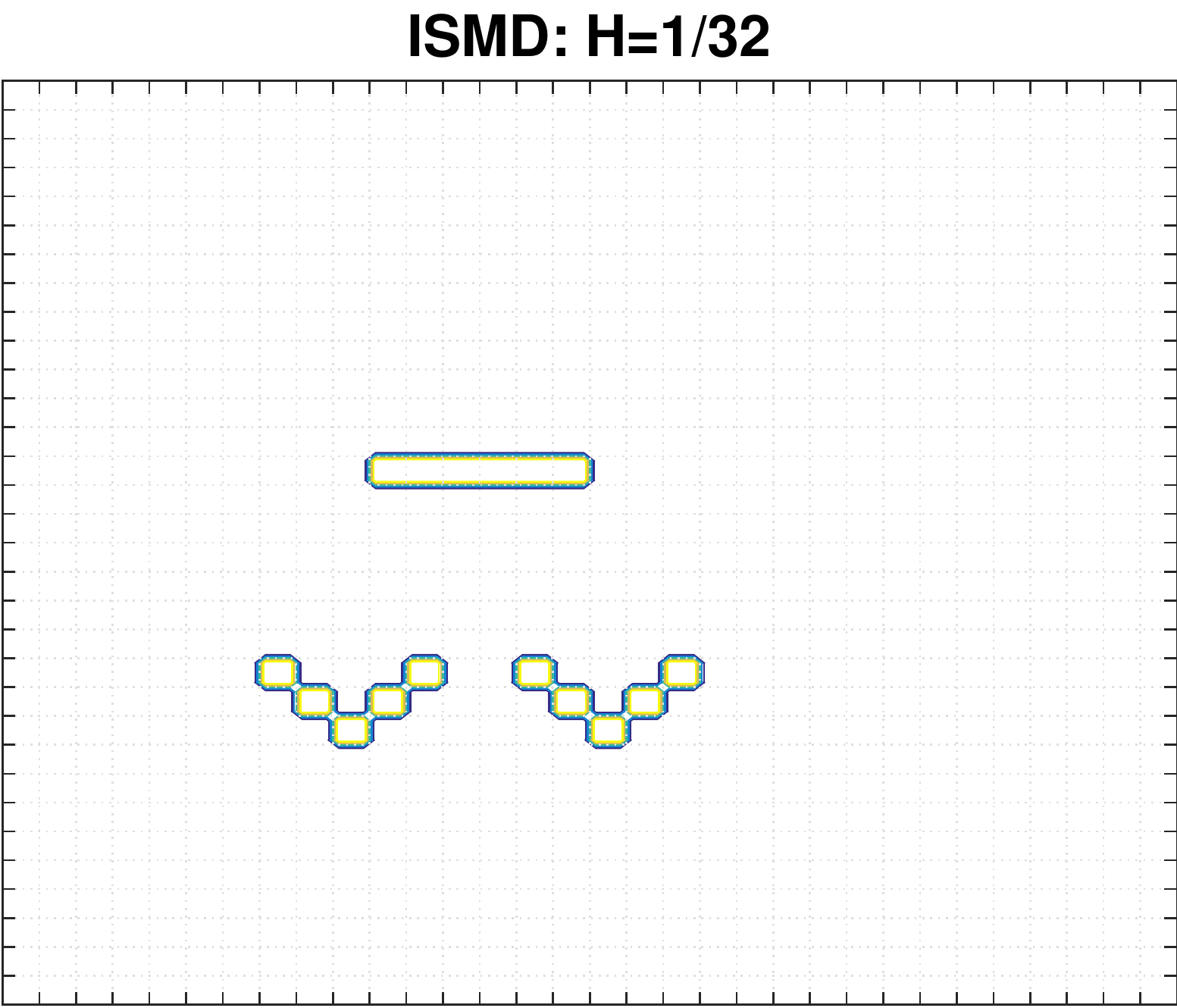}\\\includegraphics[width = 0.15\textwidth,height = 0.15\textheight]{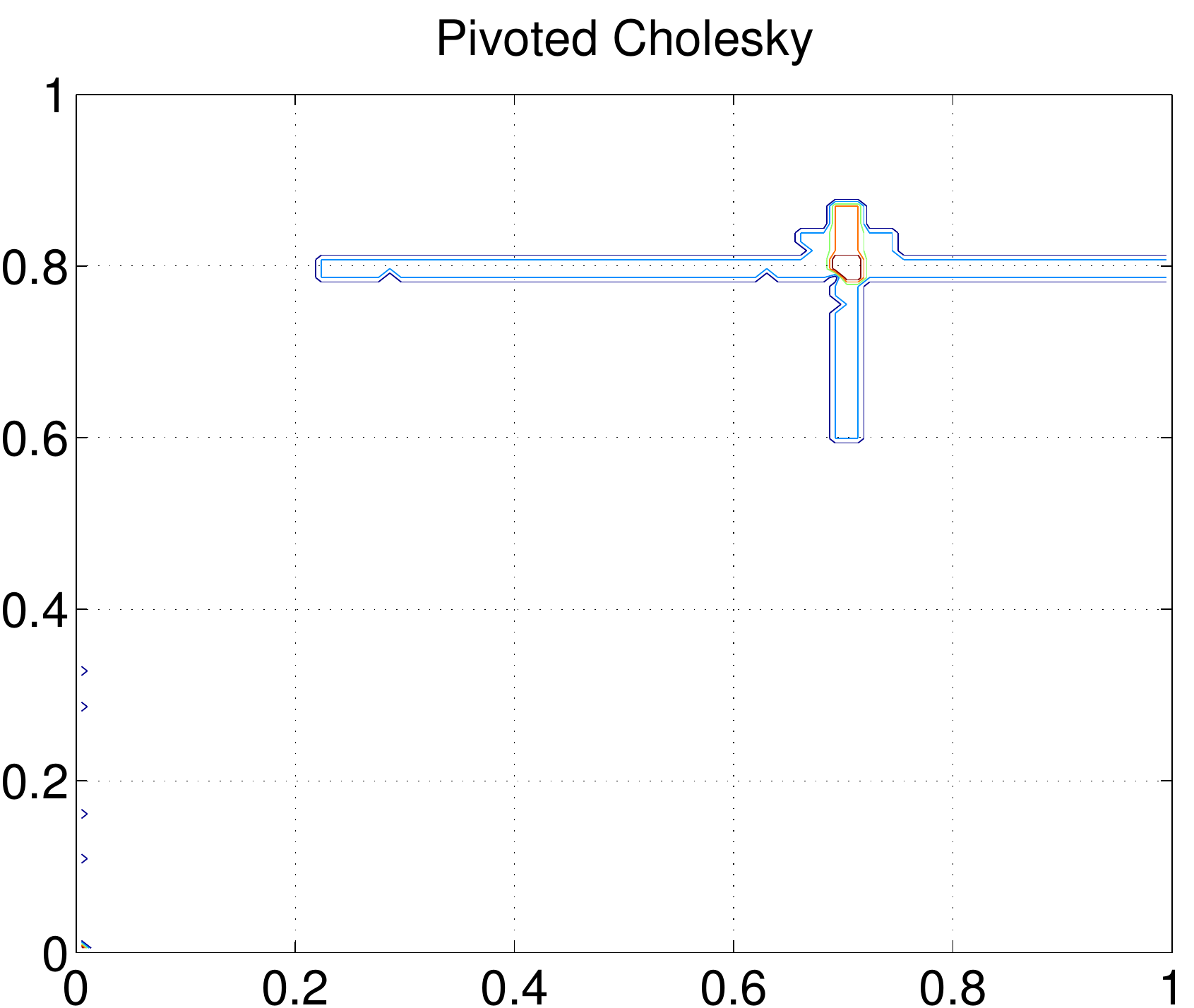}
\includegraphics[width = 0.15\textwidth,height = 0.15\textheight]{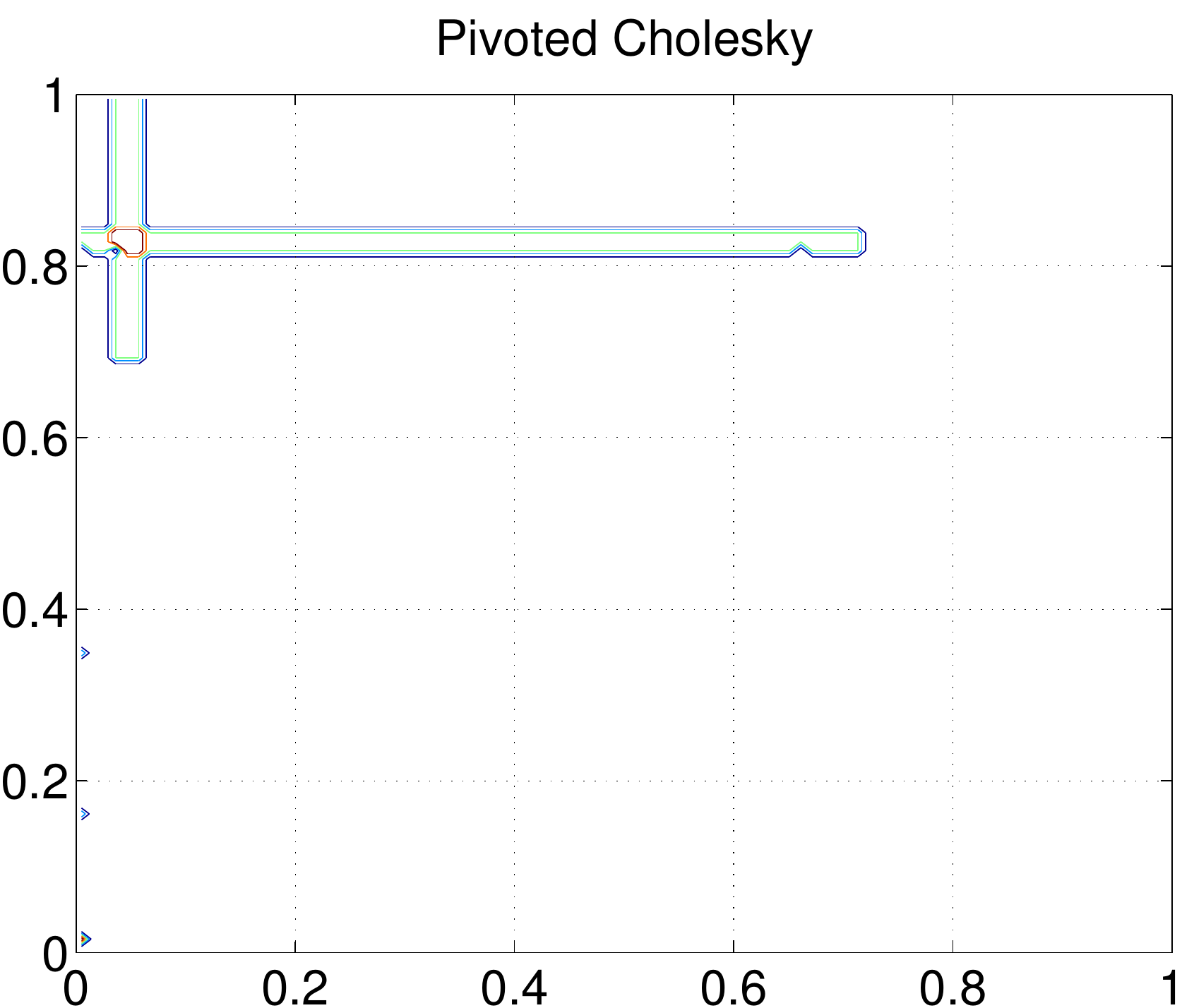}
\includegraphics[width = 0.15\textwidth,height = 0.15\textheight]{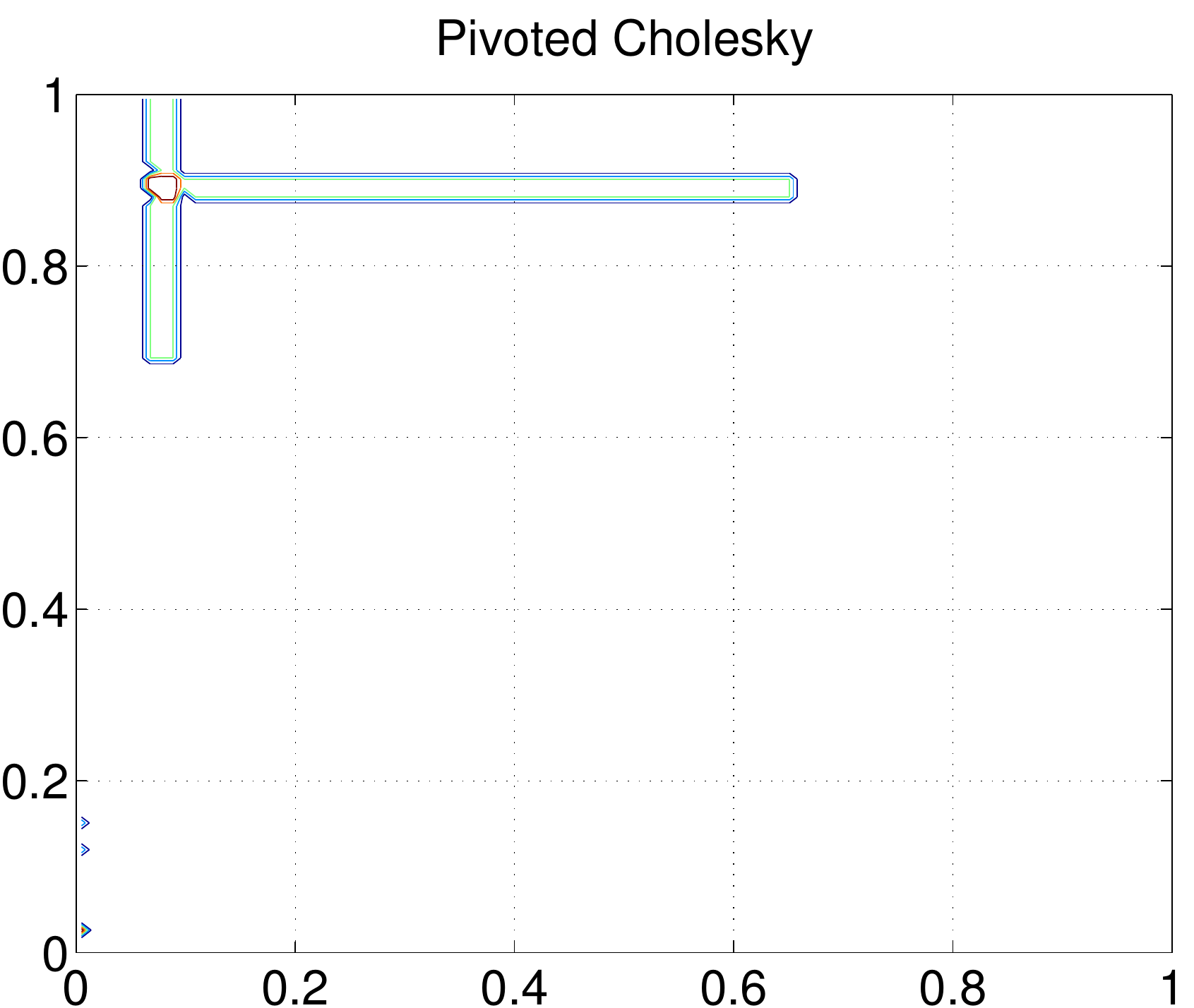}
\includegraphics[width = 0.15\textwidth,height = 0.15\textheight]{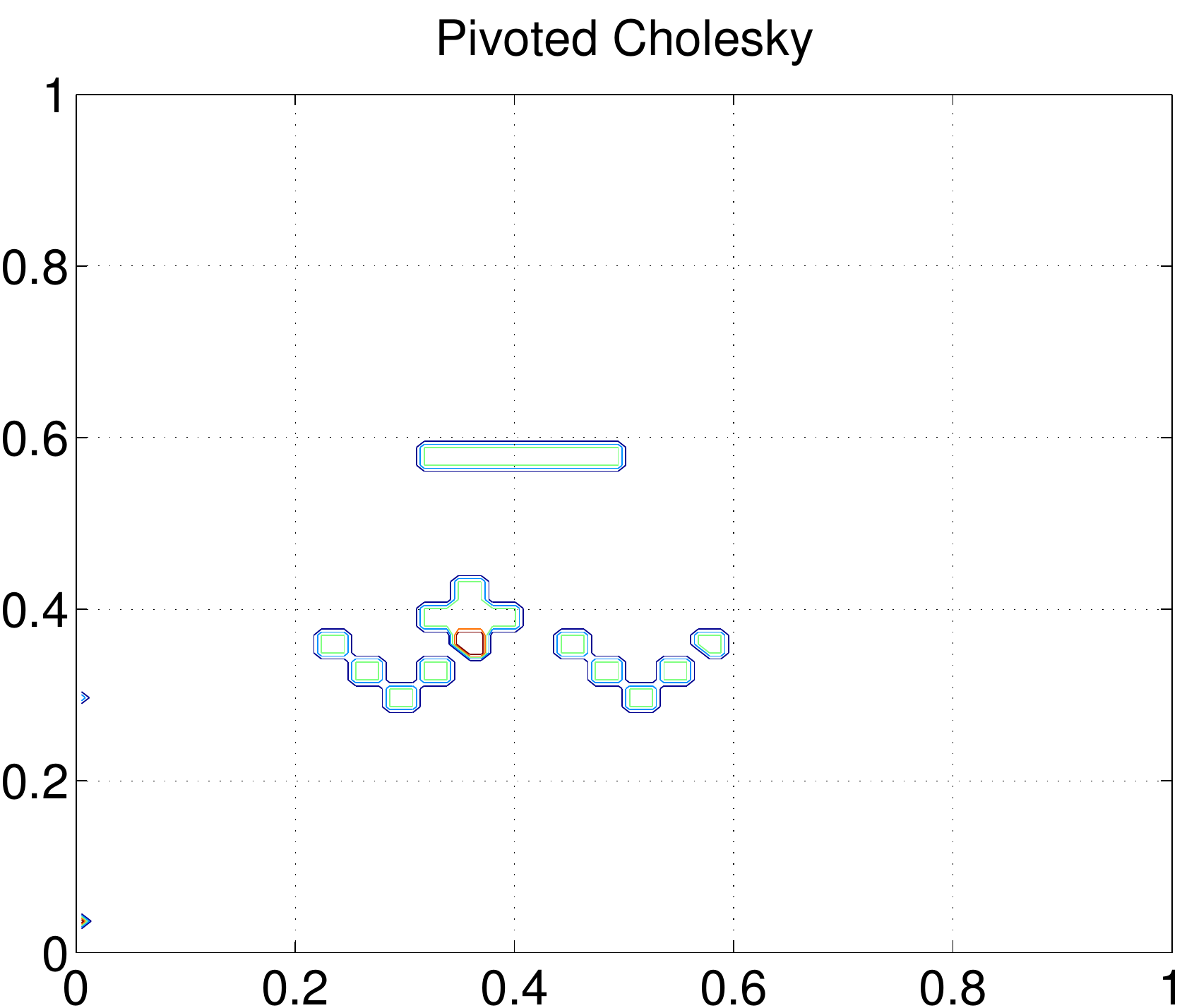}
\includegraphics[width = 0.15\textwidth,height = 0.15\textheight]{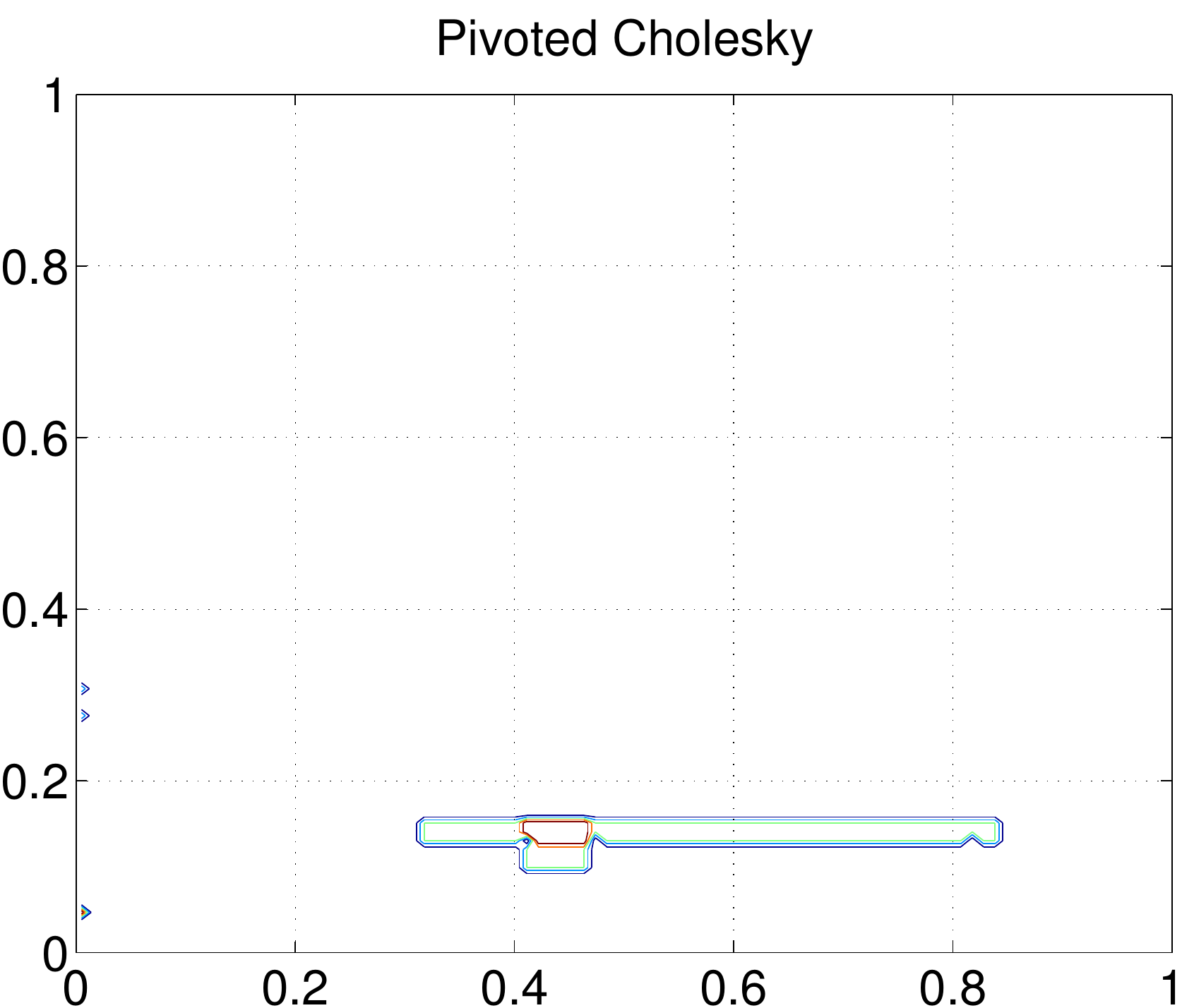} 
\includegraphics[width = 0.15\textwidth,height = 0.15\textheight]{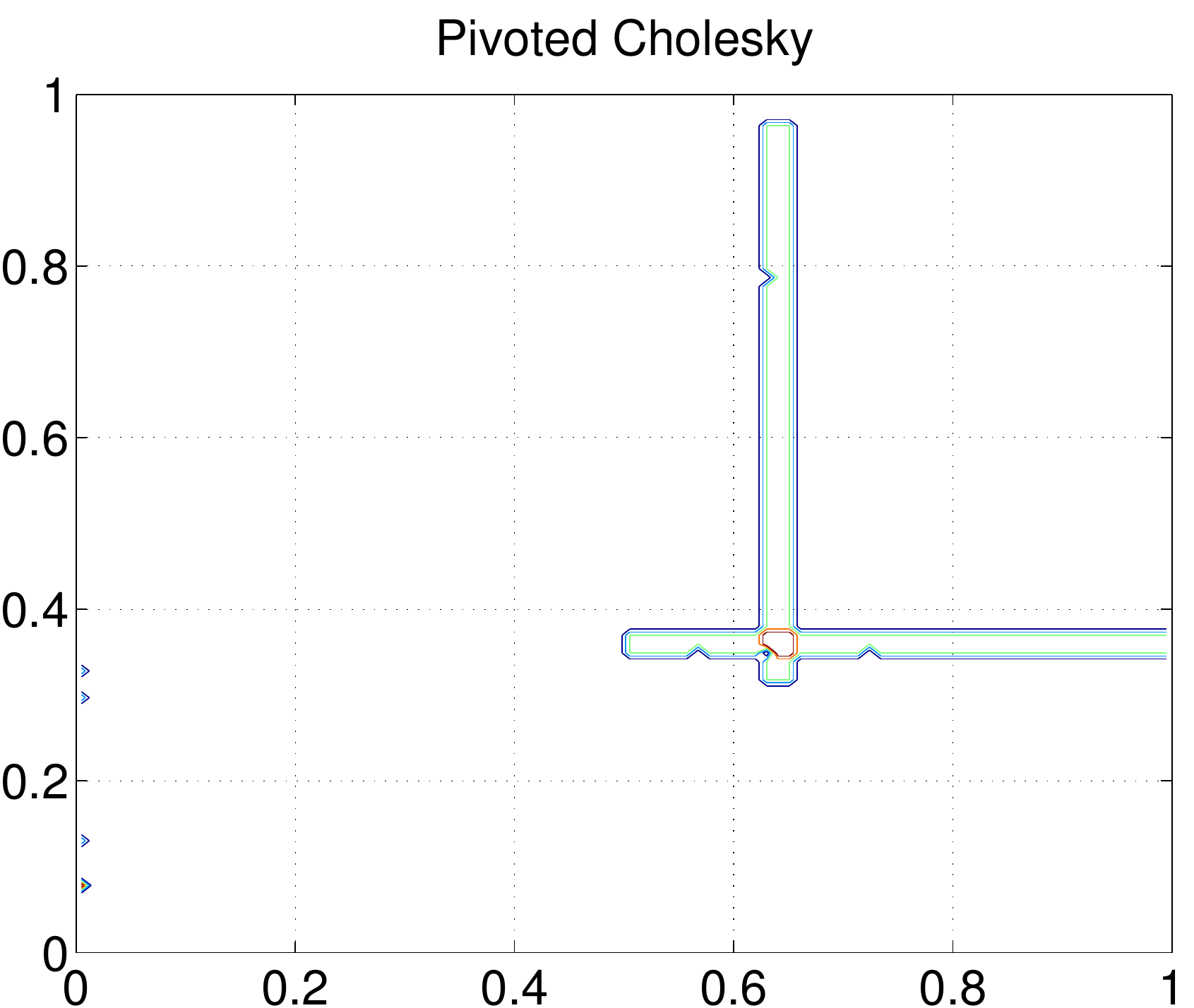}
\caption{First 6 eigenvectors (H=1); First 6 intrinsic sparse modes (H=1/8, regular-sparse); First 6 intrinsic sparse modes (H=1/32; not regular-sparse); First 6 modes from the pivoted Cholesky decomposition of $A$}\label{fig:2d_compare1}
\end{figure}

We use Lemma~\ref{lem:necessaryD} to check when the regular-sparse property fails. It turns out that for $H\ge 1/16$ the regular-sparse property holds and for $H \le 1/24$ it fails. The eigenvalues of $\Lambda$'s for $H = 1, 1/8$ and $1/32$ are plotted in Figure~\ref{fig:2d_compare2} on the left side. The eigenvalues of $\Lambda$ when $H=1$ are all 1's, since every eigenvector has patch-wise sparseness 1 in this trivial case. The eigenvalues of $\Lambda$ when $H=1/16$ are all integers, corresponding to patch-wise sparseness of the intrinsic sparse modes. The eigenvalues of $\Lambda$ when $H=1/32$ are not all integers any more, which indicates that this partition is not regular-sparse with respect to $A$ according to Lemma~\ref{lem:necessaryD}.

\begin{figure}
\centering
\includegraphics[width = 0.45\textwidth]{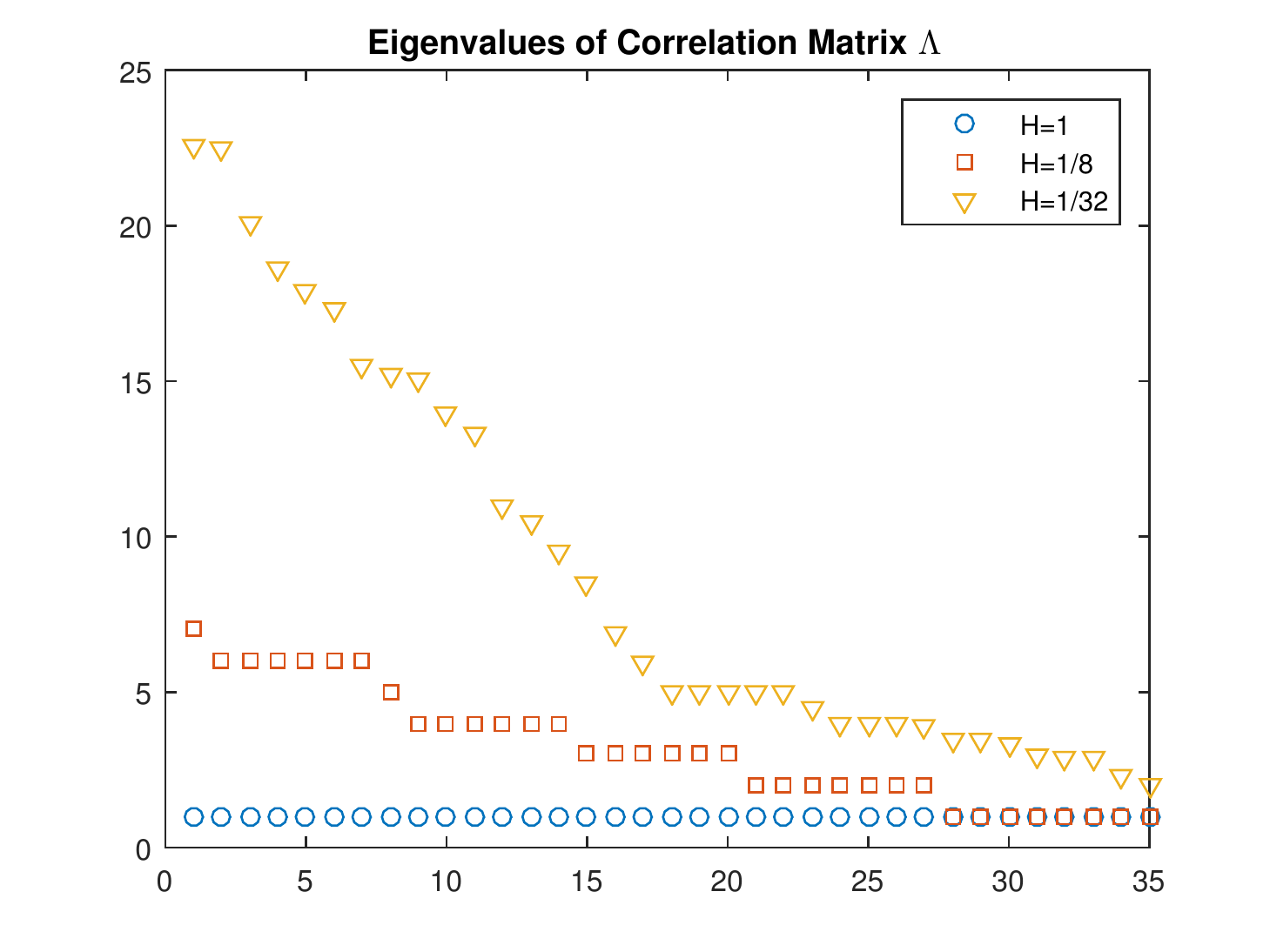}
\includegraphics[width = 0.45\textwidth]{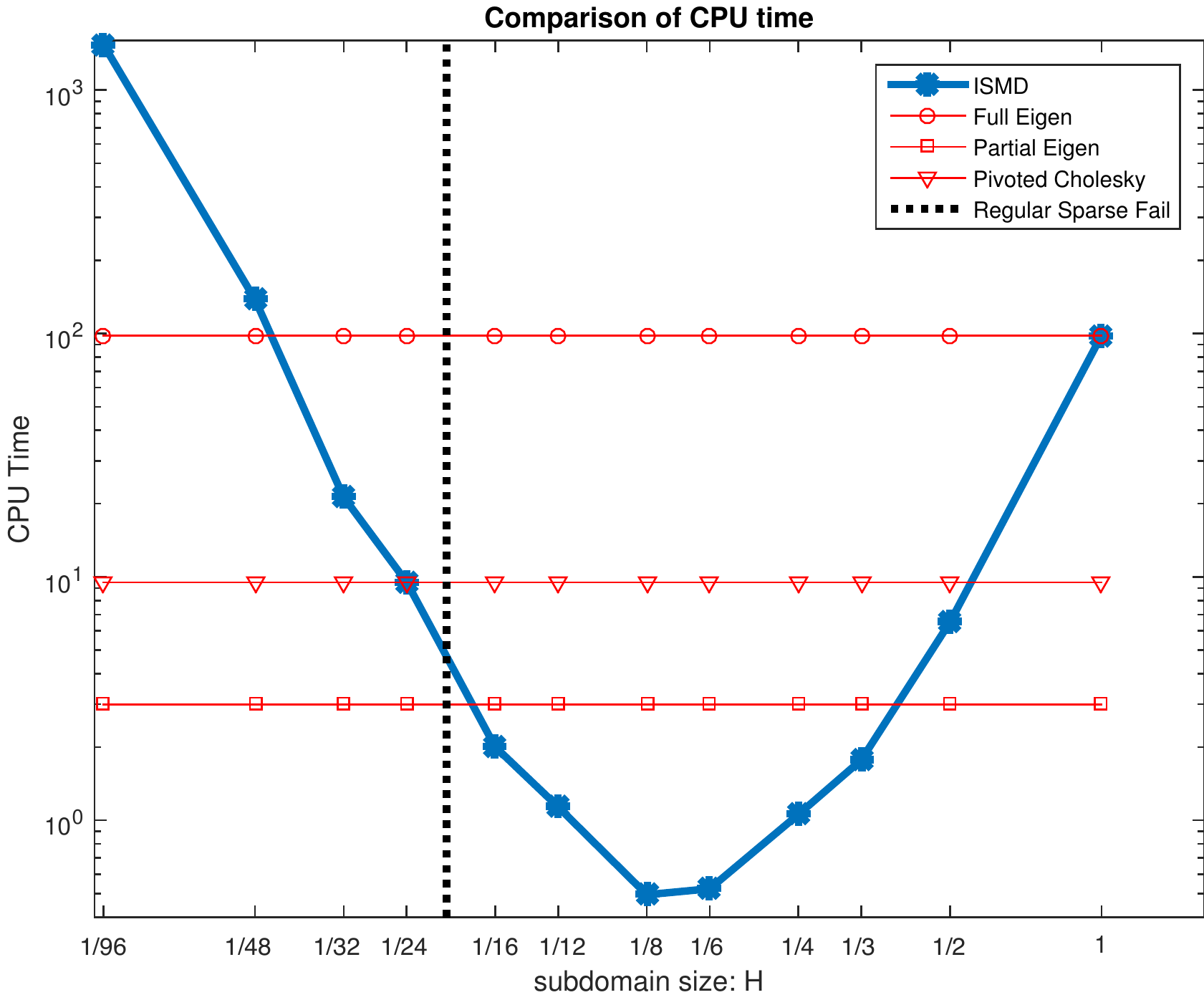}
\caption{Left: Eigen values of $\Lambda$ for $H = 1, 1/8, 1/32$. By Lemma~\ref{lem:necessaryD}, the partition with $H=1/32$ is not regular-sparse. Right: CPU time (unit: second) for different partition sizes $H$.}\label{fig:2d_compare2}
\end{figure}

The consistency of the ISMD (Theorem~\ref{thm:ISMDconsistency}) manifests itself from $H=1$ to $H=1/8$ in Figure~\ref{fig:2d_compare1}. As Theorem~\ref{thm:ISMDconsistency} states, the supports of the intrinsic sparse modes on a coarser partition contain those on a finer partition. In other words, we get sparser modes when we refine the partition as long as the partition is regular-sparse. After checking all the 35 recovered modes, we see that the intrinsic sparse modes get sparser and sparser from $H=1$ to $H=1/6$. When $H \le 1/6$, all the 35 intrinsic sparse modes are identifiable with each other and these intrinsic modes remain the same for $H = 1/8, 1/12, 1/16$. When $H\le 1/24$, the regular-sparse property fails, but we still get the sparsest decomposition (the same decomposition with $H=1/8$). For $H=1/32$, we exactly recover 33 intrinsic sparse modes but get the other two mixed together. This is not surprising since the partition is not regular-sparse any more. For $H=1/48$, we exactly recover all the 35 intrinsic sparse modes again. Table~\ref{tab:exactrecover} lists the cases when we exactly recover the sparse decomposition~\eqref{eqn:example2dsparseA} from which we construct $A$. From Theorem~\ref{thm:ISMD}, this decomposition is the optimal sparse decomposition (defined by problem~\eqref{opt:minsparseness}) for $H \ge 1/16$. We suspect that this decomposition is also optimal in the $L^0$ sense (defined by problem~\eqref{opt:minsupport}).

\begin{table}[ht]
\centering 
\begin{tabular}[ht]{c | c | c | c | c | c | c | c | c | c | c | c | c} 
\hline\hline 
   $H$      & 1 & 1/2 & 1/3 & 1/4 & 1/6 & 1/8 & 1/12 & 1/16 & 1/24 & 1/32 & 1/48 & 1/96 \\ [0.5ex] 
\hline 
regular-sparse & \ding{52}  & \ding{52}  & \ding{52}  & \ding{52}  & \ding{52}  & \ding{52} & \ding{52}  & \ding{52}  & \ding{55}  & \ding{55}  & \ding{55}  & \ding{55}  \\
Exact Recovery & \ding{55}  & \ding{55}  & \ding{55}  & \ding{55}  & \ding{52}  & \ding{52} & \ding{52}  & \ding{52}  & \ding{52}  & \ding{55}  & \ding{52}  & \ding{55} \\
\hline 
\end{tabular}
\caption{Cases when the ISMD gets exact recovery of the sparse decomposition~\eqref{eqn:example2dsparseA}} 
\label{tab:exactrecover} 
\end{table}

The CPU time of the ISMD for different $H$'s is showed in Figure~\ref{fig:2d_compare2} on the right side. We compare the CPU time for the full eigen decomposition \texttt{eig(A)}, the partial eigen decomposition \texttt{eigs(A, 35)}, and the pivoted Cholesky decomposition. For $1/16 \le H \le 1/3$, the ISMD is even faster than the partial eigen decomposition. Specifically, the ISMD is ten times faster for the case $H=1/8$. Notice that the ISMD performs the local eigen decomposition by \texttt{eig} in Matlab, and thus does not need any prior information about the rank $K$. If we also assume prior information on the local rank $K_m$, the ISMD would be even faster. The CPU time curve has a V-shape as predicted by our computational estimation~\eqref{eqn:cost3}. The cost first decreases as we refine the mesh because the cost of local eigen decompositions decreases. Then it increases as we refine further because there are $M$ joint diagonalization problem~\eqref{opt:jointDm} to be solved. When $M$ is very large, i.e., $H=1/48$ or $H=1/96$, the 2 layer for-loops from Line 5 to Line 10 in Algorithm~\ref{alg:ISMD} become extremely slow in Matlab. When implemented in other languages that have little overhead cost for multiple for-loops, e.g. C or C++, the actual CPU time for $H=1/96$ would be roughly the same with the CPU time for the pivoted Cholesky decomposition.

\subsection{Comparison with the semi-definite relaxation of sparse PCA}\label{subsec:compareL1ISMD}
In comparison, the semi-definite relaxation of sparse PCA (Problem \eqref{eqn:eigvariationalL1_relax}) gives poor results in this example. We have tested several values of $\mu$, and found that parameter $\mu = 0.0278$ gives the best performance in the sense that the first 35 eigenvectors of $W$ capture most variance in $A$. The first 35 eigenvectors of $W$, shown in Figure~\ref{fig:Weigenvectors}, explain $95\%$ of the variance, but all of them mix several intrinsic modes like what the eigen decomposition does in Figure~\ref{fig:2d_compare1}. For this example, it is not clear how to choose the best 35 columns out of all the 9216 columns in $W$, as proposed in \cite{LaiLuOsher_2014}. If columns of $W$ are ordered by $l^2$ norm in descending order, the first 35 columns can only explain $31.46\%$ of the total variance, although they are indeed localized. Figure~\ref{fig:Wcolumns} shows the first 6 columns of $W$ with largest norms.

We also compare the CPU time of the ISMD with that of the semi-definite relaxation of sparse PCA~\eqref{eqn:eigvariationalL1_relax}. The sparse PCA is computed using the split Bregman iteration. Each split Bregman iteration requires an eigen-decomposition of a matrix of size $N\times N$. In comparison, the ISMD is cheaper than a single eigen-decomposition, as shown in Figure~\ref{fig:2d_compare2}. It has been observed that the split Bregman iteration converges linearly. If we set the error tolerance to be $O(\delta)$, the number of iterations needed is about $\Or(1/\delta)$. In our implementation, we set the error tolerance to be $10^{-3}$ and we need to perform 852 iterations. Overall, to solve the convex optimization problem~\eqref{eqn:eigvariationalL1_relax} with split Bregman iteration takes over 1000 times more CPU time than the ISMD with $H=1/8$. 

It is expected that the ISMD is much faster than sparse PCA since the sparse PCA needs to perform many times of partial eigen decomposition to solve problem~\eqref{eqn:eigvariationalL1_relax}, but the ISMD has computational cost comparable to one single partial eigen decomposition. As we discussed in Section~\ref{subsec:connections}, sparse PCA is designed and works reasonably well for problem~\eqref{eqn:matrix_factorize}. When sparse PCA is applied to our sparse decomposition problem~\eqref{opt:minsparseness}, it does not work well. However, it is not always the case that the ISMD gives a sparser and more accurate decomposition of $A$ than sparse PCA. In subsection~\ref{sec:exponential}, we will present another example in which sparse PCA gives a better performance than the ISMD.

\begin{figure}[htbp]
\begin{center}
\includegraphics[width = 0.15\textwidth,height = 0.15\textheight]{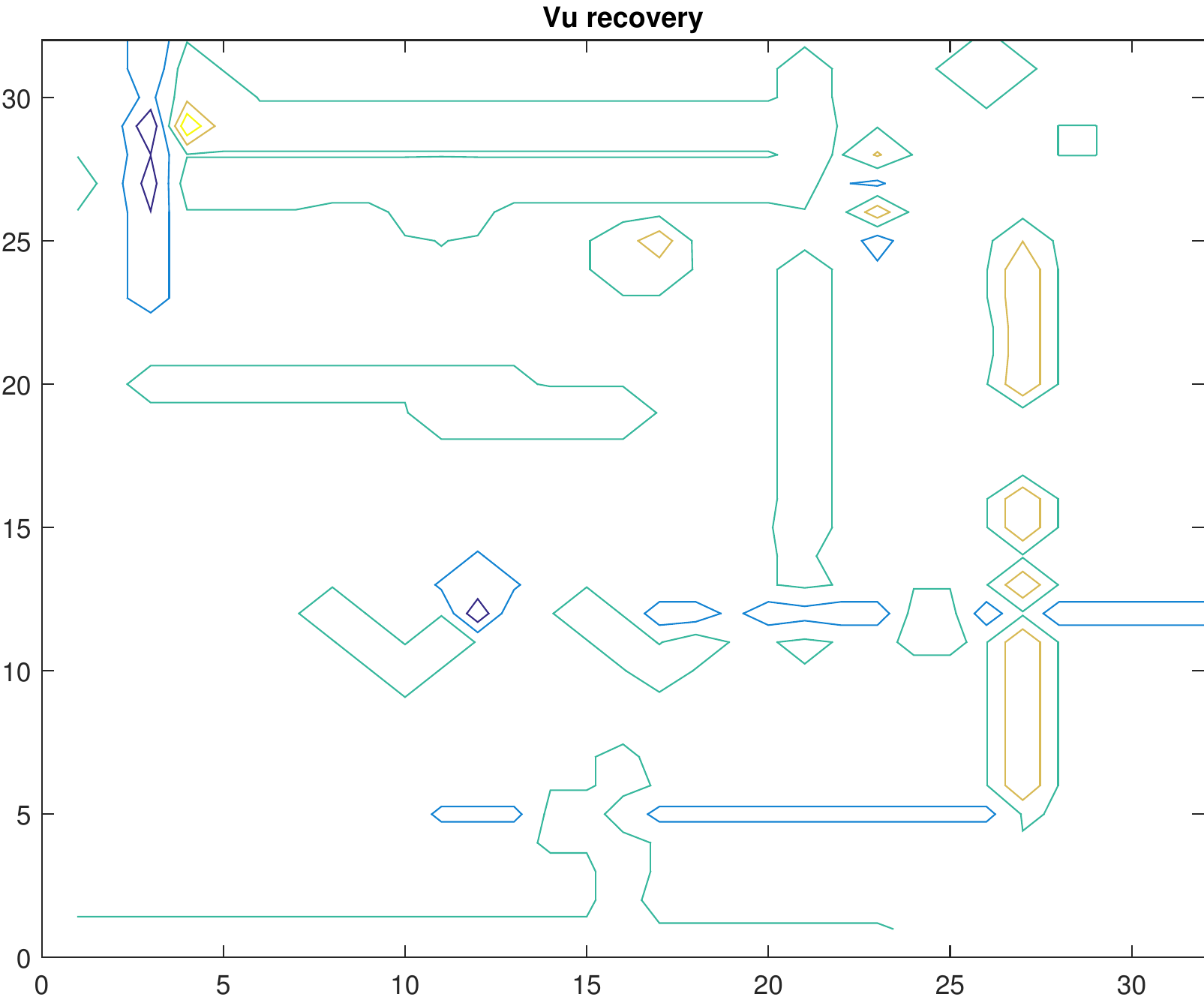}
\includegraphics[width = 0.15\textwidth,height = 0.15\textheight]{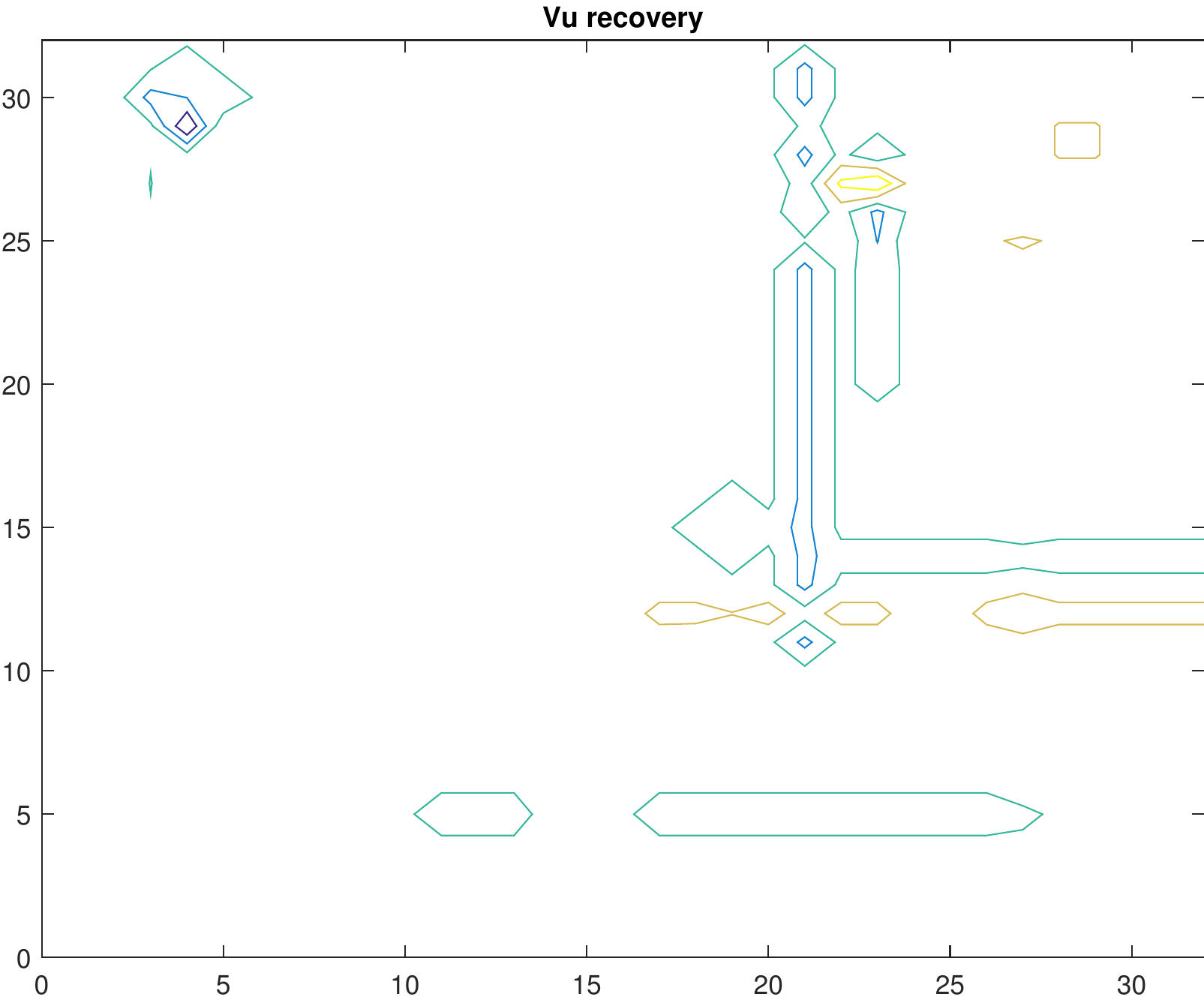}
\includegraphics[width = 0.15\textwidth,height = 0.15\textheight]{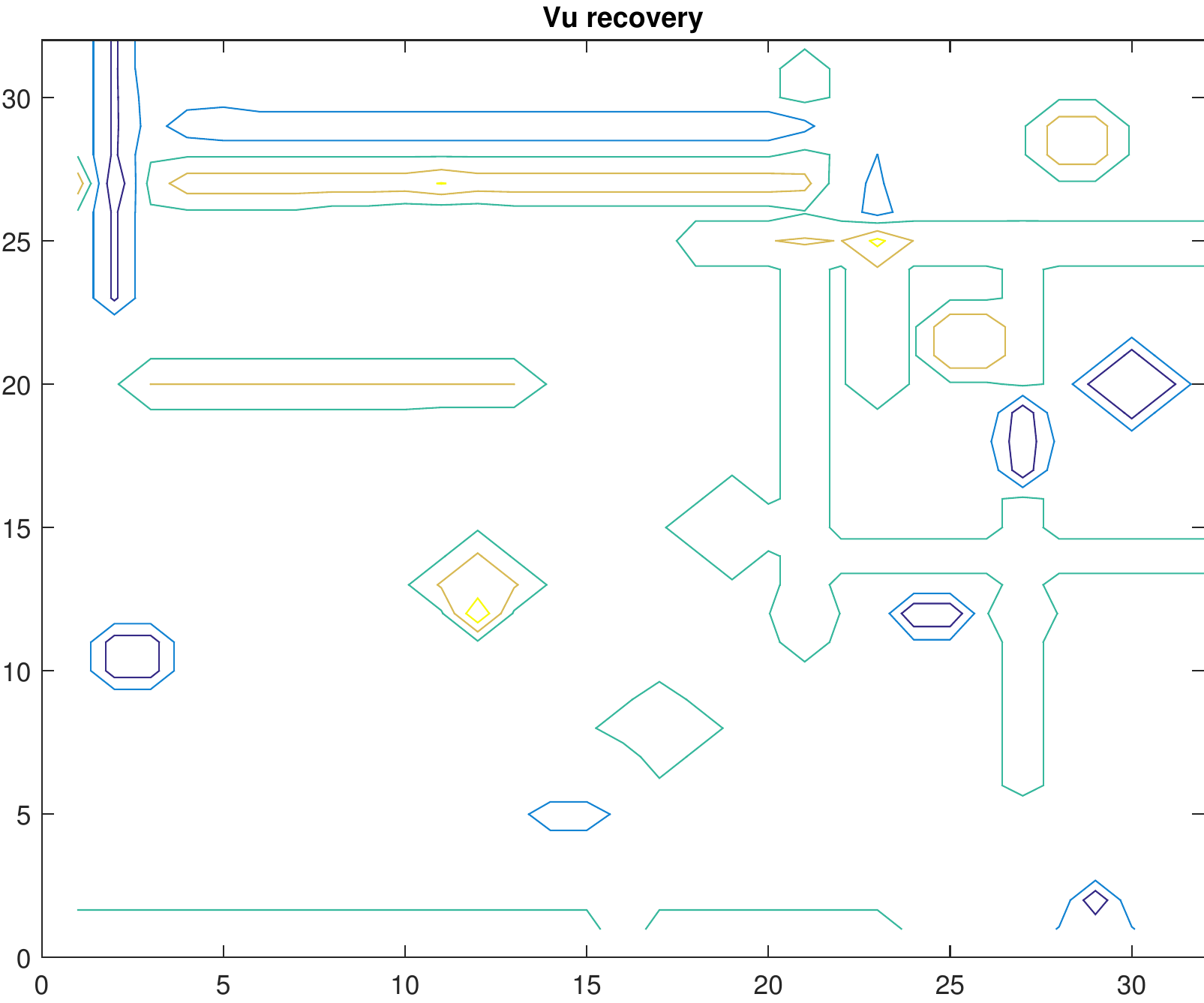}
\includegraphics[width = 0.15\textwidth,height = 0.15\textheight]{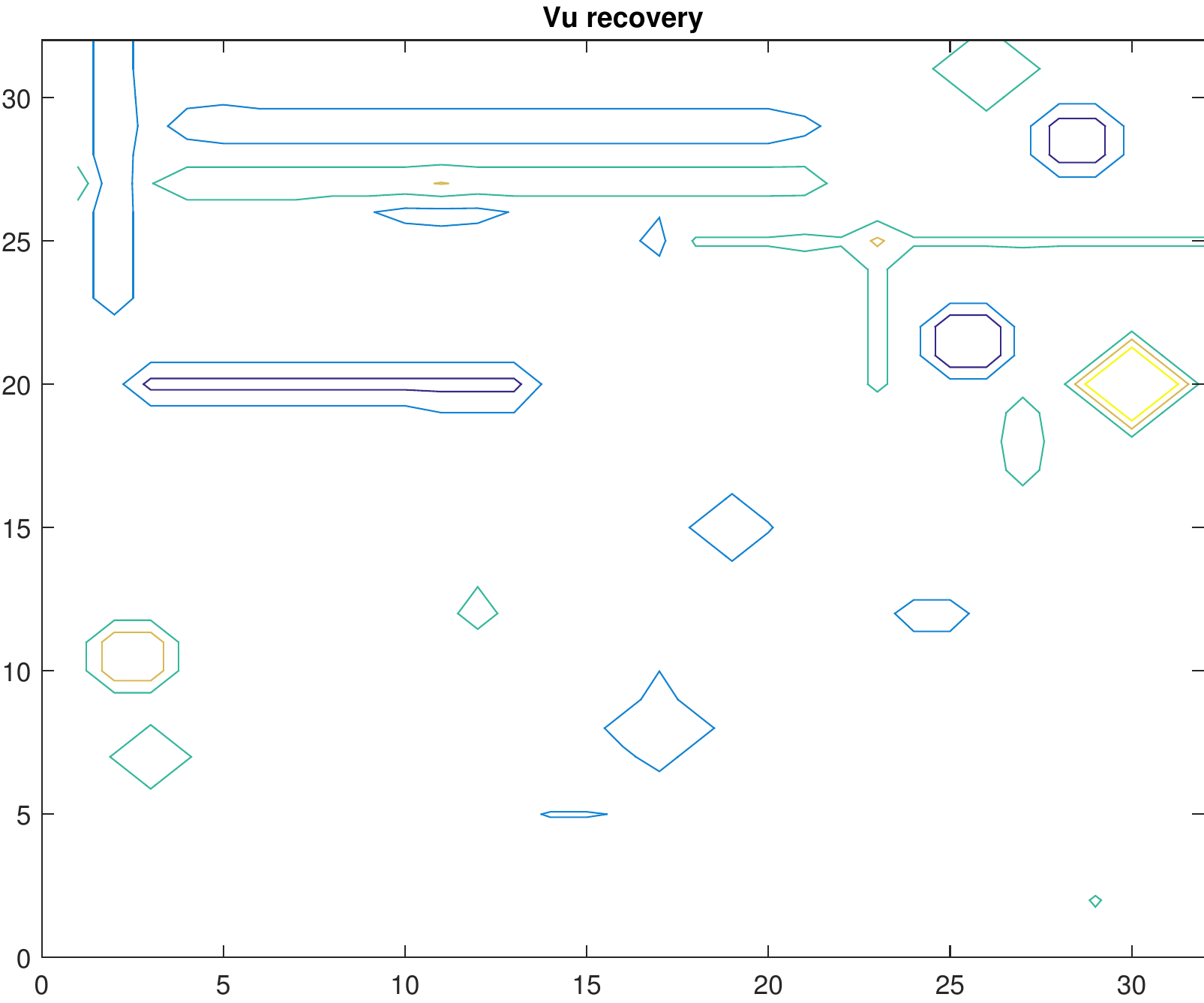}
\includegraphics[width = 0.15\textwidth,height = 0.15\textheight]{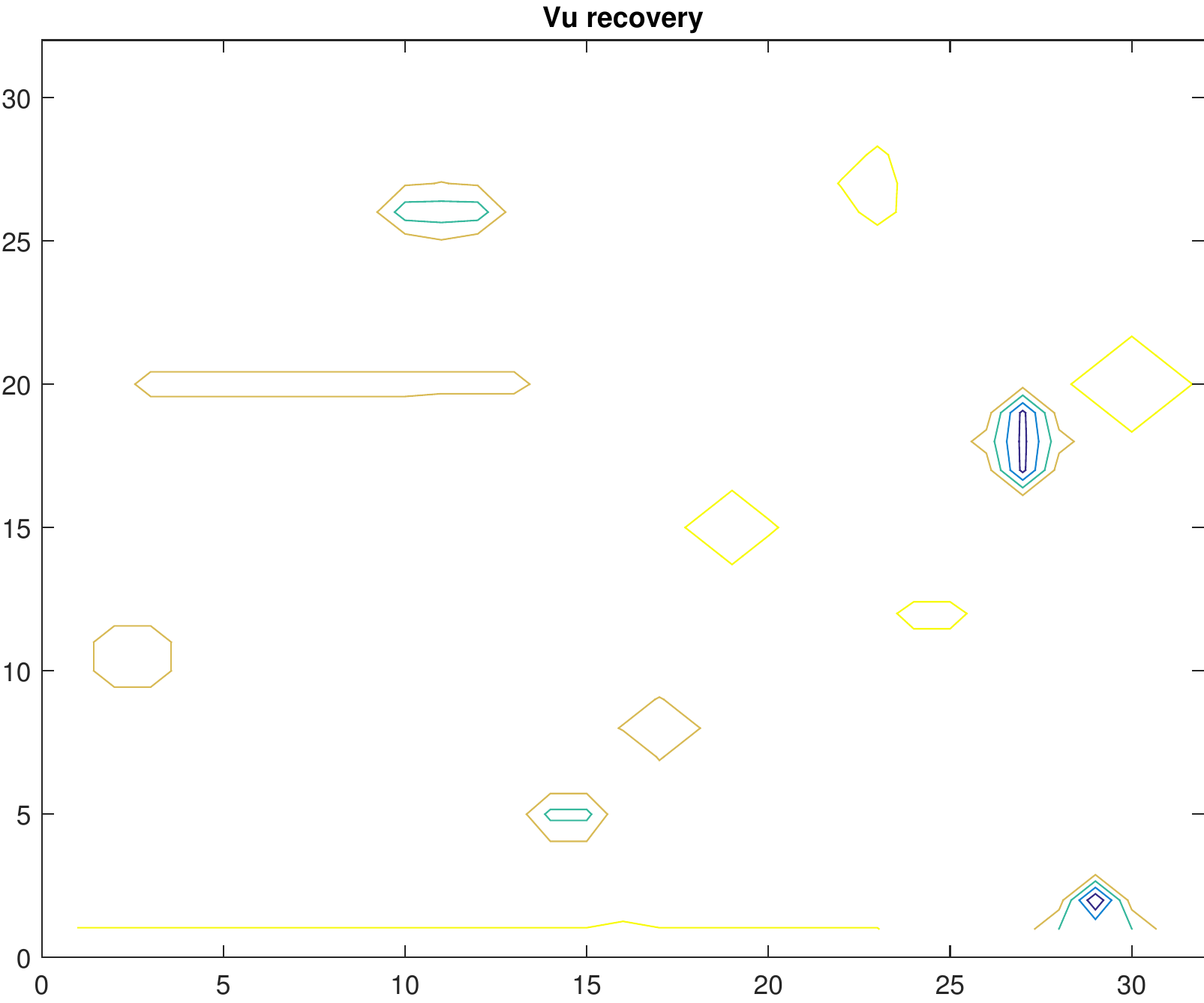}
\includegraphics[width = 0.15\textwidth,height = 0.15\textheight]{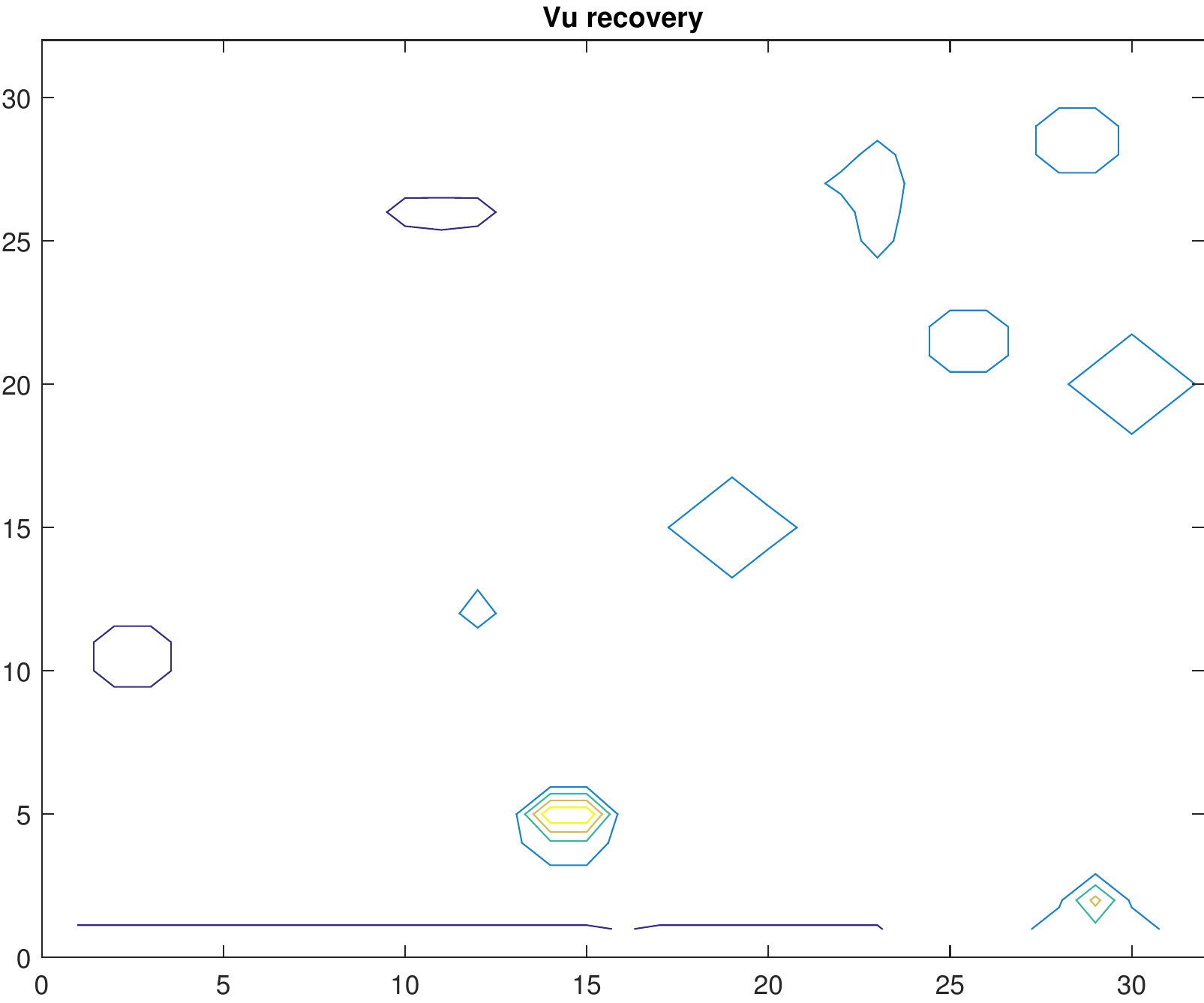}
\caption{Sparse PCA: The first 6 eigenvectors of $W$. The first 35 eigenvectors of $W$ explain $95\%$ of the variance.} \label{fig:Weigenvectors}
\end{center}
\end{figure}

\begin{figure}[htbp]
\begin{center}
\includegraphics[width = 0.15\textwidth,height = 0.15\textheight]{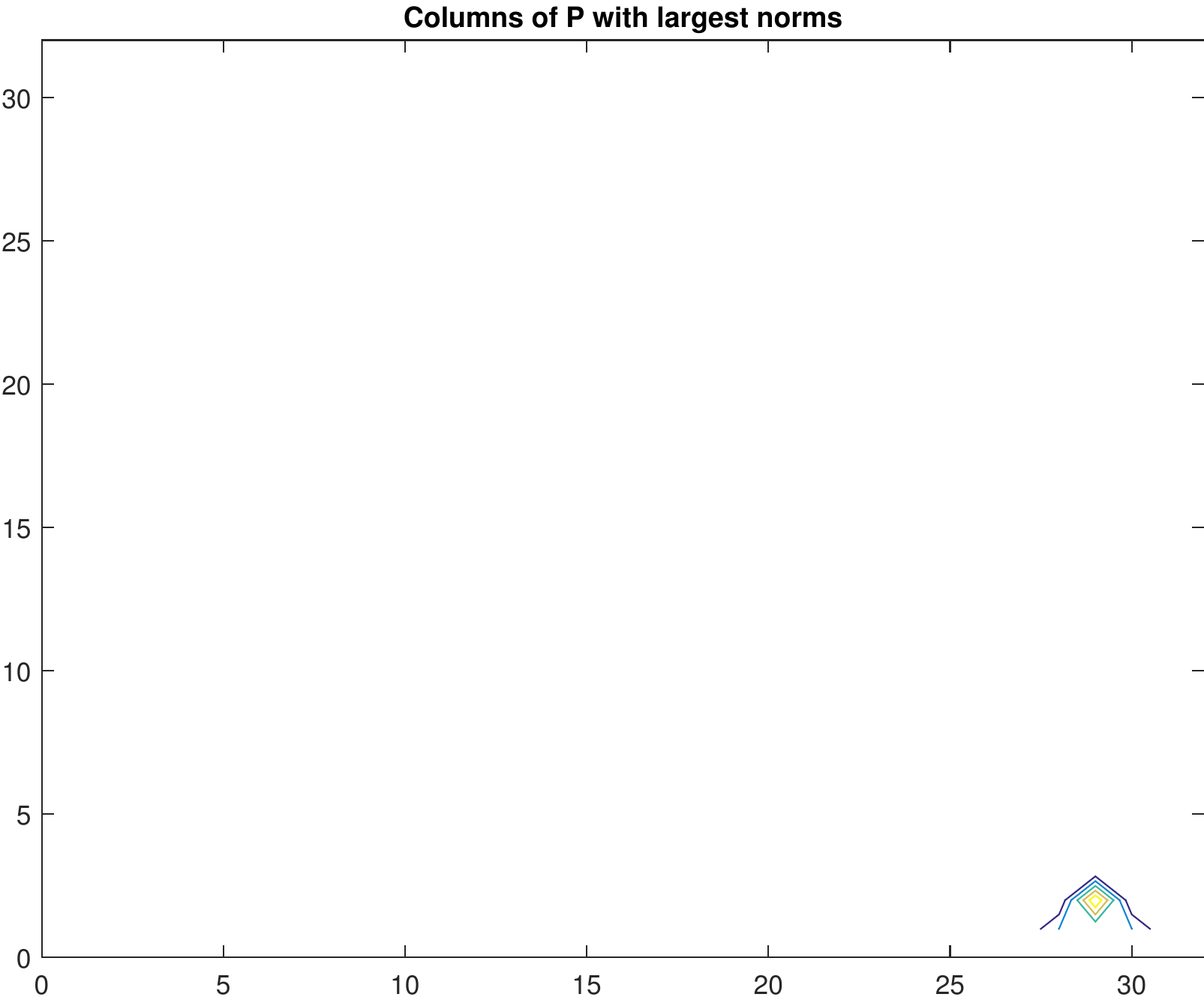}
\includegraphics[width = 0.15\textwidth,height = 0.15\textheight]{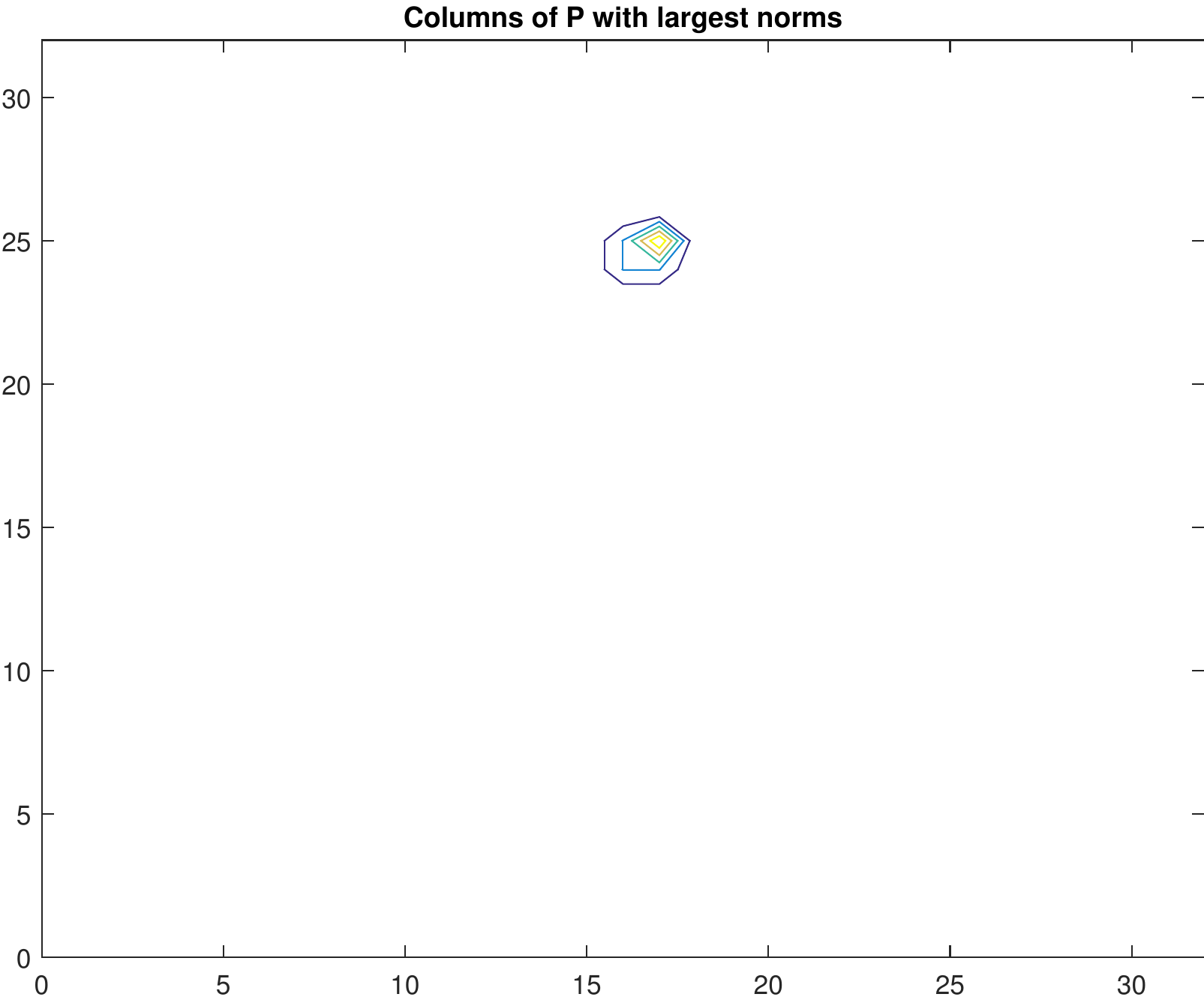}
\includegraphics[width = 0.15\textwidth,height = 0.15\textheight]{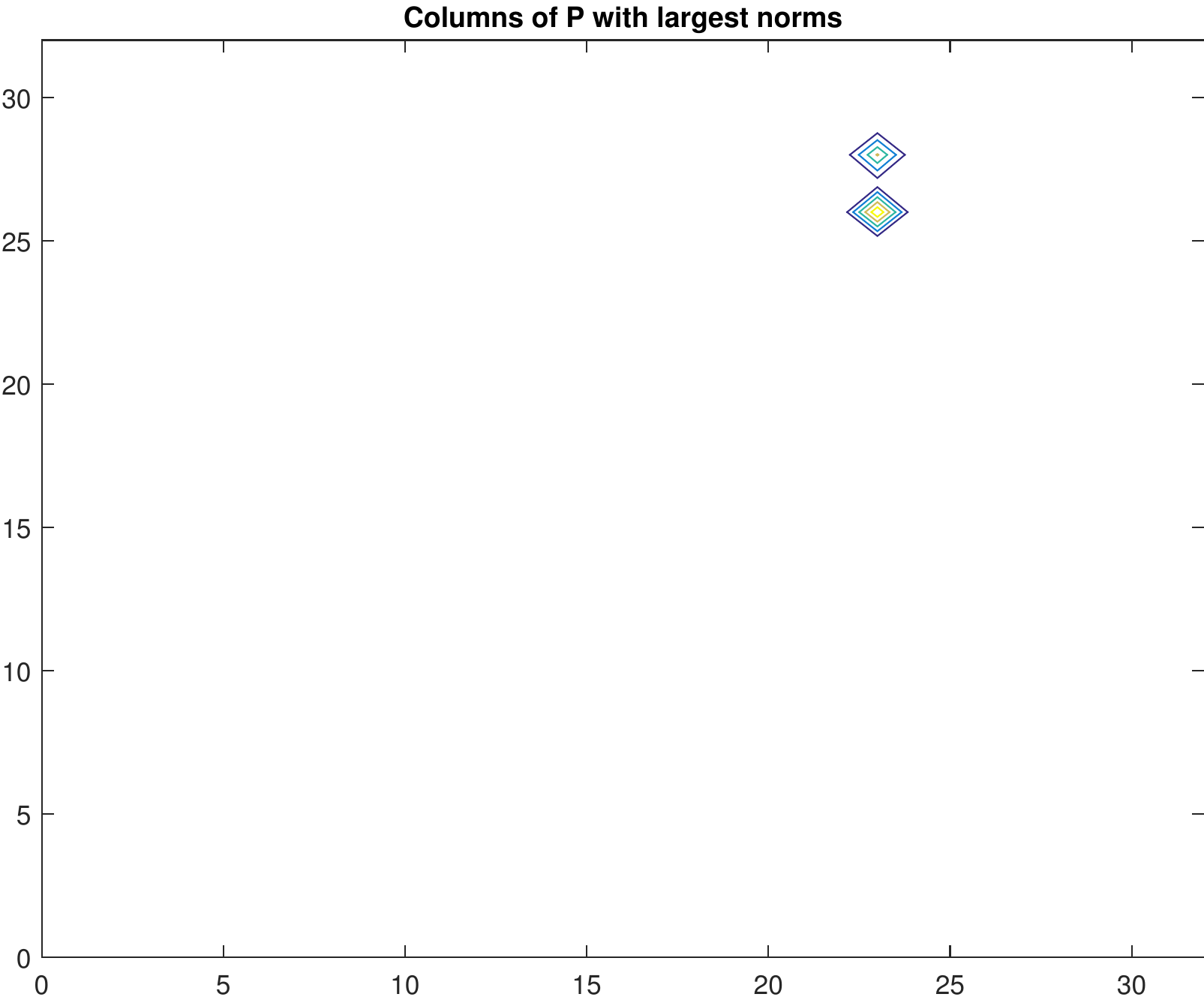}
\includegraphics[width = 0.15\textwidth,height = 0.15\textheight]{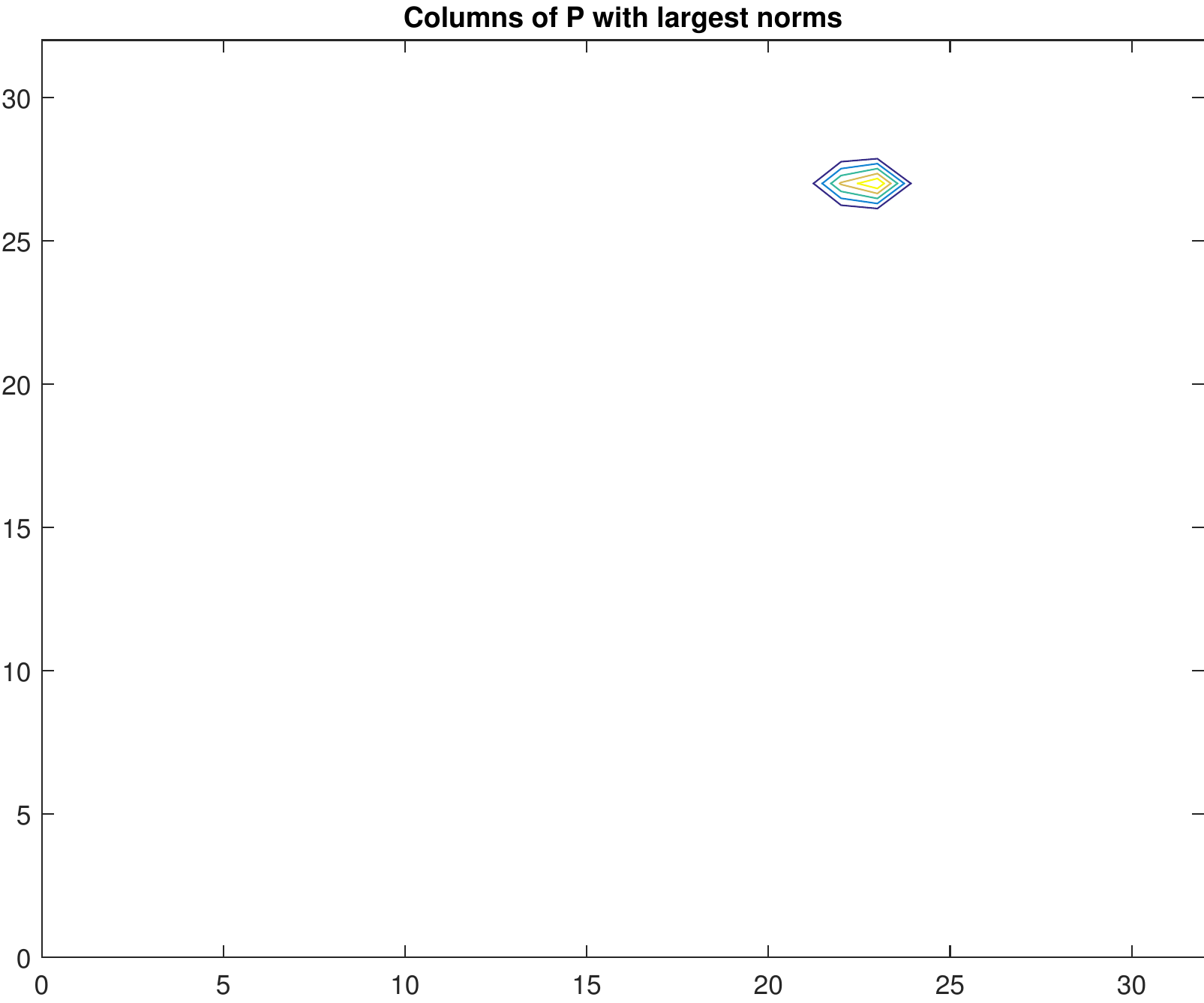}
\includegraphics[width = 0.15\textwidth,height = 0.15\textheight]{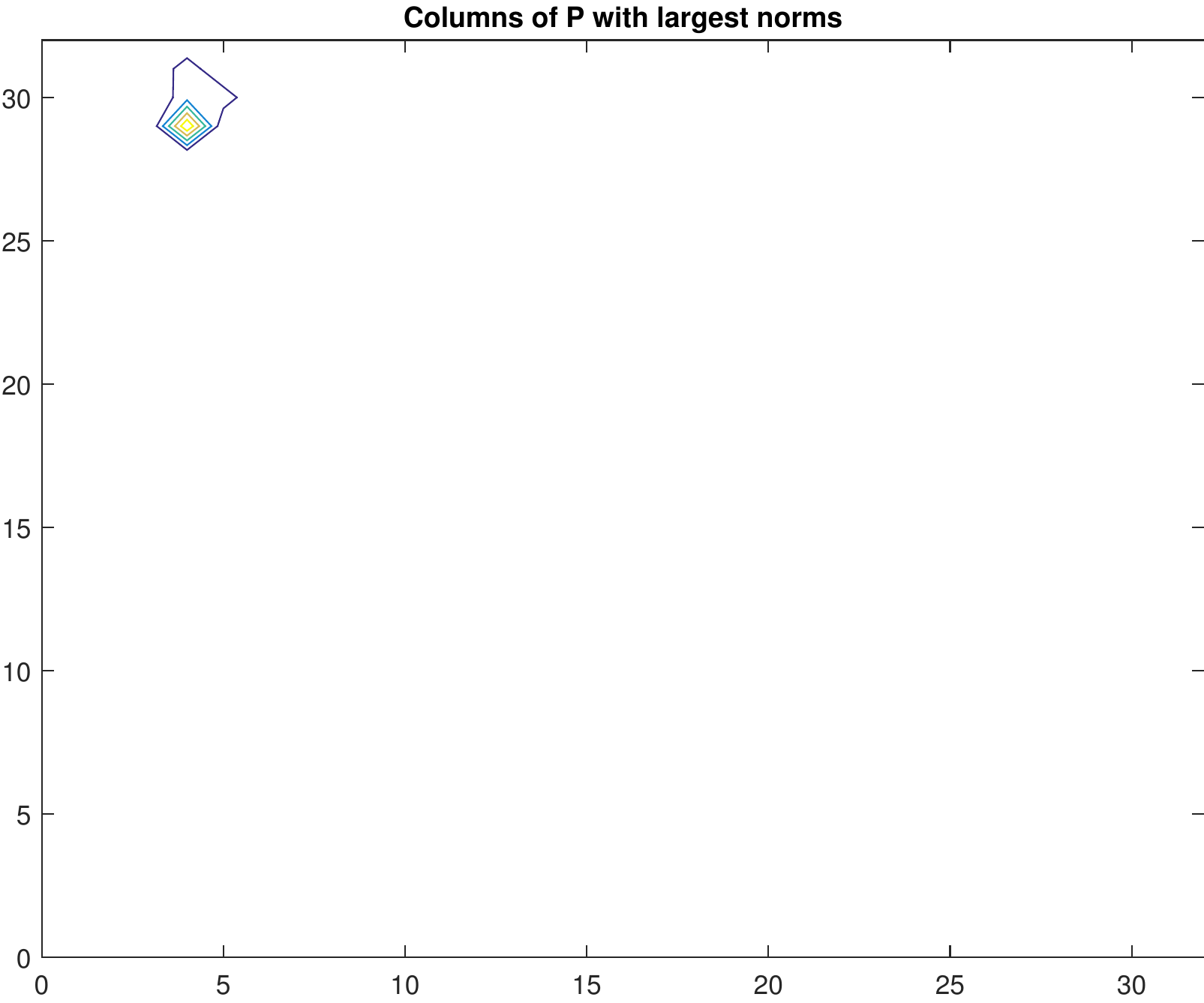}
\includegraphics[width = 0.15\textwidth,height = 0.15\textheight]{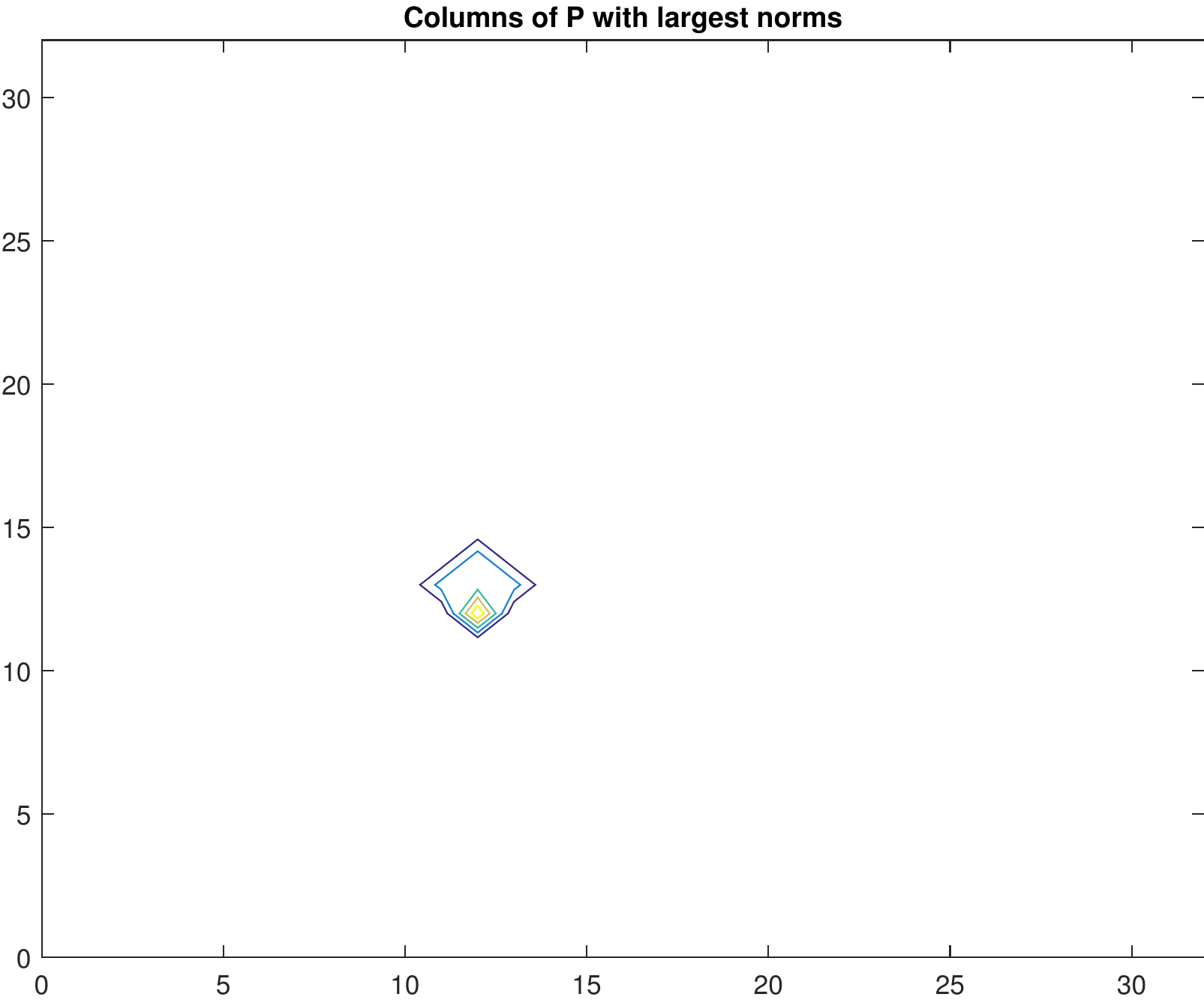}
\caption{Sparse PCA: 6 columns of $W$ with largest norms. The first 35 columns with largest norms only explain $31.46\%$ of the variance.} \label{fig:Wcolumns}
\end{center}
\end{figure}

We point out that unlike the structured sparse PCA~\cite{jenatton2010structured}, the ISMD does not take advantage of the specific (rectangular) structure of the physical modes. The ``smiling face" mode shows that the ISMD can recover non-convex and non-local sparse modes. Therefore, the ISMD is expected to perform equally well even when there is no such structures known.

\subsection{ISMD with small noises}\label{exam:noise1}
In this subsection we report the test on the robustness of the ISMD. In the following test, we perturb the rank-35 covariance matrix $A \in \RR^{9216\times 9216}$ with a random matrix:
\begin{equation*}
\hat{A} = A + \epsilon \tilde{A}\,,
\end{equation*}
where $\epsilon$ is the noise level and $\tilde{A} $ is a random matrix with i.i.d. elements uniformly distributed in $[-1,1]$. Notice that all elements in $A$ are uniformly bounded by 1, and thus $\epsilon$ is a relative noise level. Since all the intrinsic sparse modes are identifiable with each other for the partition with patch size $H=1/16$, we perform ISMD with simple thresholding~\eqref{eqn:threshold1} on $\hat{A}$ to get the perturbed intrinsic sparse modes $\hat{G} \equiv [\hat{g}_1, \dots, \hat{g}_K]$. The $l^{\infty}$ and $l^2$ error are defined as below:
\begin{equation*}
Err_{\infty} = \max\limits_{k=1,2,\cdots,K} \frac{\|\hat{g}_k - g_k\|_2}{\|g_k\|_2}, \quad Err_{2} = \sqrt{\sum_{k=1}^K \frac{\|\hat{g}_k - g_k\|_2^2}{\|g_k\|_2^2}}.
\end{equation*}
Figure \ref{fig:stability} shows that $Err_\infty$ and $Err_2$ depend linearly on the noise level $\epsilon$, which validates our stability analysis in Section \ref{sec:perturbation}. 
\begin{figure}
\centering
\includegraphics[width = 0.45\textwidth,height = 0.2\textheight]{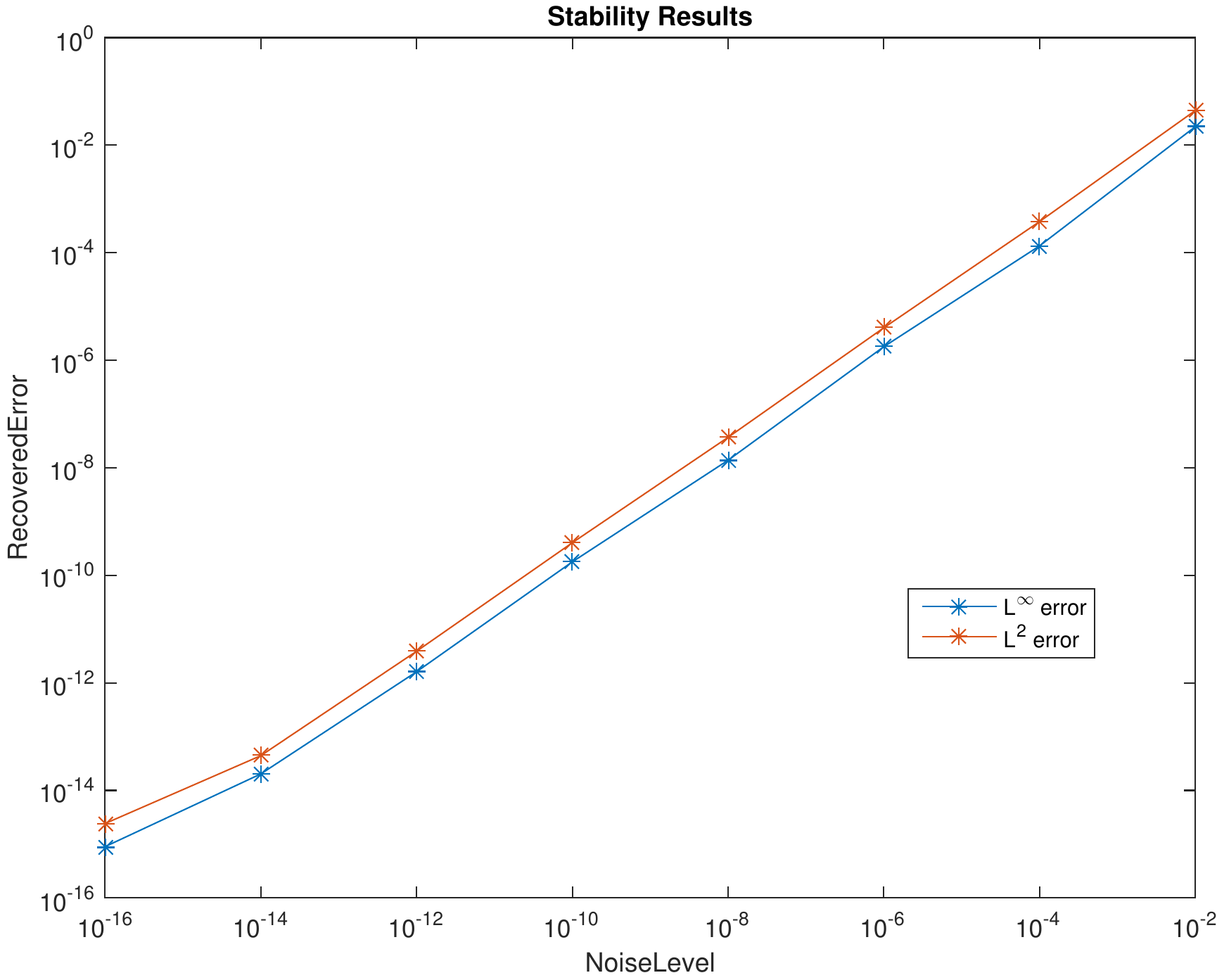}
\caption{$L^{\infty}$ and $L^2$ error increases linearly as the noise level increases.}\label{fig:stability}
\end{figure}

\subsection{Separate global and localized modes with ISMD}
In this example, we consider a more sophisticated model in which the media contain several global modes, i.e.,
\begin{equation} \label{eqn:example2dglobal}
\kappa(x,\omega) = \sum_{k = 1}^{K_1}\xi_k(\omega) f_k(x) + \sum_{k = 1}^{K_2}\eta_k(\omega) g_k(x), \quad x \in [0,1]^2,
\end{equation}
where $\{g_k\}_{k=1}^{K_2}$ and $\eta_k$ models the localized features like channels and inclusions as above, $\{f_k\}_{k=1}^{K_1}$ are functions with support on the entire domain $D=[0,1]^2$ and $\xi_k$ are the associated latent variables with global influence on the entire domain. Here, we keep the $35$ localized features as before, but add $2$ two global features with $f_1(x) = \sin(2 \pi x_1 + 4 \pi x_2)/2$, $f_2(x) = \sin(4 \pi x_1 + 2 \pi x_2)/2$. $\xi_1$ and $\xi_2$ are set to be uncorrelated and have variance 1. For this random medium, the covariance function is
\begin{equation} \label{eqn:example2dglobalA}
	a(x,y) = \sum_{k = 1}^{K_1} f_k(x) f_k(y) + \sum_{k = 1}^{K_2} g_k(x) g_k(y), \quad x, y \in [0,1]^2.
\end{equation}
As before, we discretize the covariance function with $h_x = h_y = 1/96$ and represent $A$ by a matrix of size $9216 \times 9216$. One sample of the random field (and the bird's-eye view) and the covariance matrix are plotted in Figure~\ref{fig:2dglobal}. It can be seen that the covariance matrix is dense now because we have two global modes.
\begin{figure}[h]
\centering
\includegraphics[width = 0.30\textwidth]{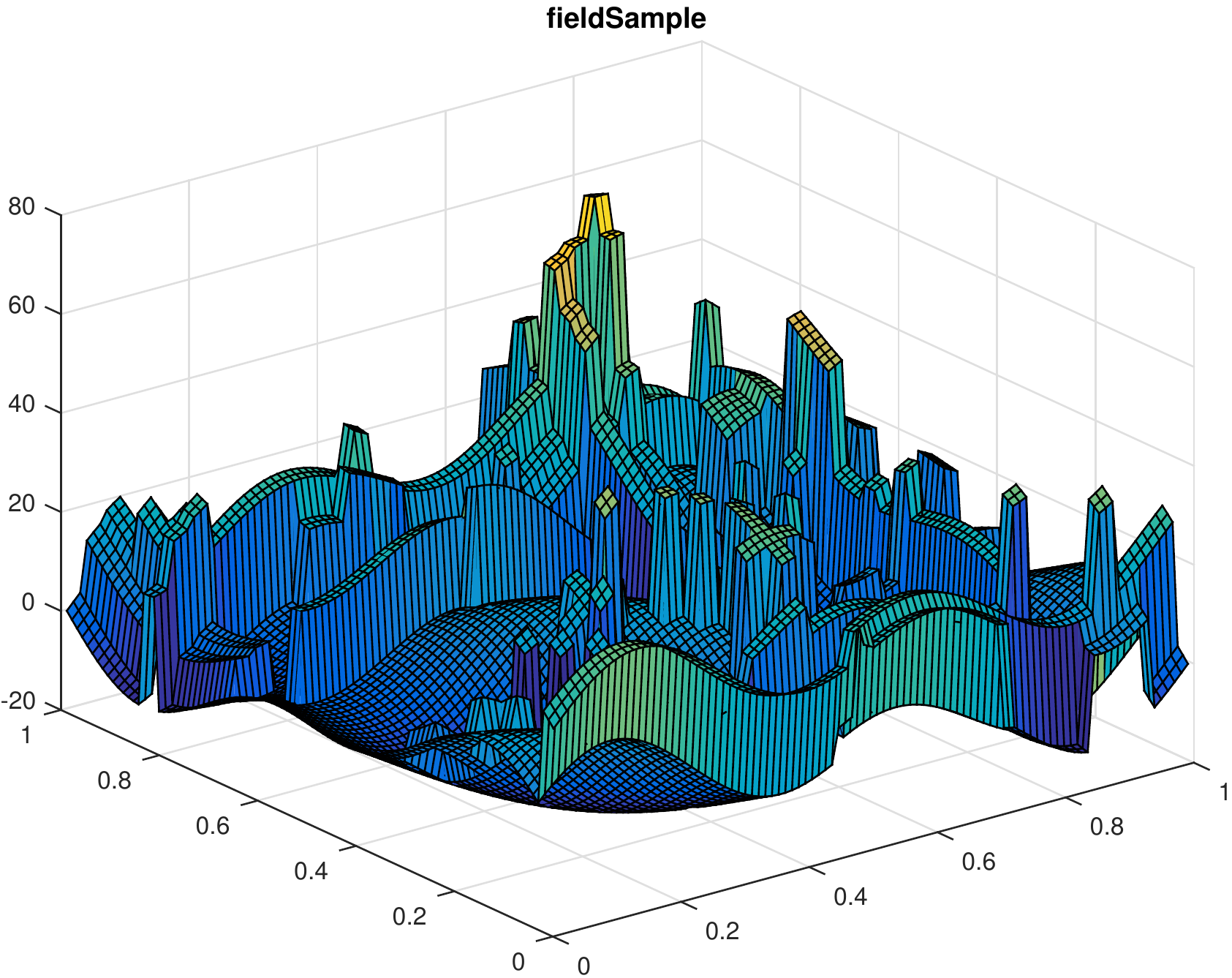}
\includegraphics[width = 0.30\textwidth]{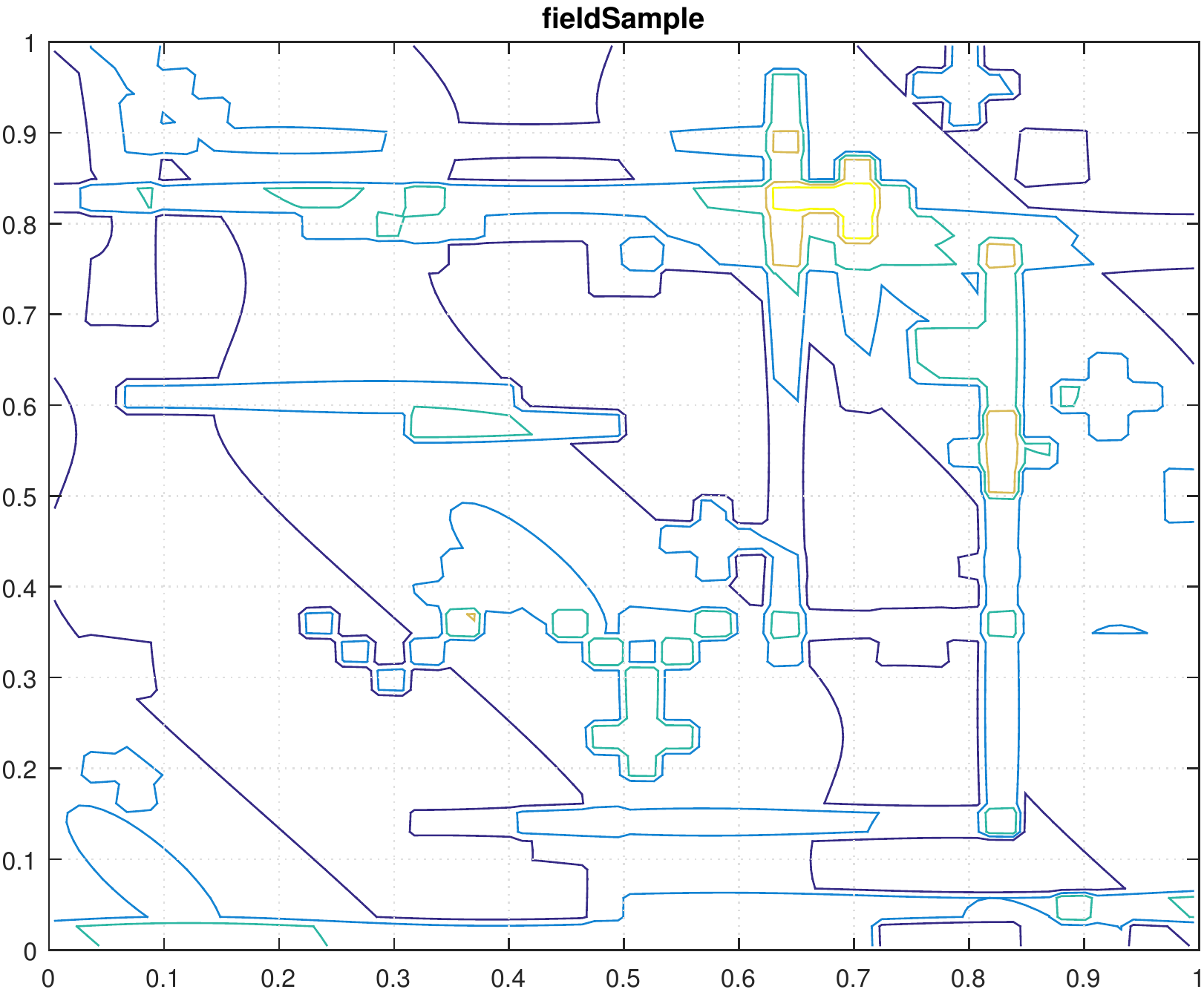}
\includegraphics[width = 0.30\textwidth]{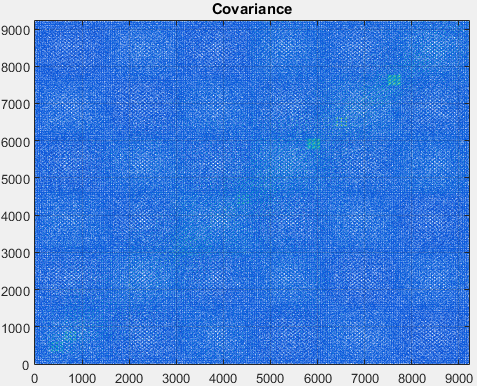}
\caption{One sample and the bird's-eye view. The covariance matrix is plotted on the right.}\label{fig:2dglobal}
\end{figure}

We apply the ISMD with patch size $H = 1/16$ on $A$ and get 37 intrinsic sparse modes as expected. Moreover, two of them are rotations of $[f_1, f_2]$ and the other 35 are exactly the 35 localized modes in the construction~\eqref{eqn:example2dglobalA}. We plot the first 6 intrinsic sparse modes in Figure~\ref{fig:2dglobal_modes}. As we can see, the ISMD separates the global modes and localized modes in $A$, or equivalently we separate the low rank dense part and sparse part of $A$. The reason why we can achieve this separation is that the representation~\eqref{eqn:example2dglobalA} in fact solves the patch-wise sparseness minimization problem~\eqref{opt:minsparseness}. The low-rank-plus-sparse decomposition (also known as Robust PCA, see \cite{chandrasekaran_rank-sparsity_2011, candes_robust_2011, luo_high_2011}) can also separate the low rank dense part and the sparse part in $A$. However, the computational cost of robust PCA is much more expensive than the ISMD.
\begin{figure}
\centering
\includegraphics[width = 0.15\textwidth,height = 0.15\textheight]{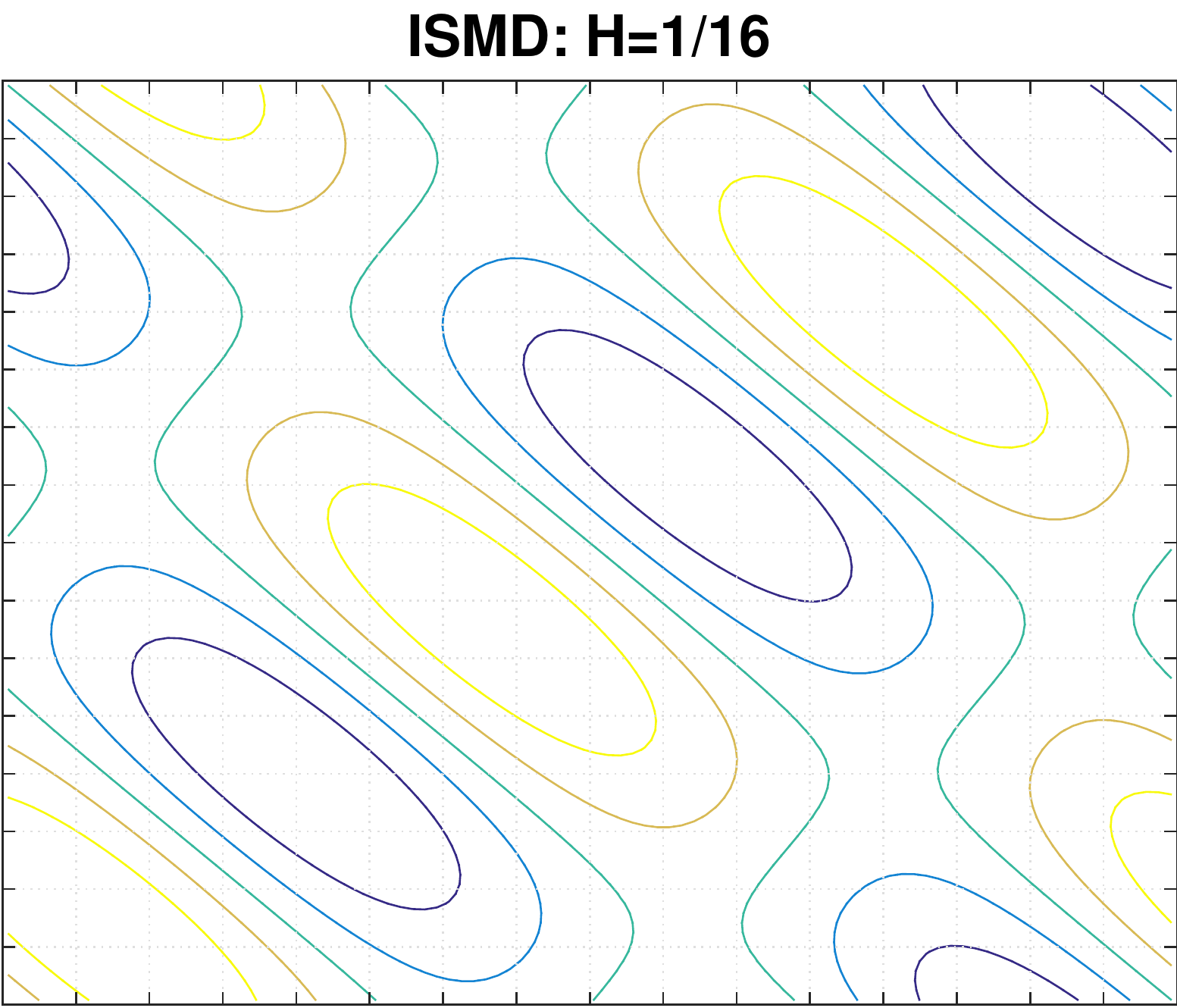}
\includegraphics[width = 0.15\textwidth,height = 0.15\textheight]{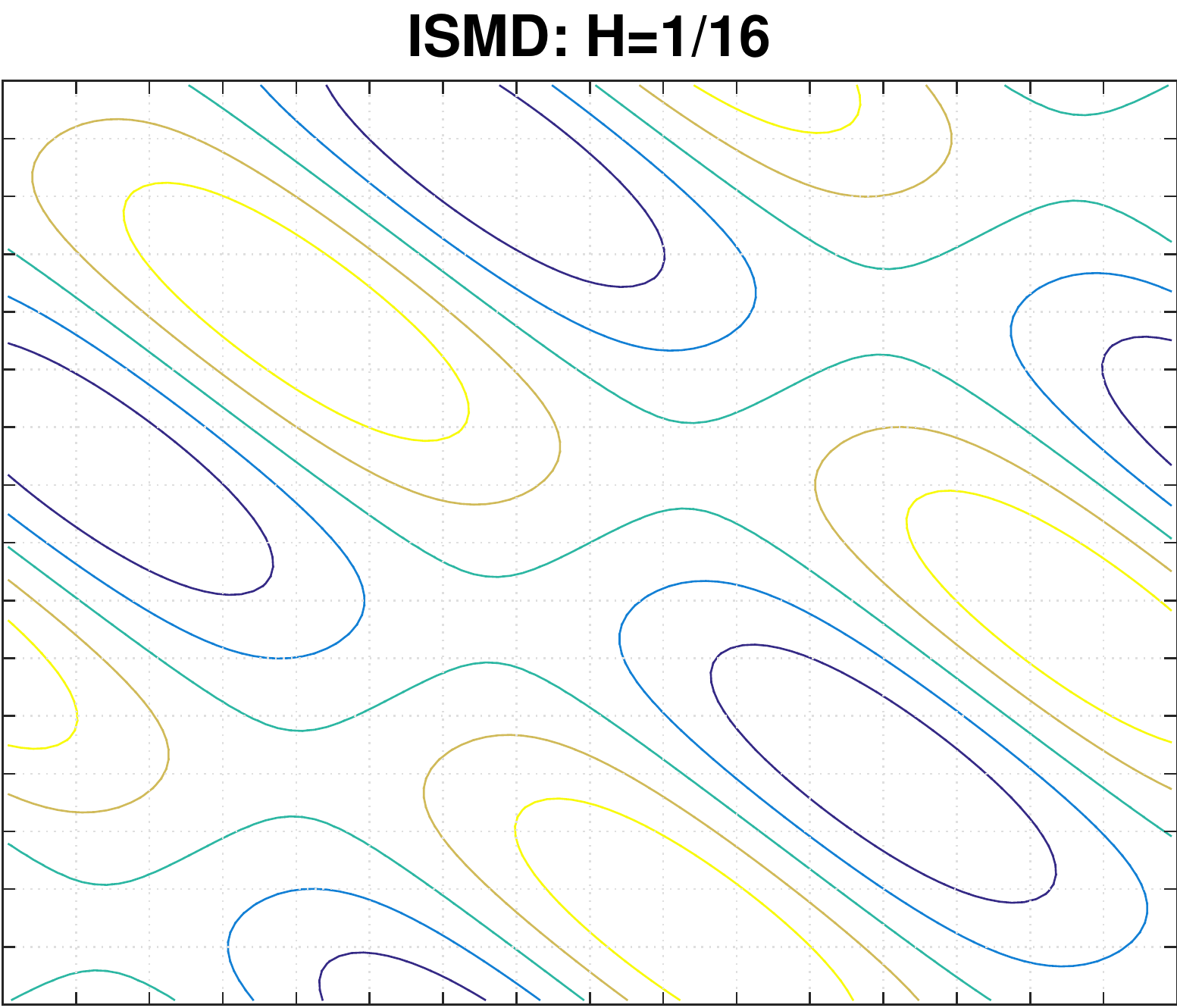}
\includegraphics[width = 0.15\textwidth,height = 0.15\textheight]{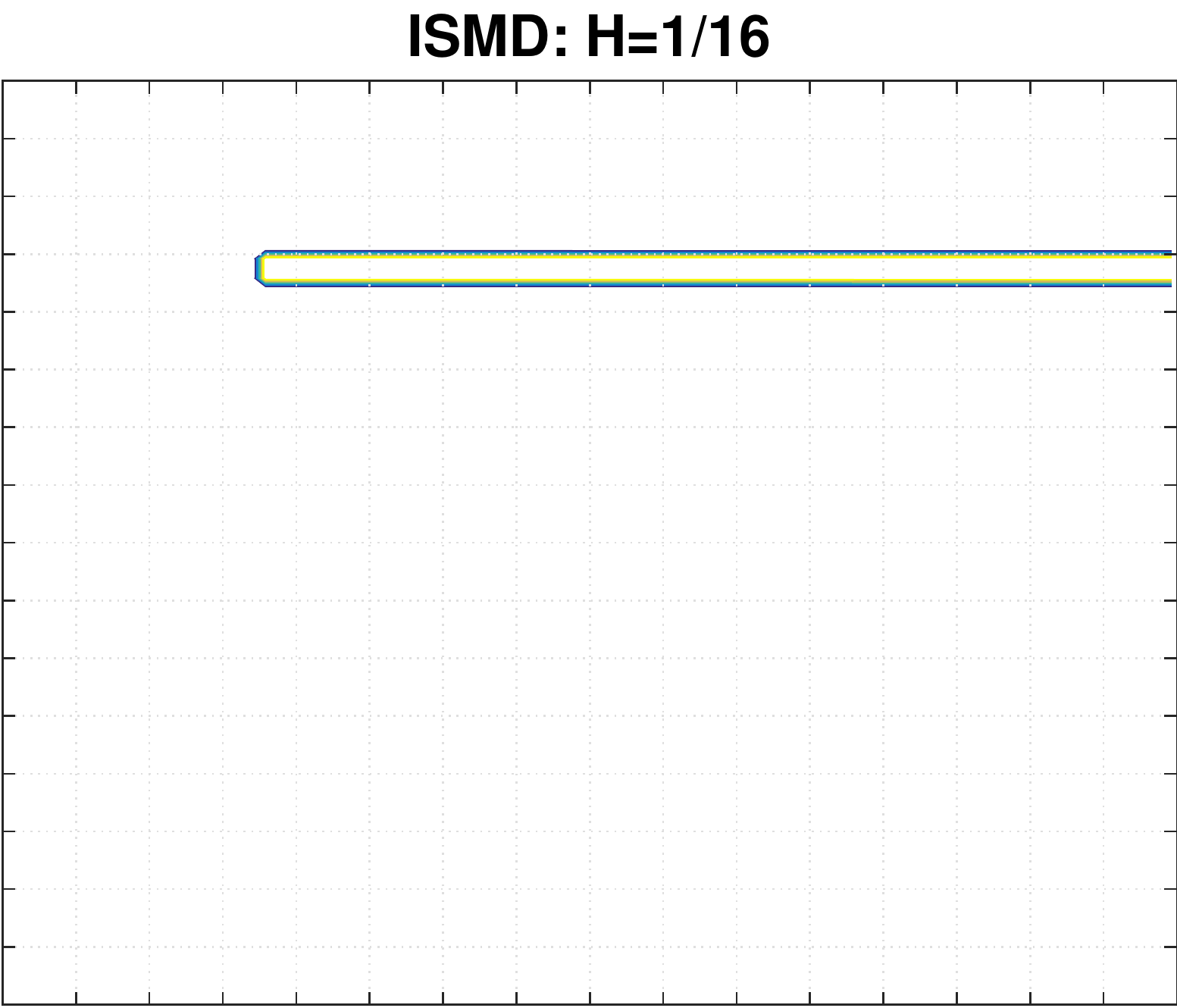}
\includegraphics[width = 0.15\textwidth,height = 0.15\textheight]{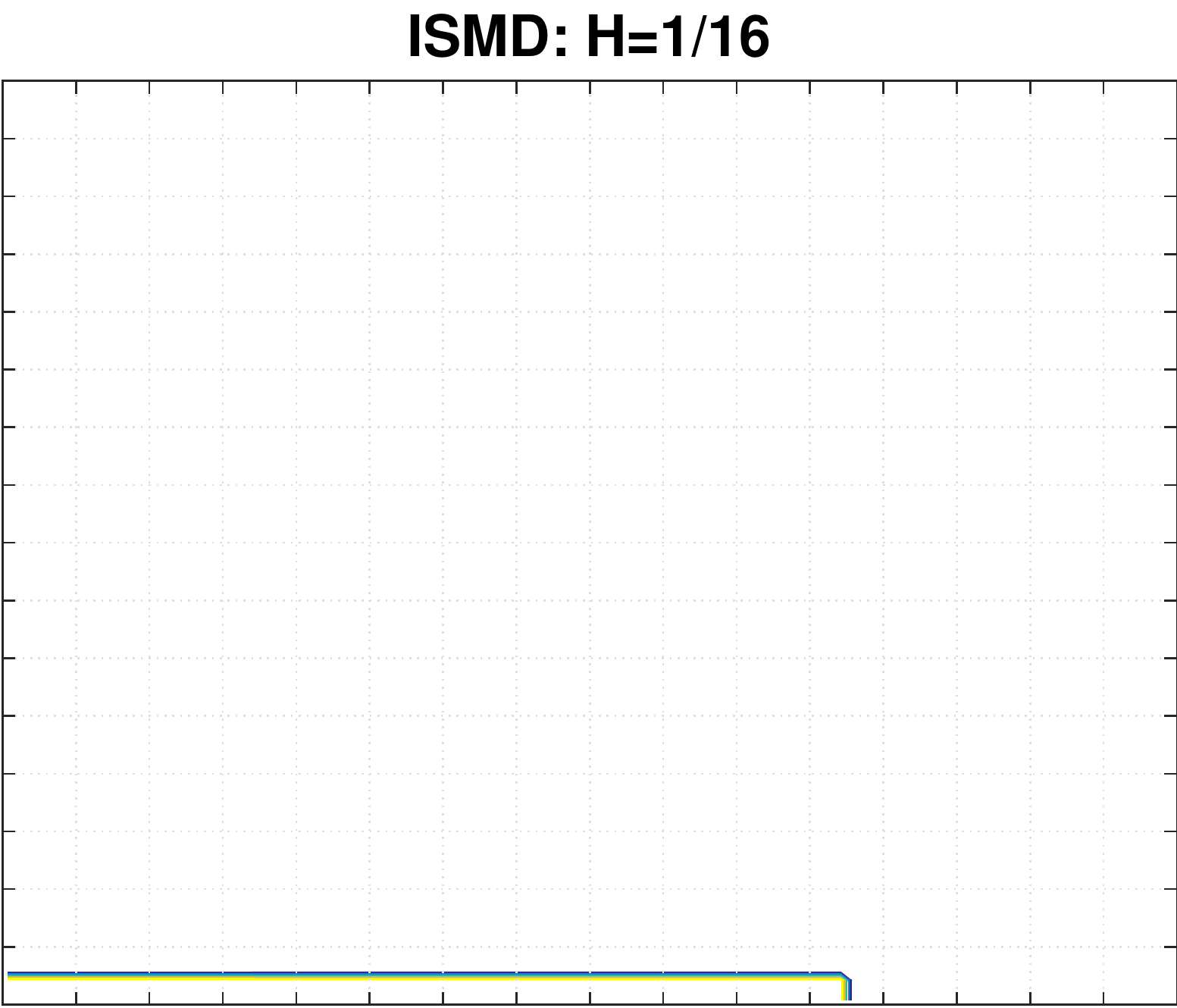} 
\includegraphics[width = 0.15\textwidth,height = 0.15\textheight]{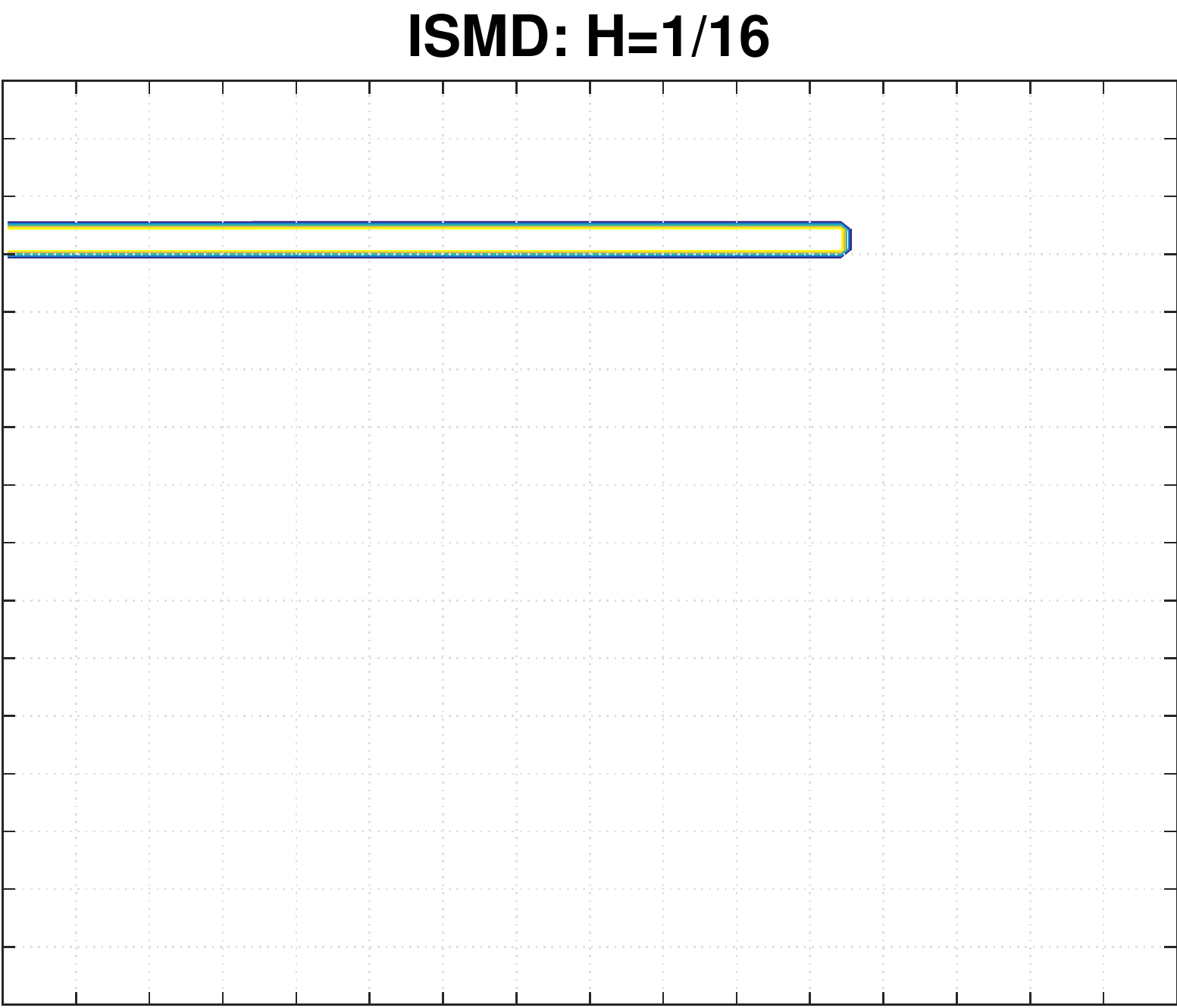}
\includegraphics[width = 0.15\textwidth,height = 0.15\textheight]{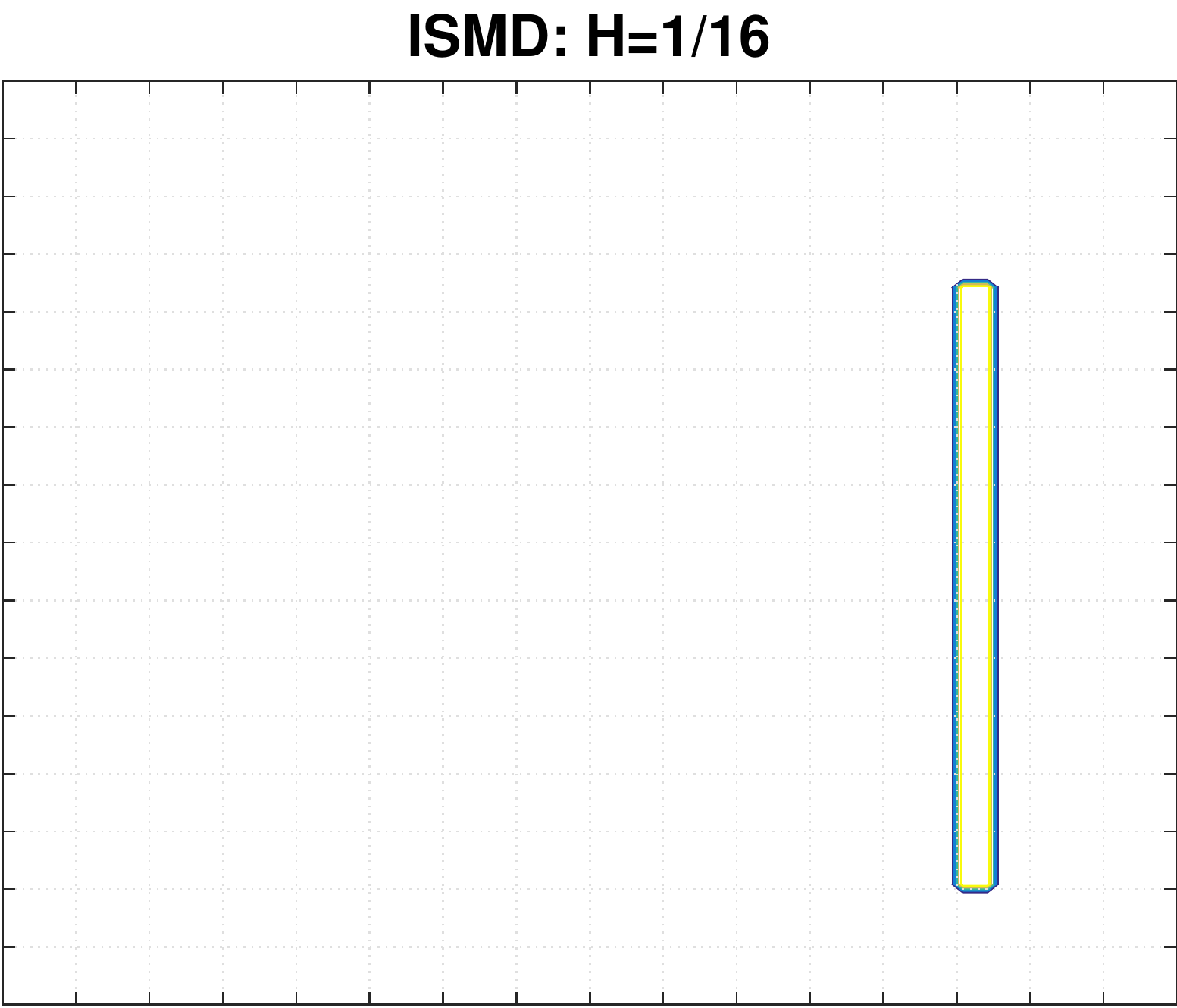}
\caption{First 6 intrinsic sparse modes (H=1/16, regular-sparse)}\label{fig:2dglobal_modes}
\end{figure}

\subsection{Application of Algorithm~\ref{alg:ISMDm1}}
When $A$ is constructed from model~\eqref{eqn:example2dglobalA} but is mixed with small noises as in Section~\ref{exam:noise1}, we cannot simply apply the thresholding~\eqref{eqn:threshold1} any more. In this case, we have unidentifiable modes $f_1$ and $f_2$ and thus $\Omega$ may contain nonzero values other than $\pm 1$. For the noise level $\epsilon = 10^{-6}$, Figure~\ref{fig:hatOmega} (left) shows the histogram of absolute values of entries in $\hat{\Omega}$. We can clearly see a gap between $\Or(\epsilon)$ entries and $\Or(1)$ entries from Figure~\ref{fig:hatOmega}(left). Therefore we choose a threshold $\epsilon_{th} = 10^{-3}$ and apply the modified ISMD algorithm~\ref{alg:ISMDm1} on $\hat{A}$. The first 6 perturbed intrinsic sparse modes $\hat{g}_k$ are shown in Figure~\ref{fig:2dglobal_modes_noise}. We can see that their supports are exactly the same as those of the unperturbed intrinsic sparse modes $g_k$ in Figure~\ref{fig:2dglobal_modes}. In fact, the first 37 perturbed intrinsic sparse modes $\{\hat{g}_k\}_{k=1}^{37}$ exactly capture the supports of the unperturbed intrinsic sparse modes $\{g_k\}_{k=1}^{37}$. However, we have several extra perturbed intrinsic sparse modes with very small $l^2$ error since $\hat{\Omega}$ has rank more than $37$.

When we raise the noise level $\epsilon$ to $10^{-4}$, the histogram of the absolute values in $\hat{\Omega}$ is shown in Figure~\ref{fig:hatOmega}(right). In this case, we cannot identify a gap any more. From Figure~\ref{fig:hatOmega}(left), we see that the exact $\Omega$ has entries in the order of $10^{-3}$. Therefore, the noise level $\epsilon=10^{-4}$ is large enough to mix the true nonzero values and noisy null values in $\hat{\Omega}$ together. In Figure~\ref{fig:hatOmega} the total counts are different because only values between $10^{-16.5}$ and $10^{0.5}$ are counted.
\begin{figure}
\centering
\includegraphics[width = 0.45\textwidth,height = 0.2\textheight]{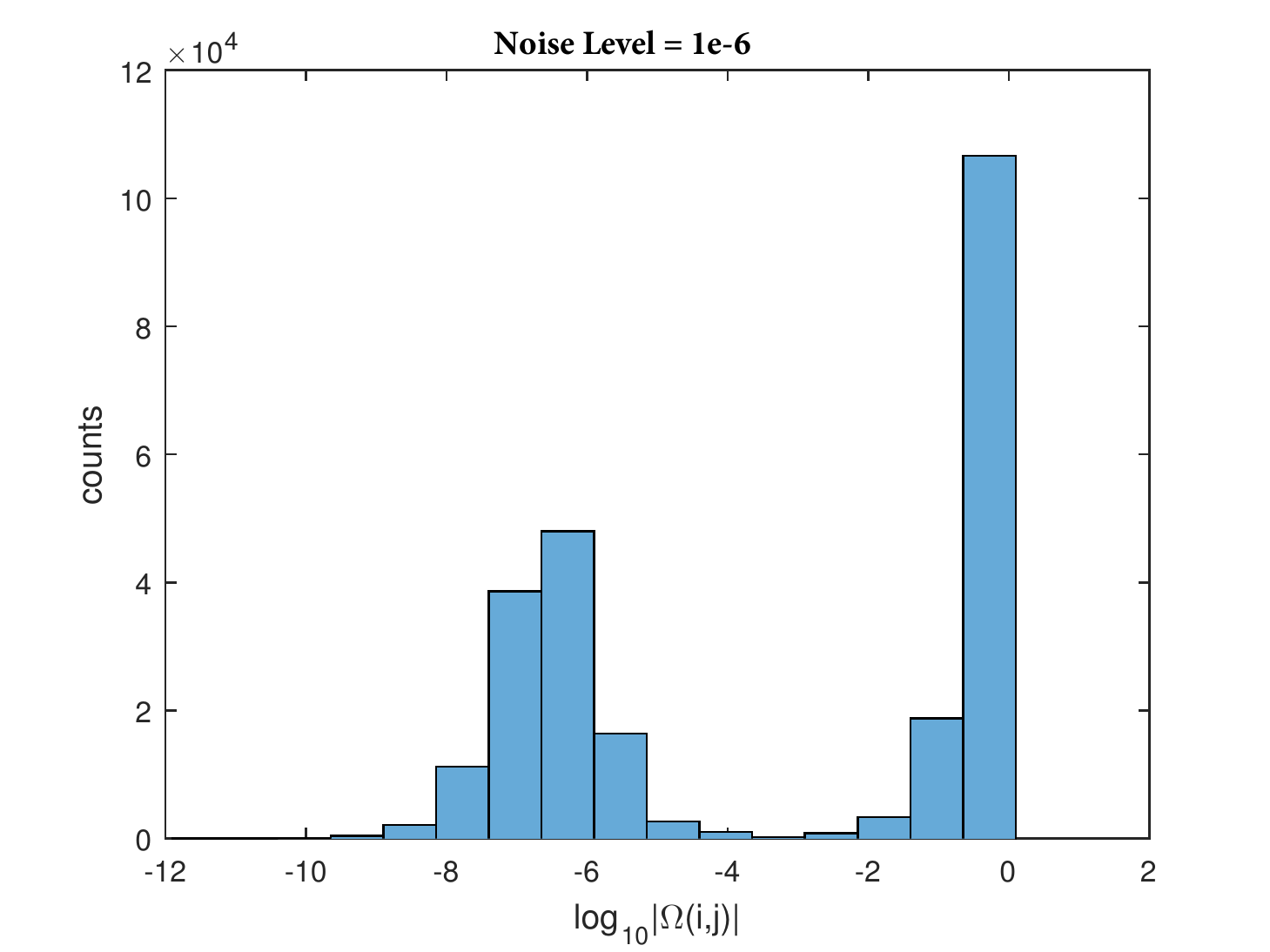}
\includegraphics[width = 0.45\textwidth,height = 0.2\textheight]{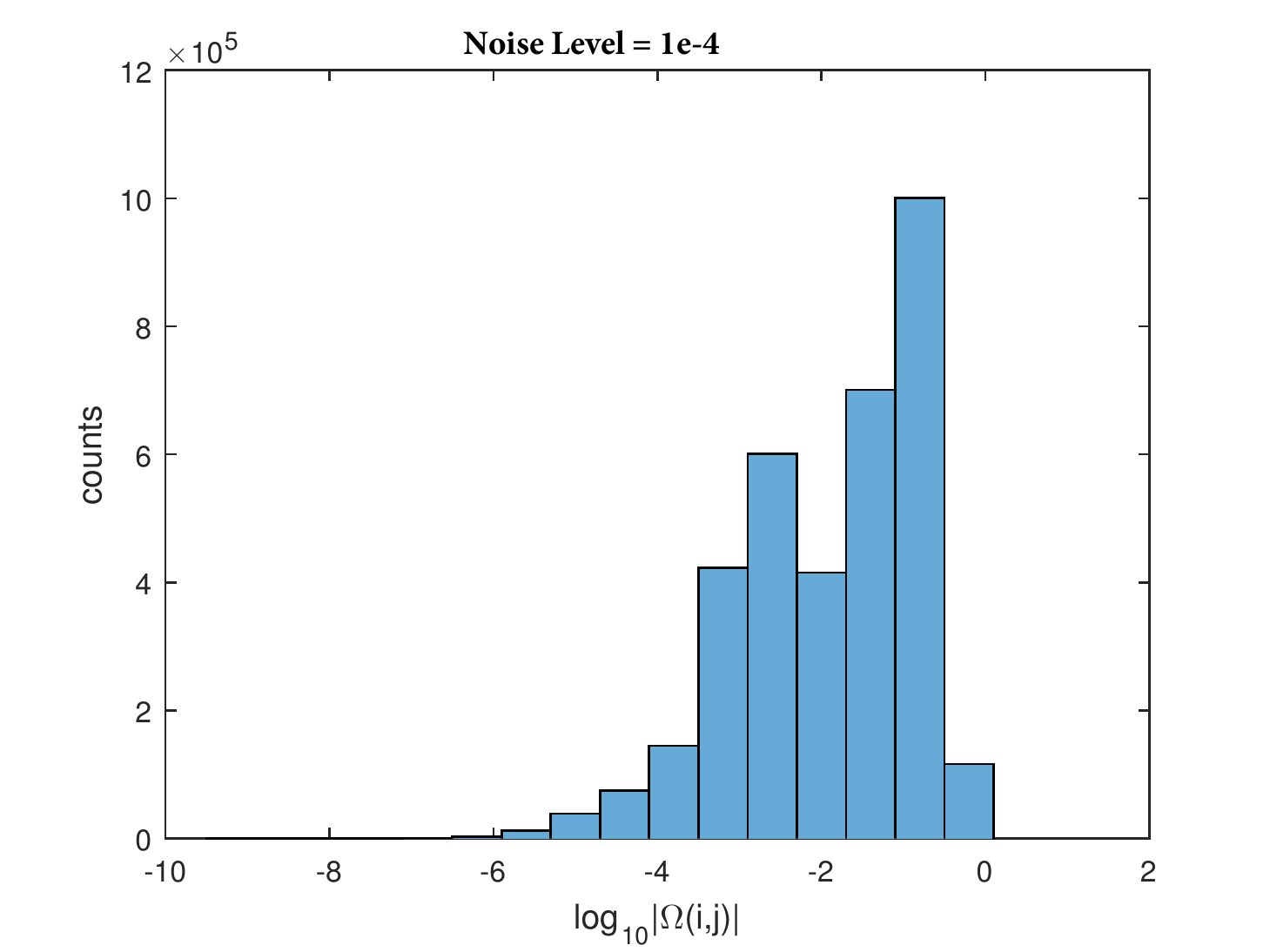}
\caption{Histogram of absolute values of entries in $\hat{\Omega}$.}\label{fig:hatOmega}
\end{figure}

\begin{figure}
\centering
\includegraphics[width = 0.15\textwidth,height = 0.15\textheight]{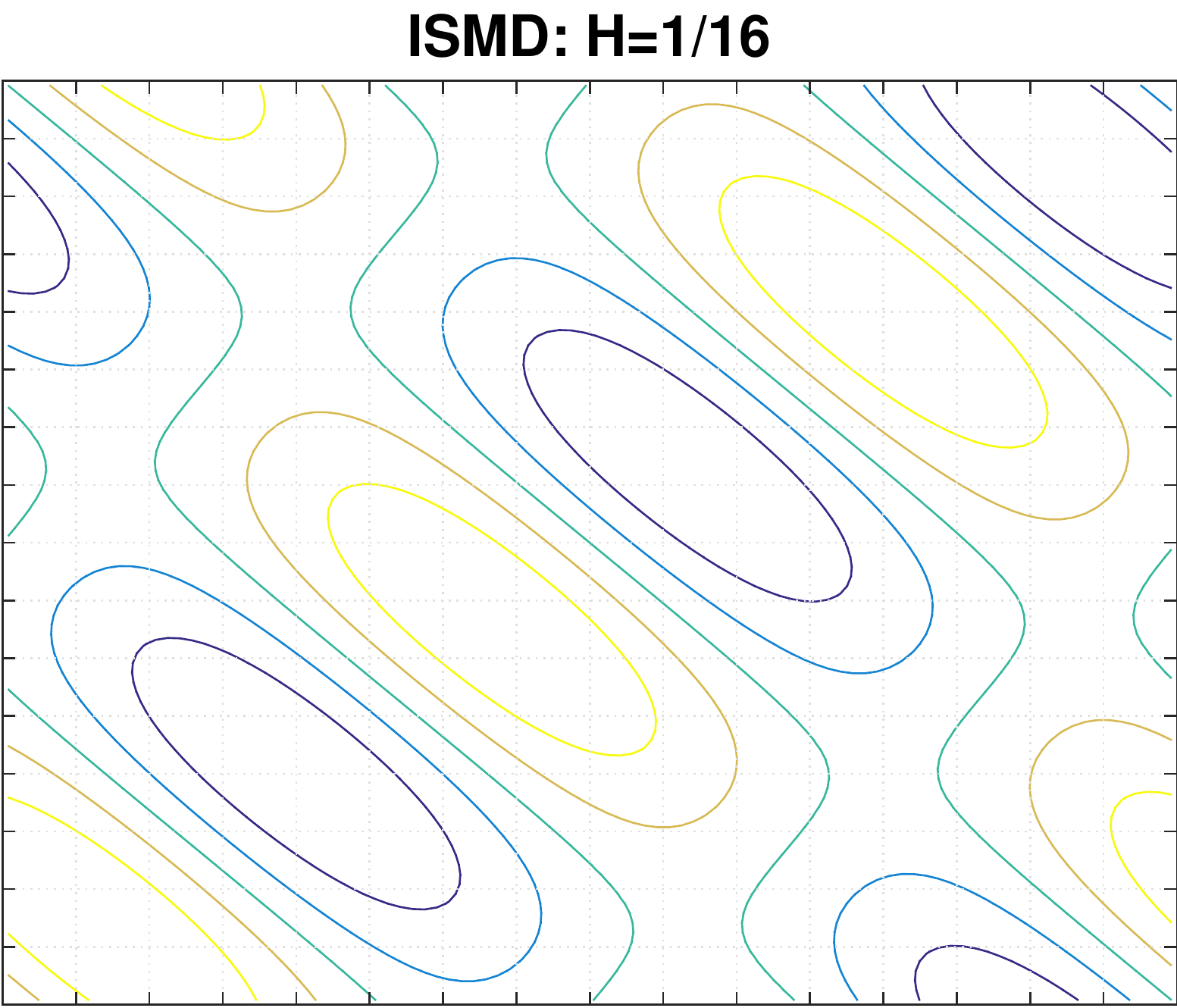}
\includegraphics[width = 0.15\textwidth,height = 0.15\textheight]{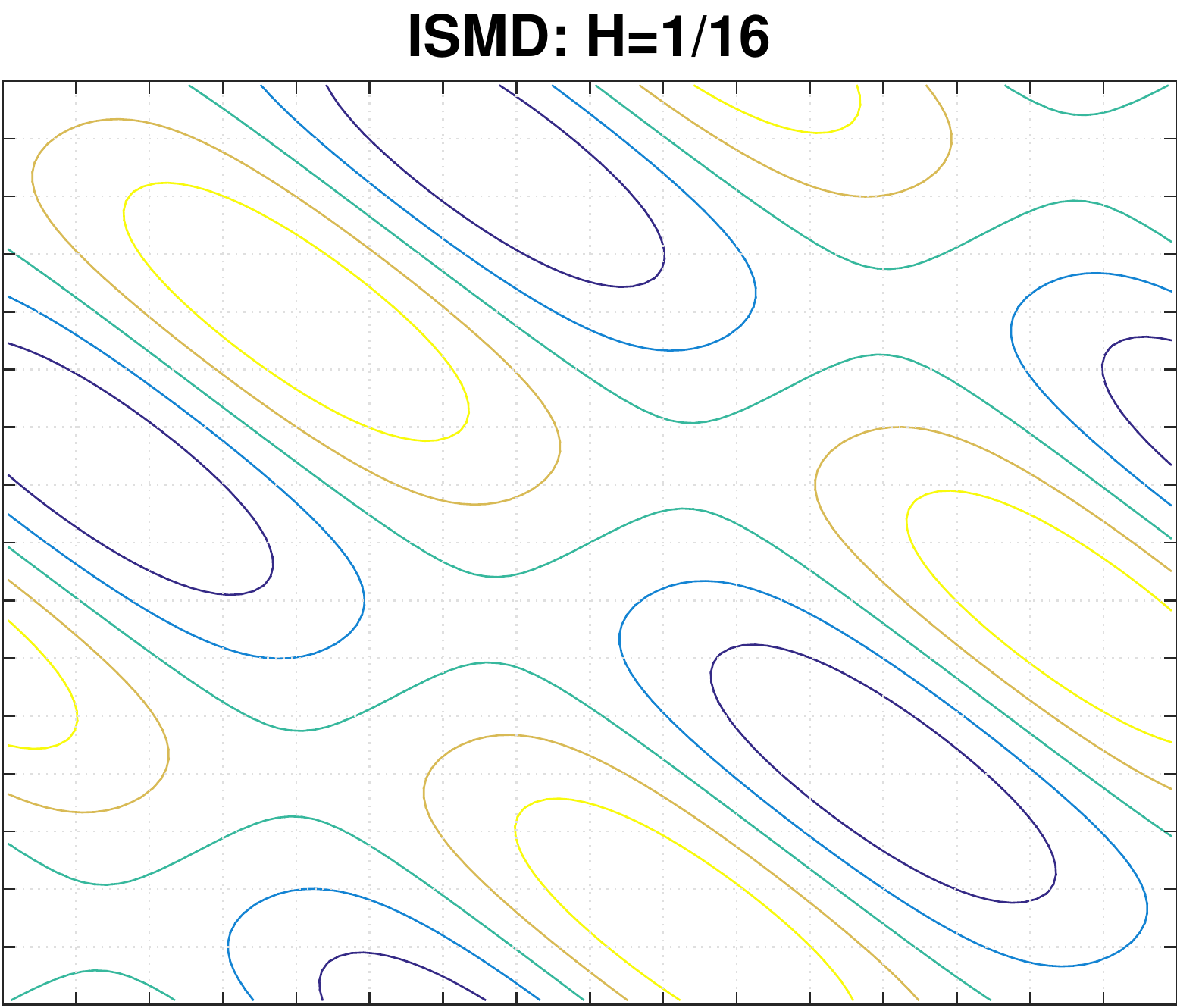}
\includegraphics[width = 0.15\textwidth,height = 0.15\textheight]{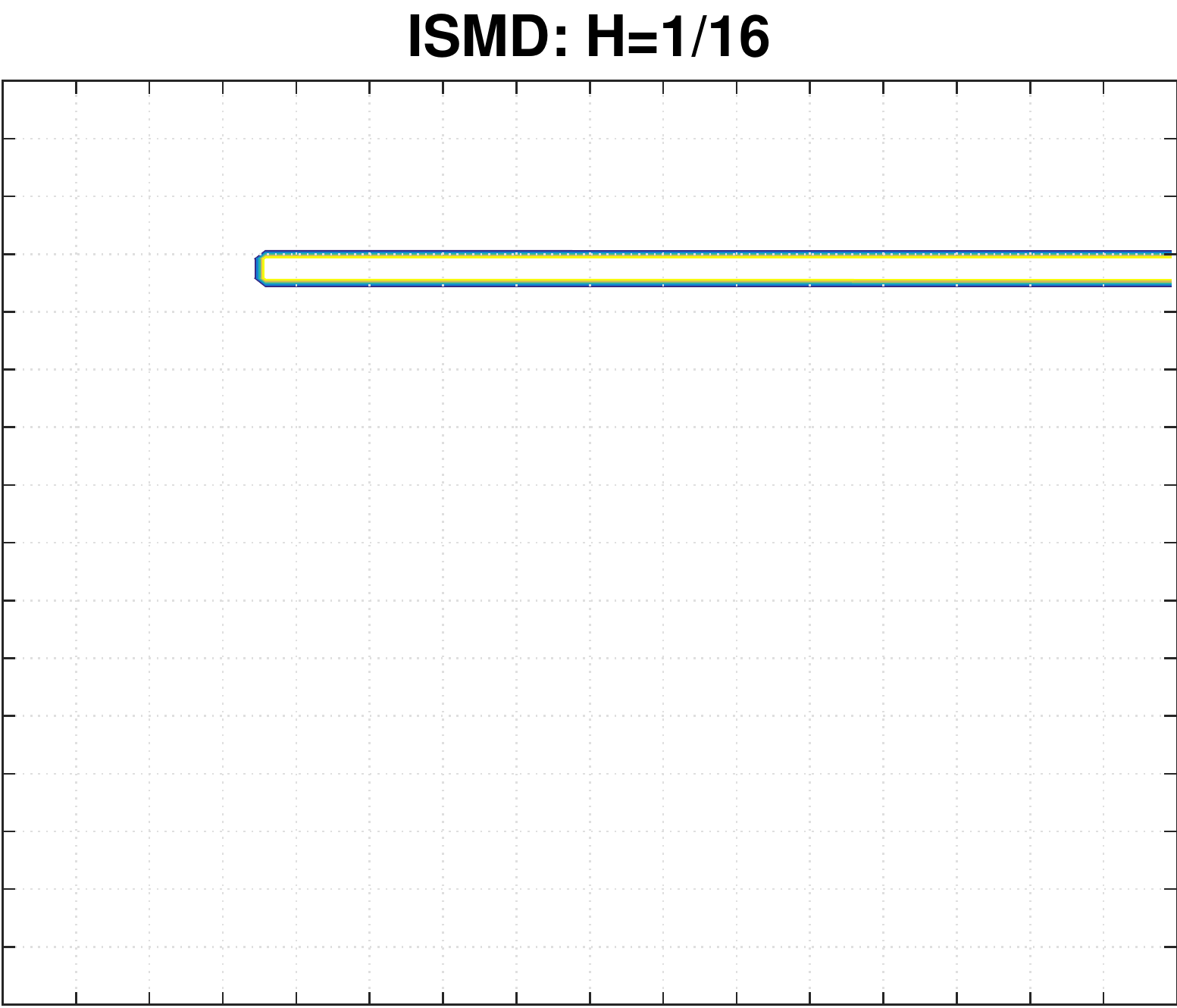}
\includegraphics[width = 0.15\textwidth,height = 0.15\textheight]{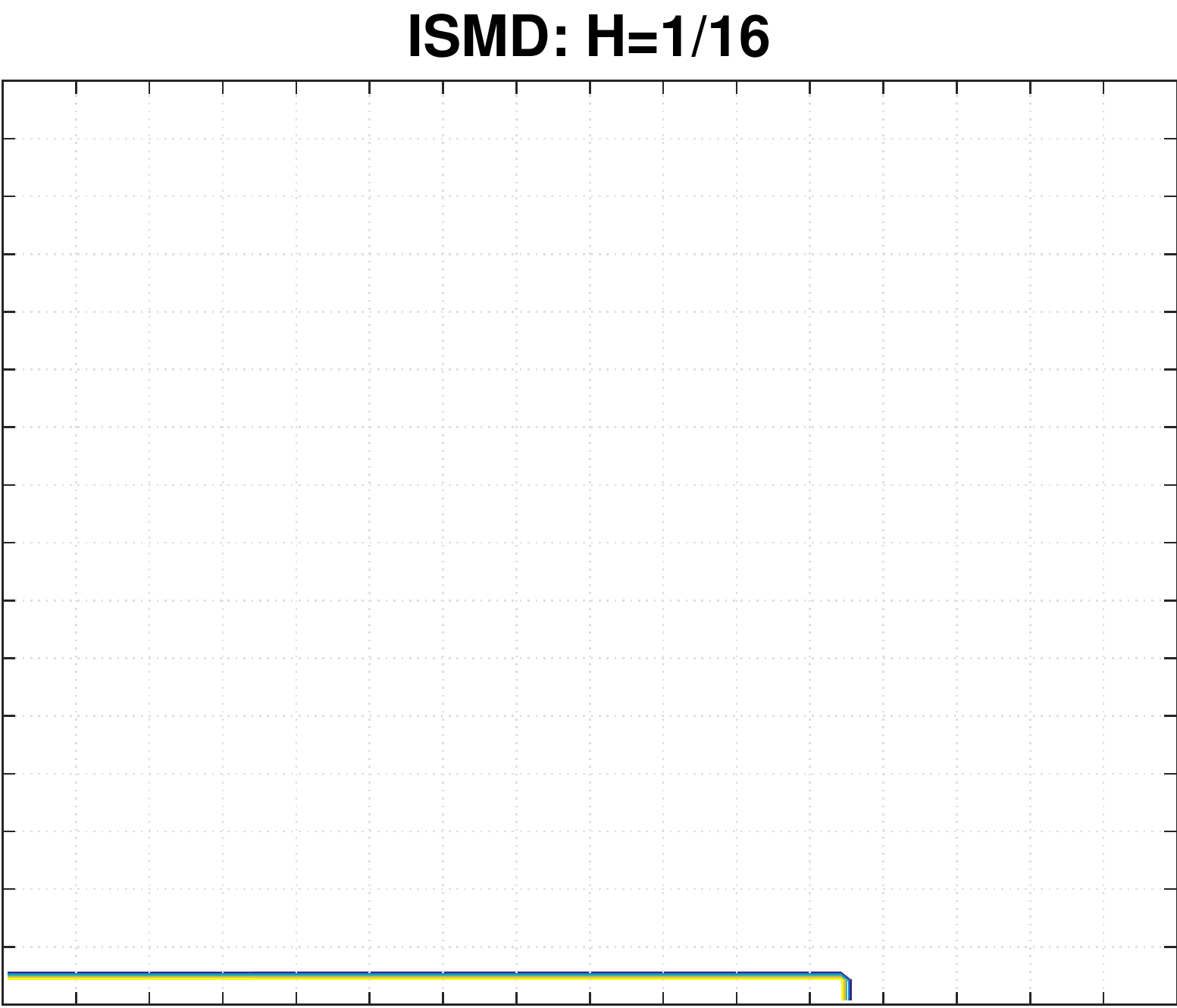} 
\includegraphics[width = 0.15\textwidth,height = 0.15\textheight]{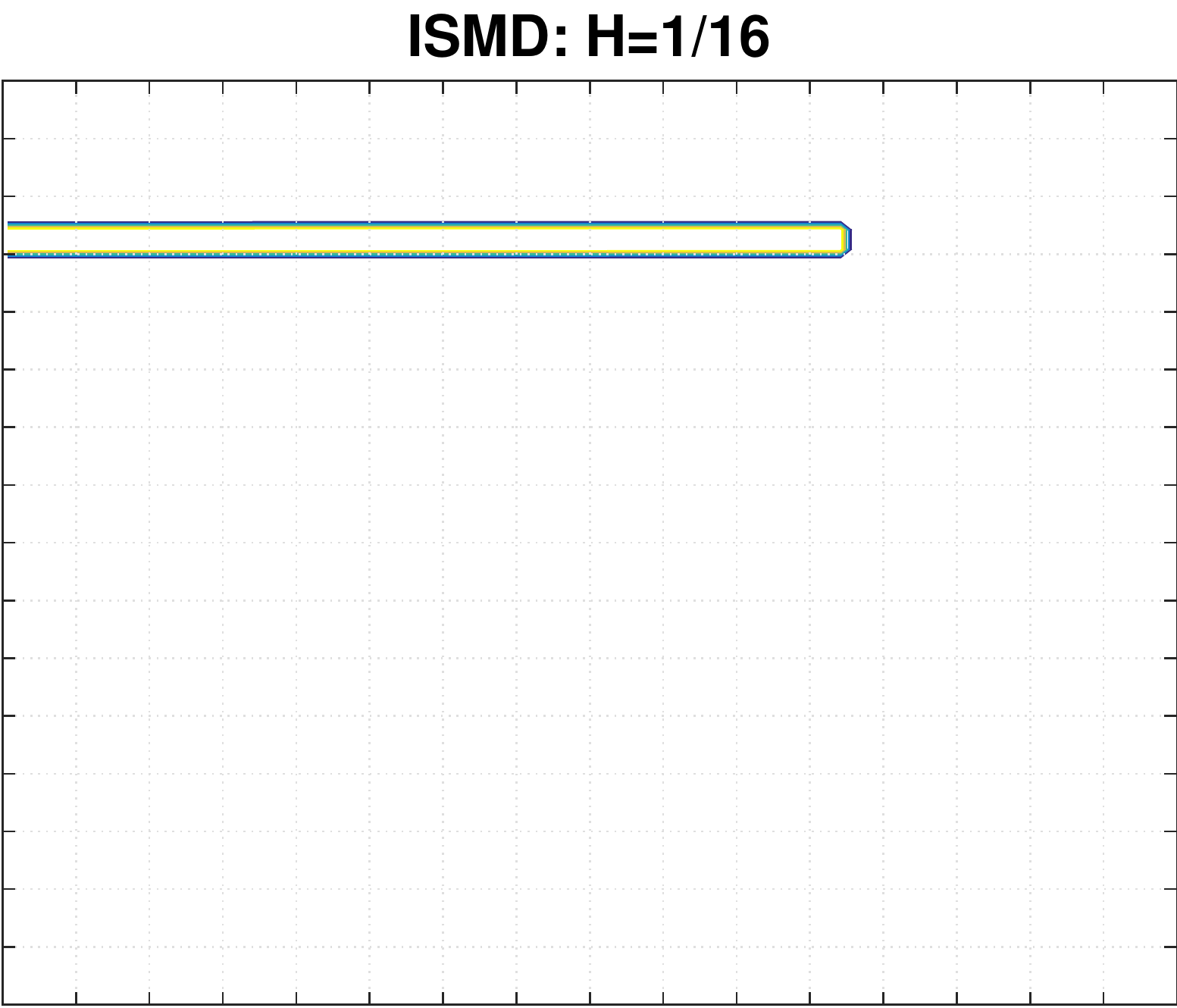}
\includegraphics[width = 0.15\textwidth,height = 0.15\textheight]{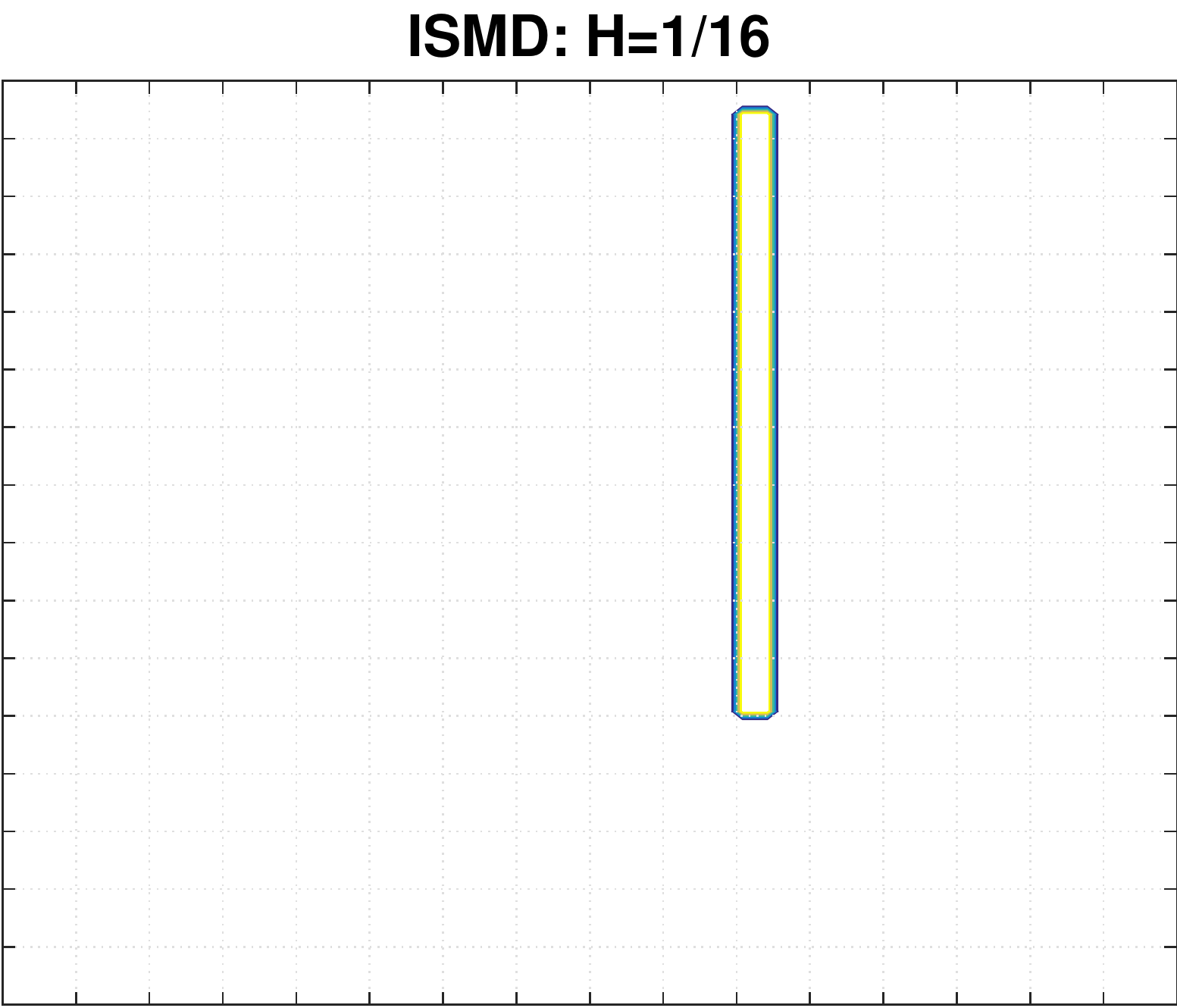}
\caption{Application of Algorithm~\ref{alg:ISMDm1} (H=1/16, approximately regular-sparse): first 6 intrinsic sparse modes}\label{fig:2dglobal_modes_noise}
\end{figure}

\subsection{Application of Algorithm~\ref{alg:ISMDm2}}\label{sec:exponential}
In this section, we consider the one-dimensional Poisson kernel:
\begin{equation*}
a(x,y) = e^{-\frac{|x-y|}{l}}\,,\quad x,y\in[-1,1]\,.
\end{equation*}
where $l = 1/16$. To refine the small scale, $a(x,y)$ is discretized by a uniform grid with $h = 1/512$, resulting in $A \in \RR^{1024\times 1024}$. In Figure~\ref{fig:1dcovKL} we plot the covariance matrix. By truncating the eigen decomposition with 45 modes, we can approximate $A$ with spectral norm error $5\%$, and these 45 KL modes are plotted on the right panel of the figure. As one can see, they are all global functions. 
\begin{figure}[h]
\centering
\includegraphics[width = 0.45\textwidth,height = 0.2\textheight]{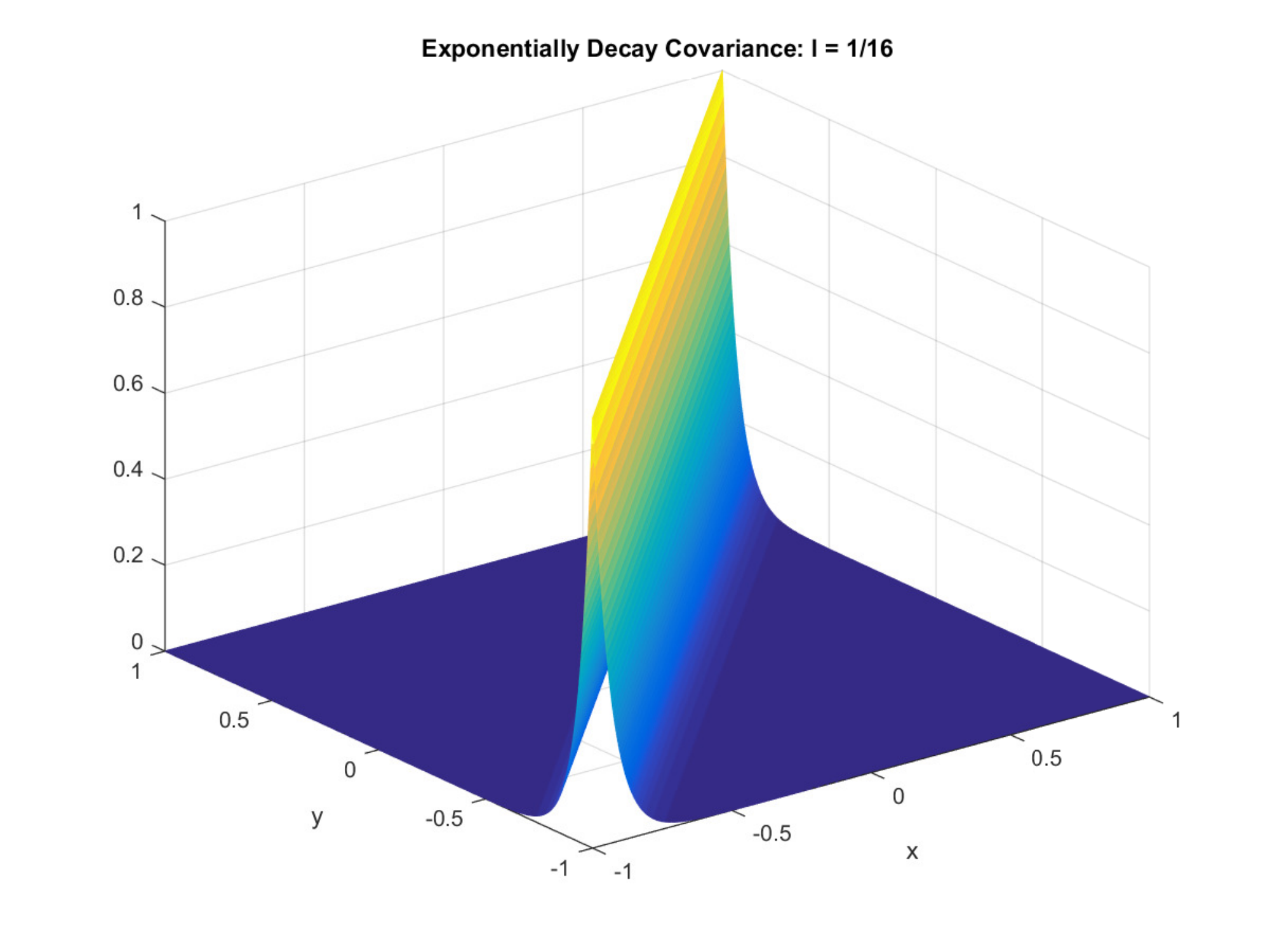}
\includegraphics[width = 0.45\textwidth,height = 0.2\textheight]{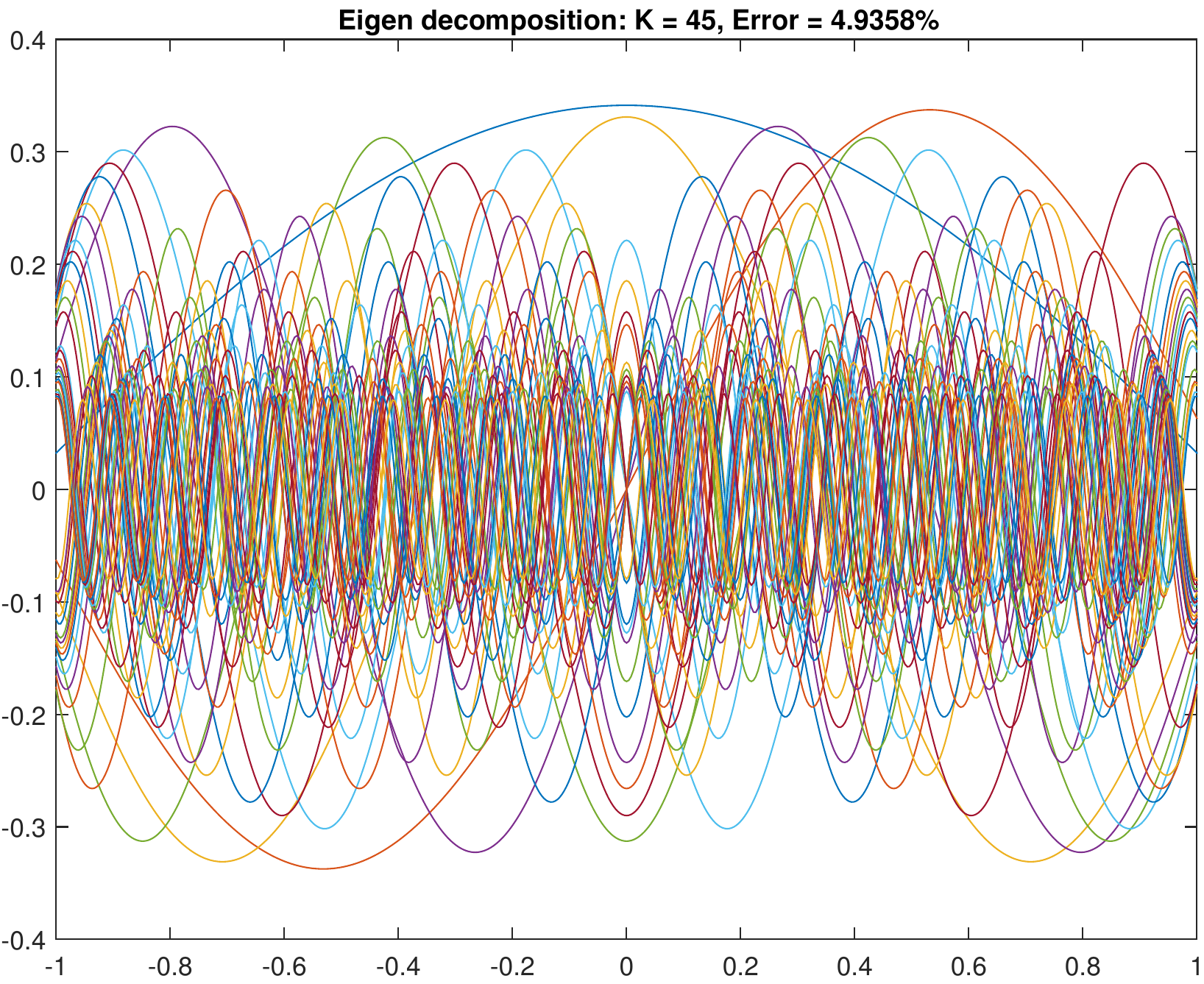}
\caption{Eigen-Decomposition: Covariance function and its first 45 KL modes. Error is $4.936\%$. Both local and global dimension are 45.}\label{fig:1dcovKL}
\end{figure}

We decompose the domain into $2$, $4$ and $8$ patches respectively and apply the Algorithm~\ref{alg:ISMDm2} with thresholding~\eqref{eqn:threshold1} to each case. For all the three cases, every mode has patch-wise sparseness either 1 or 2. In Figure~\ref{fig:1dnum}, the left panels show the modes that are nonzero on more than one patch, and the right panels collect the modes that are nonzero on only one patch. To achieve the same accuracy with the eigen decomposition, the numbers of modes needed are 45, 47 and 49 respectively. The total number is slightly larger than the number of eigen modes, but most modes are localized. For the two-patch case, each patch contains 23 nonzero modes, and for the four-patch case, each patch contains either 12 or 13 nonzero modes, and for the eight-patch case, each patch contains only 7 nonzero modes.
\begin{figure}[h]
\centering
\includegraphics[width = 0.45\textwidth,height = 0.15\textheight]{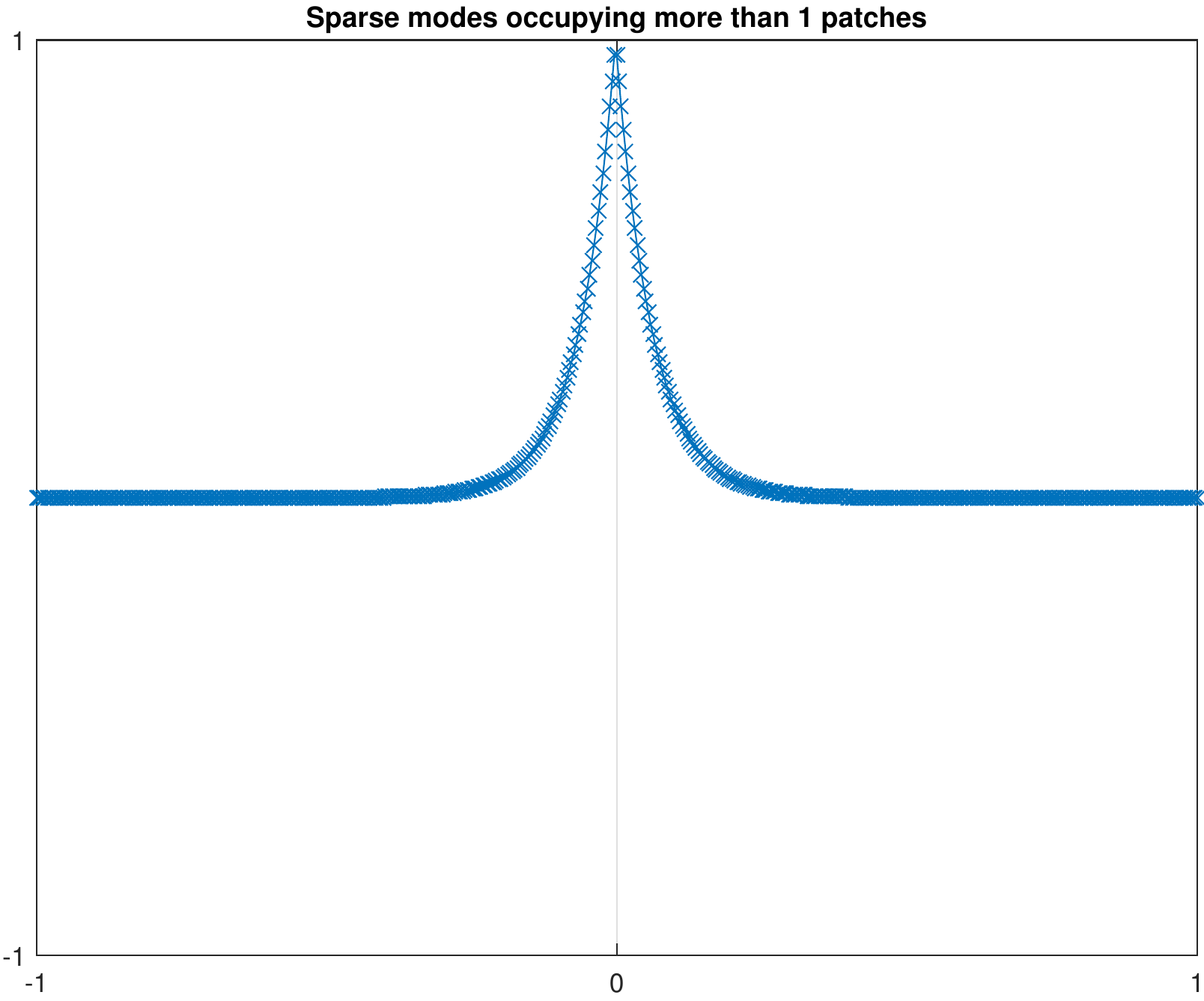}
\includegraphics[width = 0.45\textwidth,height = 0.15\textheight]{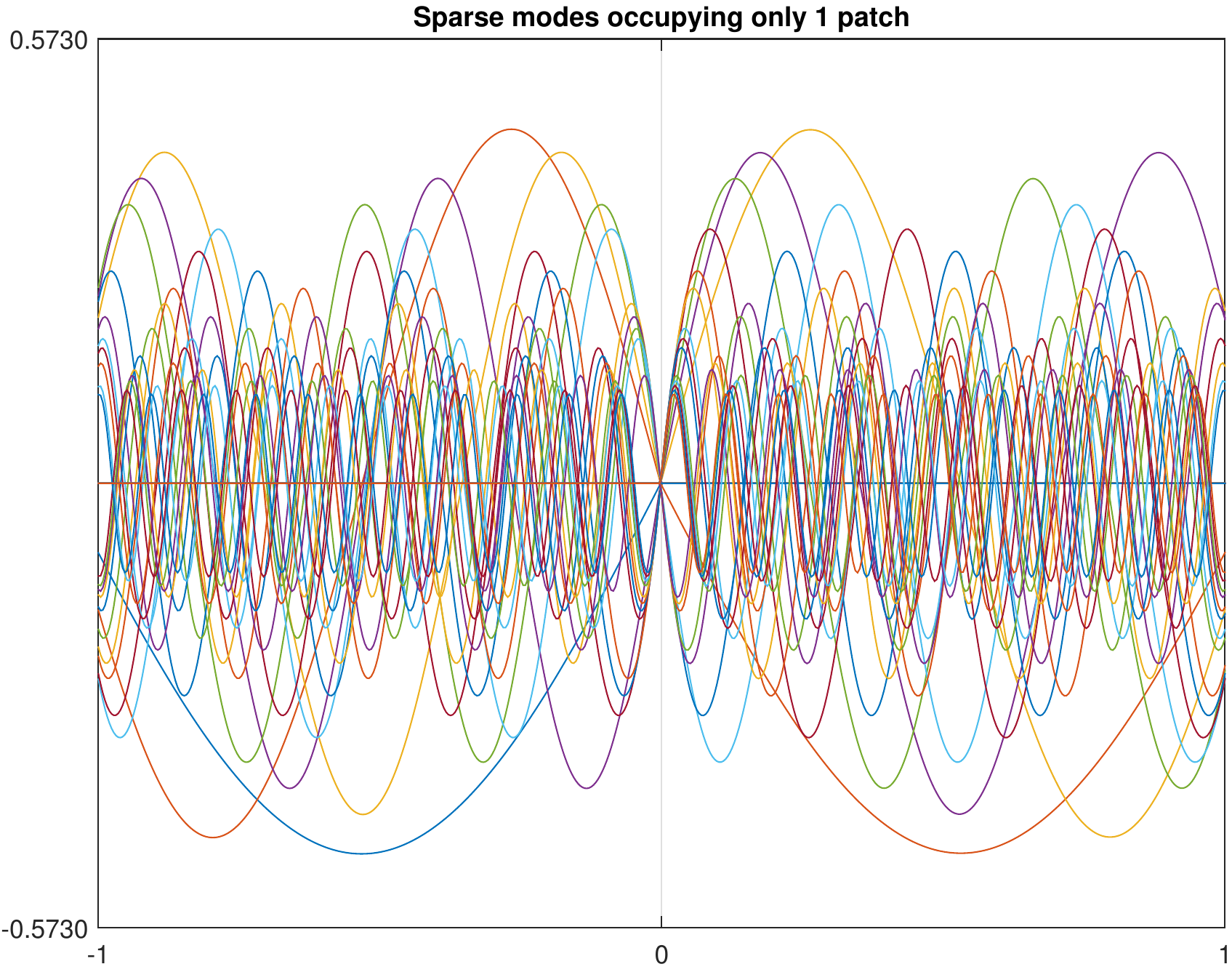}\\
\includegraphics[width = 0.45\textwidth,height = 0.15\textheight]{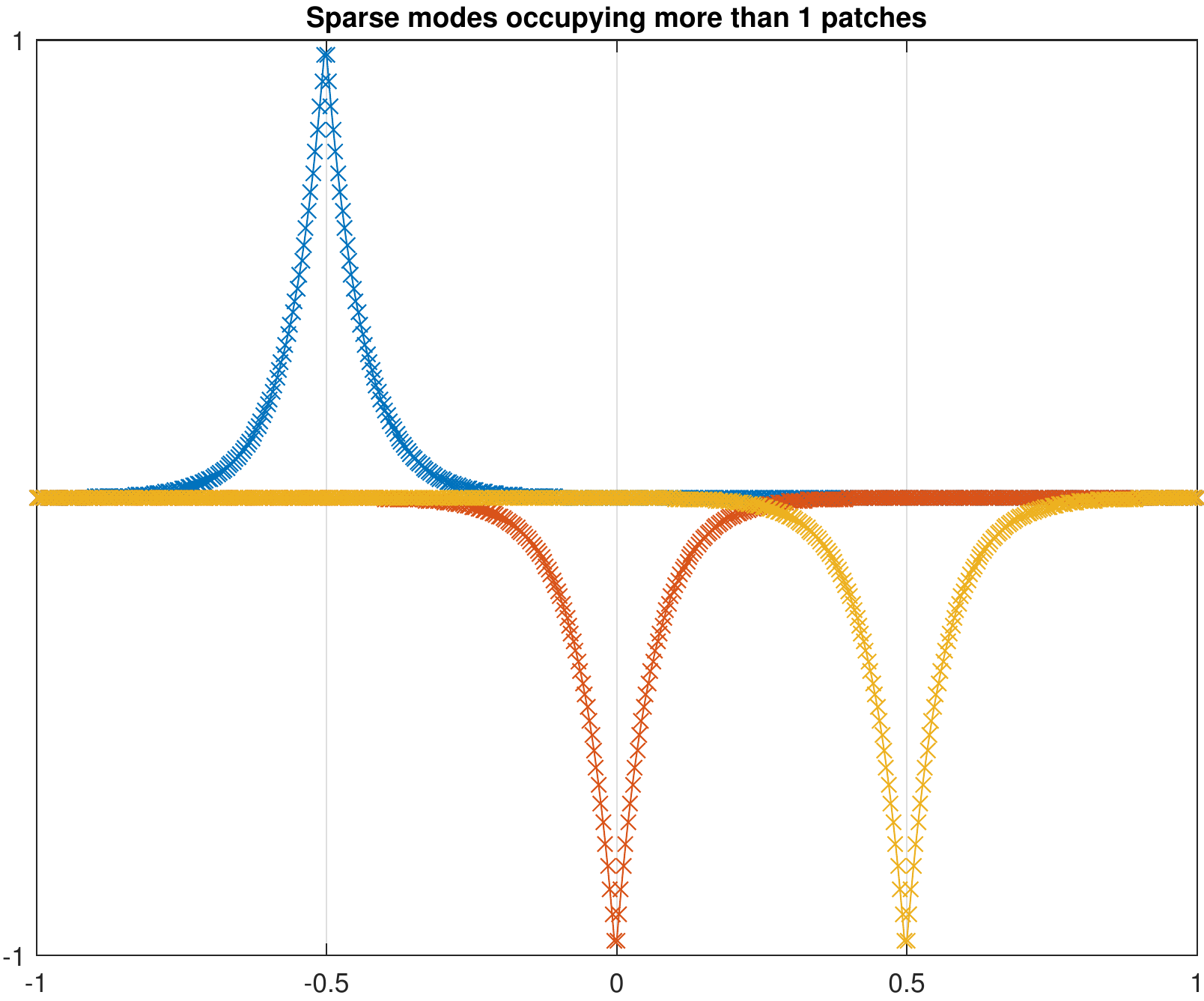}
\includegraphics[width = 0.45\textwidth,height = 0.15\textheight]{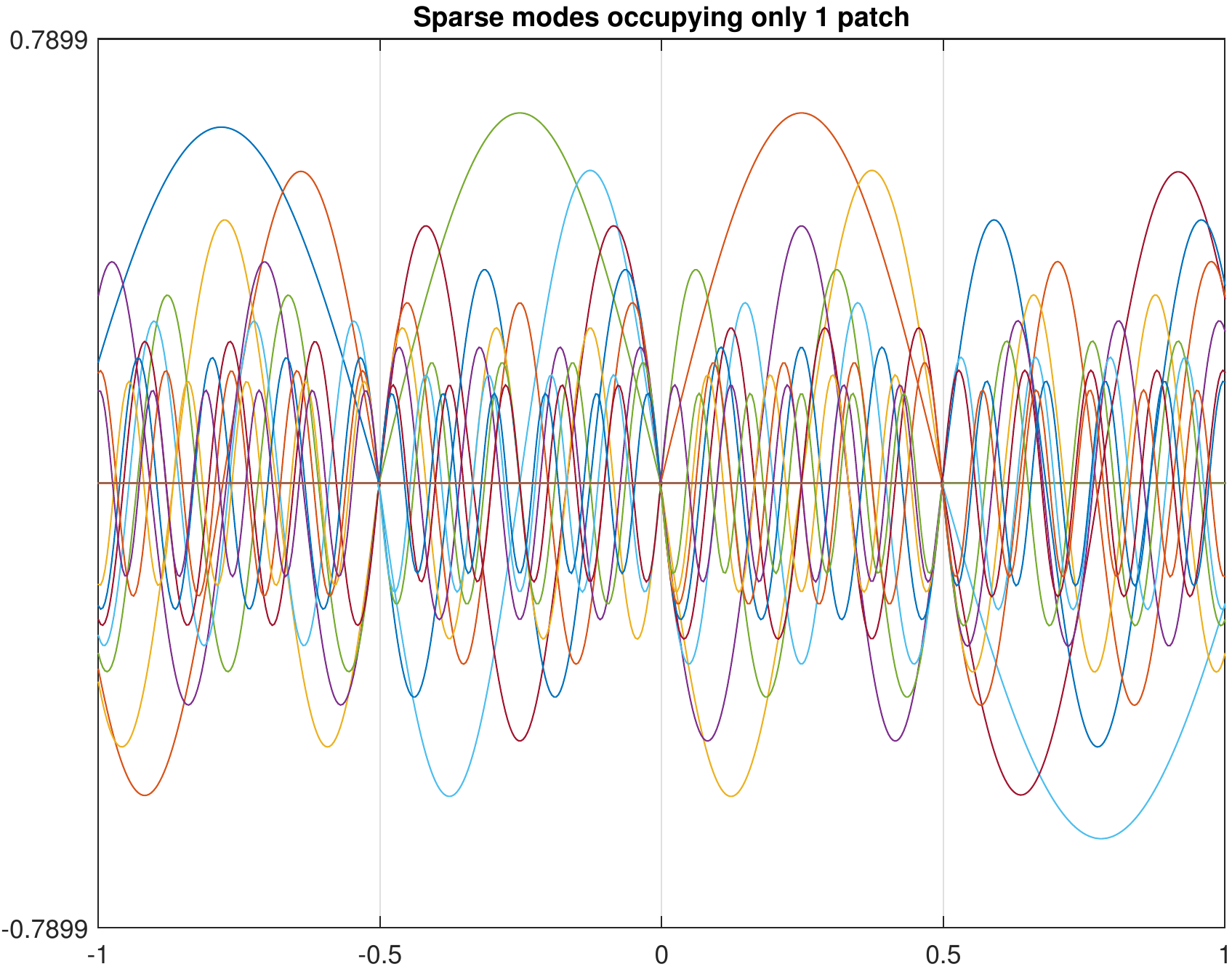}\\
\includegraphics[width = 0.45\textwidth,height = 0.15\textheight]{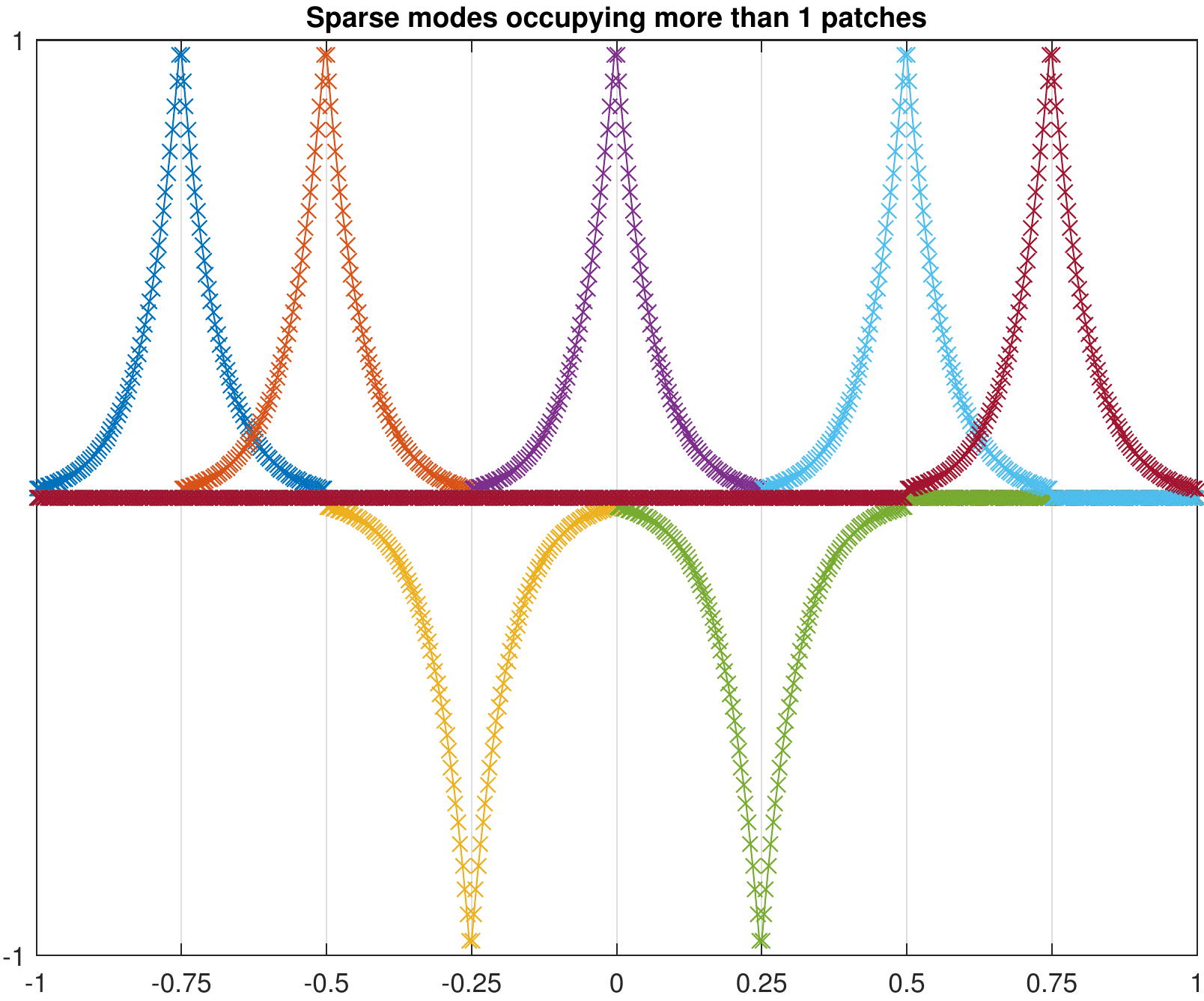}
\includegraphics[width = 0.45\textwidth,height = 0.15\textheight]{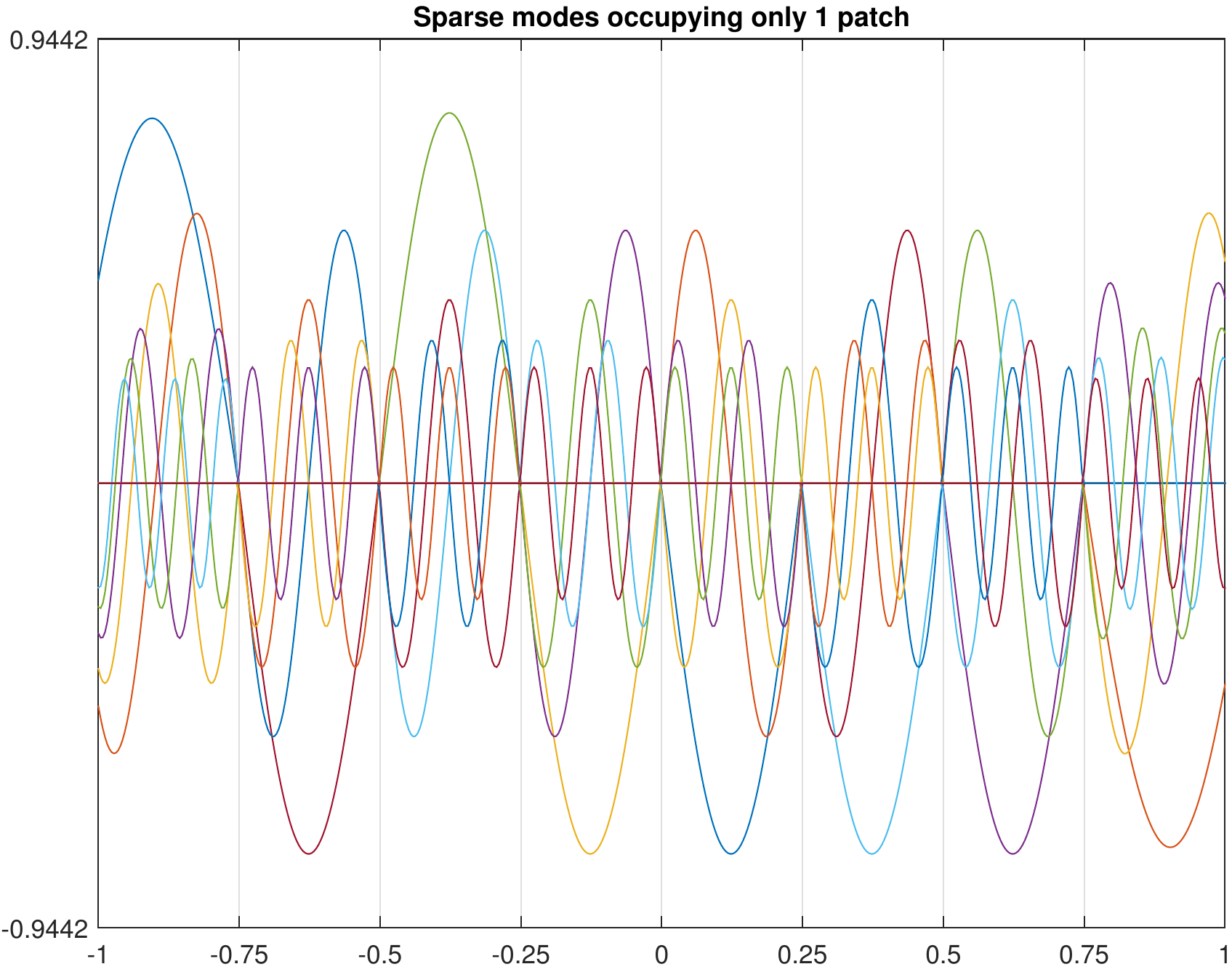}
\caption{Upper: Two patches case. Error is $4.95\%$. Global dimension is 45 and the local dimension is 23 for both patches. 
Middle: Four patches case. Error is $4.76\%$. Global dimension is 47 and the local dimension is 12, 13, 13, 12 respectively.
Bottom: Eight patches case. Error is $4.42\%$. Global dimension is 49 and the local dimension is 7 for all patches.}\label{fig:1dnum}
\end{figure}

For this translational invariant Poisson kernel, the semi-definite relaxation of sparse PCA (problem \eqref{eqn:eigvariationalL1_relax}) also gives satisfactory sparse approximation in the sense of problem~\eqref{eqn:eigvariationalL1_matrix}. Numerical tests show that when $\mu < 2$, sparse PCA tends to put too much weight on the sparsity and it leads to poor approximation to $A$ (over $90\%$ error). In Figure~\ref{fig:sparsePCA1} we plot $47$ physical modes selected out of $513$ columns of $W$, with $\mu = 2.7826$. The error is $4.94\%$. We also show 5 out of them on the right panel. Note that we have used the translation invariance property in selecting the columns of $W$. 
\begin{figure}
\centering
\includegraphics[width = 0.45\textwidth,height = 0.2\textheight]{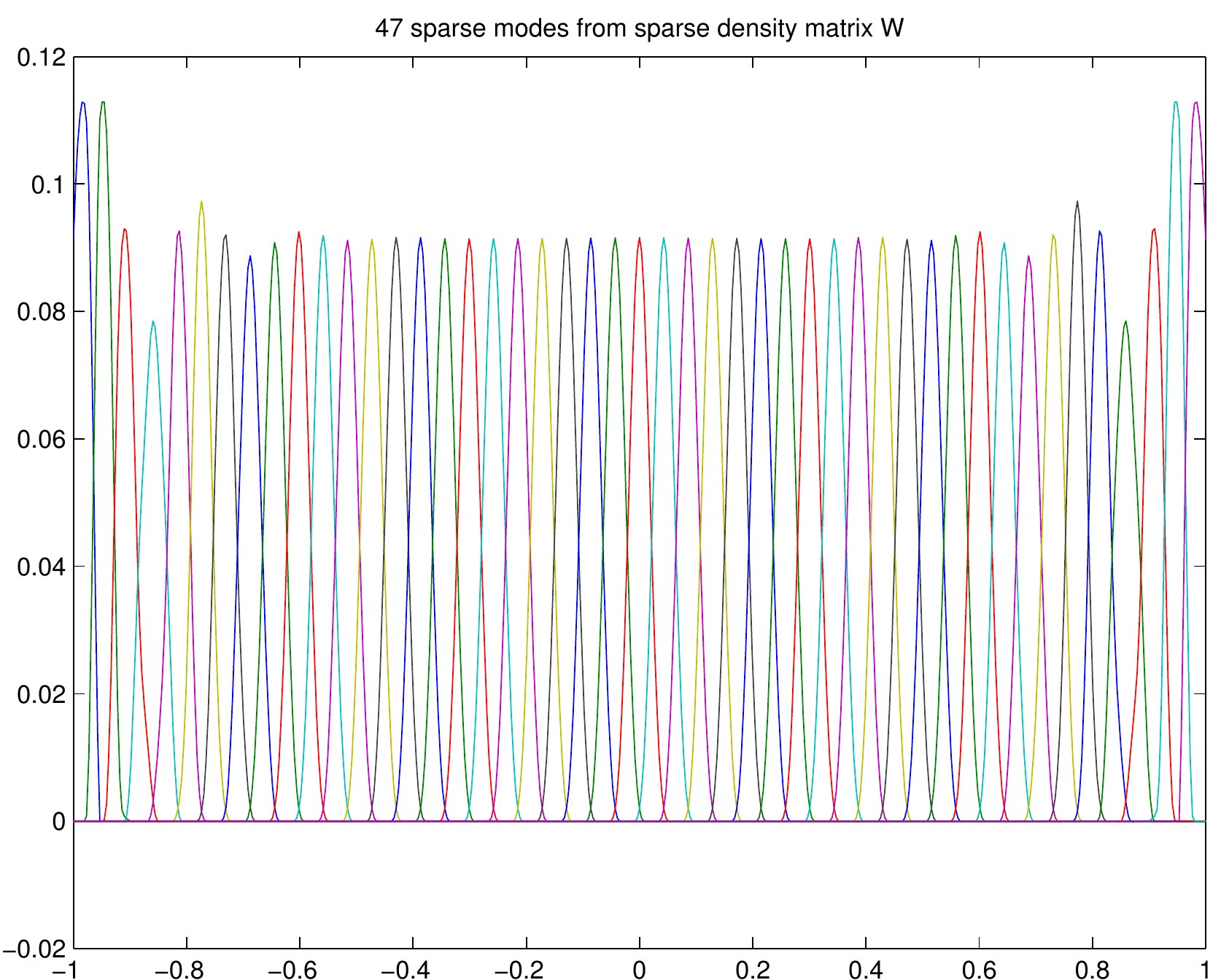}
\includegraphics[width = 0.45\textwidth,height = 0.2\textheight]{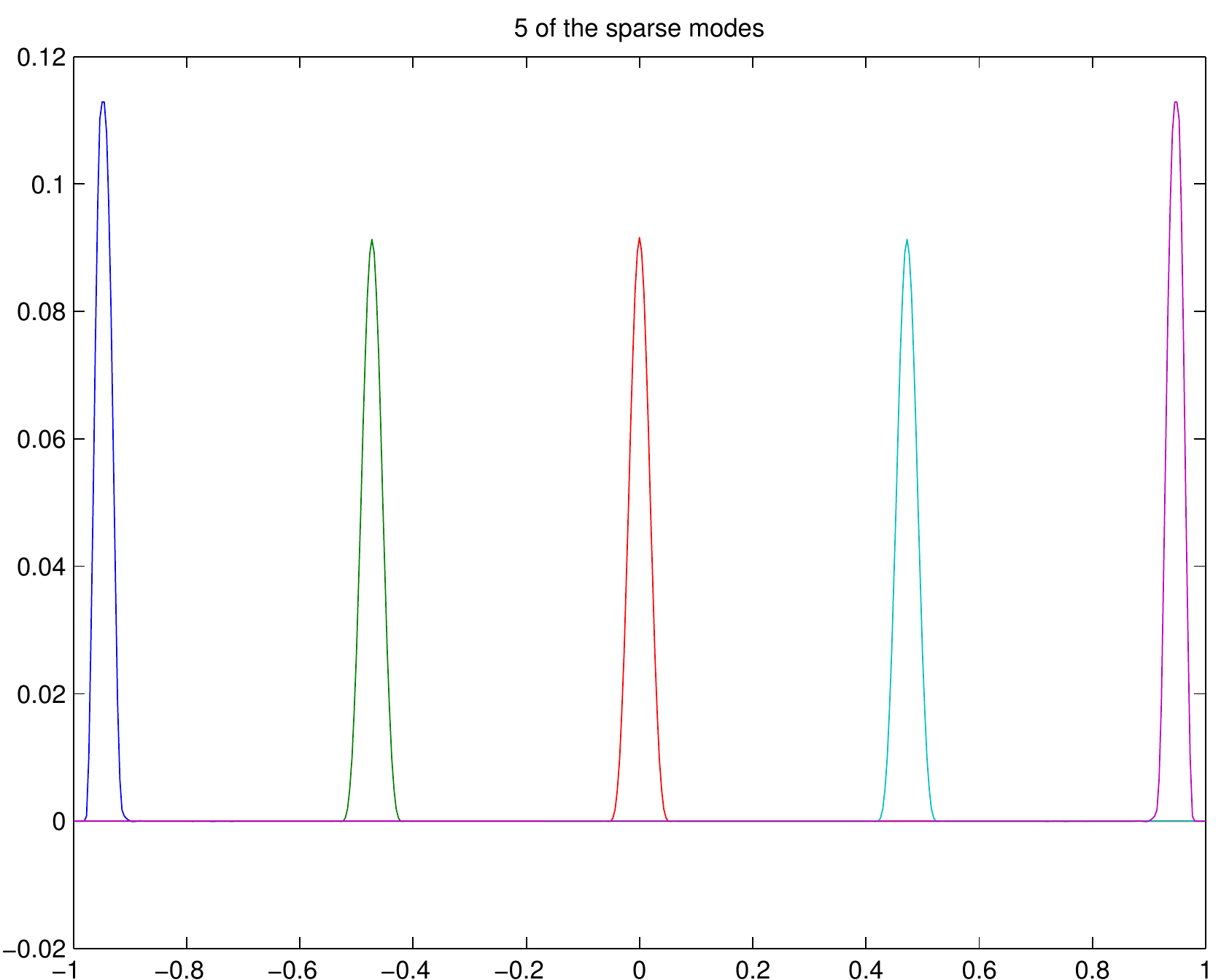}
\caption{Sparse PCA: $\mu = 2.7826$. We specifically choose 47 columns out of $W$ and show all and 5 of them.}\label{fig:sparsePCA1}
\end{figure}

\section{Conclusions and future work}
In this paper, we introduced a new matrix factorization method, the intrinsic sparse mode decomposition (ISMD), to obtain a sparse decomposition of low rank symmetric positive semidefinite matrices. Instead of minimizing the total number of nonzero entries of the decomposed modes, the ISMD minimizes the total patch-wise sparseness with a prescribed partition of index set $[N]$. The decomposed modes from the ISMD are called intrinsic sparse modes for the decomposed matrix with respect to the partition. The ISMD is equivalent to the eigen decomposition for the coarsest partition and recovers the pivoted Cholesky decomposition for the finest partition. If the partition is regular-sparse with respect to the matrix to be decomposed, we prove that the ISMD gives the optimal patch-wise sparse decomposition. We also prove that as long as the partition is regular-sparse, the decomposed modes gets sparser (in the sense of $l^0$ norm) as the partition is refined. Finally, we provide a preliminary results on perturbation analysis of the ISMD based on the assumption that the partition is regular-sparse and the intrinsic sparse modes are identifiable with each other. Numerical examples on synthetic data demonstrate the robustness and efficiency of the ISMD.

Currently, the perturbation analysis is based on an extra assumption that roughly requires that the local eigen decomposition be well conditioned, see Eqn.~\eqref{eqn:Annperturb}. It would be desirable to perform a perturbation analysis without such assumption or propose a more stable version of the ISMD. In the paper, we also discussed the differences between the sparse-orthogonal matrix factorization problem~\eqref{eqn:ismd_svd} and the general sparse matrix factorization problem~\eqref{eqn:matrix_factorize}. We pointed out that the ISMD is not designed to solve the general matrix factorization problem. The ISMD is recommended as a sparse matrix factorization method only if the orthoganality in decomposition coefficients $U$ is required and an exact (or nearly exact) decomposition is desired. Finally, we have provided a heuristic algorithm (e.g. Algorithm~\ref{alg:ISMDm2}) to solve problem~\eqref{opt:minsparseness} for matrix factorization with large noise. Ultimately, the complete resolution of this matrix factorization problem in the presence of large noise requires a better formulation and a more robust algorithm.

\appendix
\section{Proof of Proposition~\ref{prop:localGL}}\label{appendix:proofs}
\begin{enumerate}
\item
$l_k^{(\psi)}$, divided into patches, can be written as $l_k^{(\psi)} = [l_{1,k}; l_{2,k}; \cdots; l_{M,k}]$. From the definition~\eqref{eqn:defLm}, we have $\|l_{m,k}\|_1 = 1$ if $\psi_k|_{_{P_m}} \neq \vct{0}$ and 0 otherwise. Therefore, we obtain
$$\|l_k^{(\psi)}\|_1 = \sum_{m=1}^M \|l_{m,k}\|_1 = s_k(\psi_k; \CalP).$$
Moreover, on patch $P_m$ different $\psi_k$'s correspond to different local pieces in $\Psi_m$ (when they are identical, we keep both when constructing $\Psi_m$), and thus different columns in $L_m^{(\psi)}$ have disjoint supports. Therefore, different columns in $L^{(\psi)}$ have disjoint supports.

\item
From the definition~\eqref{eqn:defLm}, the $j$-th row of $L_n^{(\psi)}$ is equal to $\vct{e}_{k_j^m}^T$, where $\vct{e}_{k_j^m}$ is the $k_j^m$-th column of $\mathbb{I}_K$. Then we have $(L^{(\psi)}_n)^T L^{(\psi)}_n = \sum_{j=1}^{d_n} \vct{e}_{k_j^n} \vct{e}_{k_j^n}^T$. Therefore, we obtain
\begin{equation}\label{eqn:Bnmpsi}
B_{n;m}^{(\psi)} \equiv L_m^{(\psi)}(L^{(\psi)}_n)^T L^{(\psi)}_n (L^{(\psi)}_m)^T = \sum_{j=1}^{d_n} L_m^{(\psi)} \vct{e}_{k_j^n} (L_m^{(\psi)} \vct{e}_{k_j^n})^T =  \sum_{j=1}^{d_n} l_{m,k_j^n} l_{m,k_j^n}^T,
\end{equation} 
where $l_{m,k_j^n}$ is the $k_j^n$-th column of $L_m^{(\psi)}$. 

From the definition~\eqref{eqn:defLm}, $l_{m,k_i^m}$, the $k_i^m$-th column of $L_m^{(\psi)}$, is equal to $\vct{e}_i$ for $i \in [d_m]$ and all other columns are $\vct{0}$. Therefore, 
\begin{equation}\label{eqn:LmLmT}
\sum_{k=1}^{K} l_{m,k} l_{m,k}^T = \sum_{i=1}^{d_m} l_{m,k_i^m} l_{m,k_i^m}^T =  \sum_{i=1}^{d_m} \vct{e}_i \vct{e}_i^T = \mathbb{I}_{d_m}.
\end{equation} 
Eqn.~\eqref{eqn:Bnmpsi} sums over $k \in \{k_j^n\}_{j=1}^{d_n} \subset [K]$ and then we conclude that $B_{n;m}^{(\psi)}$ is diagonal with diagonal entries either 1 or 0. Moreover, if $B_{n;m}^{(\psi)}(i,i) = 1$ the term $\vct{e}_i \vct{e}_i^T$ has to be included in the summation in~\eqref{eqn:Bnmpsi}. Among all terms $\{l_{m,k} l_{m,k}^T\}_{k=1}^K$, only $l_{m,k_i^m} l_{m,k_i^m}^T$ is equal to $\vct{e}_i \vct{e}_i^T$ due to the definition of $L_m^{(\psi)}$. Therefore, the term $l_{m,k_i^m} l_{m,k_i^m}^T$ has to be included in the summation in~\eqref{eqn:Bnmpsi}. Therefore, there exists $j \in [d_n]$ such that $k_j^n = k_i^m$. In other words, there exist $k \in [K]$ and $j \in [d_n]$ such that $\psi_k|_{_{P_m}} = \psi_{m,i}$ and $\psi_k|_{_{P_n}} = \psi_{n,j}$. 
\end{enumerate}

\section{A simple lemma about regular-sparse partitions} \label{supp:refinepartitionlemma}
\begin{lemma}\label{lem:refinepartition}
Suppose that $A \in \RR^{N\times N}$ is symmetric and PSD. Let $\CalP_c$ be a partition of $[N]$ and $\CalP_f$ be a refinement of $\CalP_c$. If the finer partition $\CalP_f$ is regular-sparse with respect to $A$, then the coarser partition $\CalP_c$ is also regular-sparse with respect to $A$.
\end{lemma}
\begin{proof}
By the definition of regular-sparseness, suppose that $A = \sum_{k=1}^K g_k^{(f)} \left(g_k^{(f)}\right)^T$ and that on every patch in $\CalP^{(f)}$ the nontrivial modes $\{g_k^{(f)}\}_{k=1}^K$ on this patch are linearly independent. For any $P_m^{(c)} \in \CalP_c$, assume
\begin{equation} \label{eqn:coarse}
	\sum_{i=1}^{d_m} \alpha_i g_{k_i^m}^{(f)} \equiv 0 \qquad \text{on patch } P_m^{(c)},
\end{equation}
where $d_m$ is the local dimension of decomposition $A = \sum_{k=1}^K g_k^{(f)} \left(g_k^{(f)}\right)^T$ on $P_m^{(c)}$ and $\{g_{k_i^m}^{(f)}\}_{i=1}^{d_m}$ are the modes which are non zero there. Since $\CalP_f$ is a refinement of $\CalP_c$, for any $i \in [d_m]$, there exists one patch $P_n^{(f)} \subset P_m^{(c)}$ such that $g_{k_i^m}^{(f)} \neq 0$ on this smaller patch. Restricting Eqn.~\eqref{eqn:coarse} to $P_n^{(f)}$, we get $\alpha_i = 0$ due to regular-sparse property of $\CalP_f$. Therefore, $\{g_{k_i^m}^{(f)} \}_{i=1}^{d_m}$ are linearly independent on $P_m^{(c)}$. Since the patch $P_m^{(c)}$ is arbitrarily chosen, we conclude that $\CalP_c$ is regular-sparse.
\end{proof}

\section{Joint diagonalization of matrices}\label{appendix:jointDiagonalization}
Joint diagonalization is often used in Blind Source Separation (BSS) and Independent Component Analysis (ICA), and it has been well studied. We adopt its algorithm and sensitivity analysis in the ISMD. Suppose a series of $n$-dimensional symmetric matrices $\{M_k\}_{k=1}^K$ can be decomposed into:
\begin{equation}\label{eqn:jointM}
	M_k = D \Lambda_k D^T,
\end{equation} 
where $D$ is an $n$-dimensional unitary matrix that jointly diagonalizes $\{M_k\}_{k=1}^K$ and the eigenvalues are stored in diagonal matrices $\Lambda_k = \text{diag}\{\lambda_1(k), \lambda_2(k), \cdots, \lambda_n(k)\}$. Denote $\vct{\lambda}_i \equiv [\lambda_i(1), \lambda_i(2), \dots, \lambda_i(K)]^T \in \RR^K$. To find the joint eigenvectors $D$, we solve the following optimization problem:
\begin{equation}\label{opt:jointDapp}
\min_{V\in \OO(n)}\quad \sum_{k=1}^{K} \sum_{i \neq j} |(V^T M_k V)_{i,j}|^2.
\end{equation}

Obviously the minimum of problem~\eqref{opt:jointDapp} is 0 and $D$ is an minimizer. However, the minimizer is not unique. The so-called unicity assumption, i.e., $\vct{\lambda}_i \neq \vct{\lambda}_j$ for any $i \neq j$, is widely used in existing literatures and guarantees that $D$ is unique up to column permutation and sign flips. In general, we assume that there are $m$ ($m \le n$) distinct eigenvalues $\{\vct{\lambda}_i\}_{i=1}^m$ with multiplicity $\{q_i\}_{i=1}^m$ respectively. Minimizers of problem~\eqref{opt:jointDapp} are characterized by the following theorem.
\begin{theorem}\label{thm:jointDunique}
Suppose that $\{M_k\}_{k=1}^K$ are generated by \eqref{eqn:jointM} and that $V$ is a global minimizer of problem~\eqref{opt:jointDapp}. There exists a permutation matrix $\Pi\in\RR^{n\times n}$ and block diagonal matrix $R$ such that 
\begin{equation}
V \Pi = D R \,, \qquad R = \text{diag}\{R_1, \dots, R_m\}\,,
\end{equation}
in which $R_i \in \OO(q_i)$.
\end{theorem}

Theorem~\ref{thm:jointDunique} is the generalization of eigen decomposition of a single symmetric matrix to the case with multiple matrices. Although it is elementary, we provide the sketch of its proof here for completeness.

\begin{proof}
Since $V$ is a global minimizer and thus achieves zero in its objective function, $V^T M_k V$ is diagonal for any $k\in [K]$. Denote $\Gamma \equiv V^T M_k V = \text{diag}\{\gamma_1(k), \gamma_2(k), \cdots, \gamma_n(k)\}$ and $\vct{\gamma}_i \equiv [\gamma_i(1), \gamma_i(2), \dots, \gamma_i(K)]^T \in \RR^K$. Define $D = [d_1, d_2, \dots, d_n]$ and $V = [v_1, v_2, \dots, v_n]$. If $\vct{\gamma}_i \neq \vct{\lambda}_j$, then $v_i^T d_j = 0$ since they belong to different eigen spaces for at least one $M_k$. Both $D$ and $V$ span the full space $\RR^n$, and thus there is a one-to-one mapping between $\{\vct{\gamma}_i\}_{i=1}^n$ to $\{\vct{\lambda}_i\}_{i=1}^m$ with multiplicity $\{q_i\}_{i=1}^m$. Therefore, there exists a permutation matrix $\Pi$ such that
\begin{equation*}
	\left[\vct{\gamma}_1, \vct{\gamma}_2, \dots, \vct{\gamma}_n \right] \Pi = \left[ \vct{\lambda}_1, \vct{\lambda}_2, \dots, \vct{\lambda}_n \right].
\end{equation*}
Correspondingly, denoting $\tilde{D} = [\tilde{d}_1, \tilde{d}_2, \dots, \tilde{d}_n] \equiv V \Pi$, we have
\begin{equation*}
	M_k d_{i,j} = \lambda_i(k) d_{i,j}\, , \qquad M_k \tilde{d}_{i,j} = \lambda_i(k) \tilde{d}_{i,j}\, ,
\end{equation*}
where $\{d_{i,j}\}_{j=1}^{q_i}$ and $\{\tilde{d}_{i,j}\}_{j=1}^{q_i}$ are the eigenvectors in $D$ and $\tilde{D}$ respectively corresponding to the eigenvalue $\vct{\lambda}_i$. By orthogonality between eigenspaces and completeness of $D$ and $\tilde{D}$, $\{d_{i,j}\}_{j=1}^{q_i}$ and $\{\tilde{d}_{i,j}\}_{j=1}^{q_i}$ must span the same $q_i$-dimensional subspace. Since both  $\{d_{i,j}\}_{j=1}^{q_i}$ and $\{\tilde{d}_{i,j}\}_{j=1}^{q_i}$ are orthonormal, there exists $R_i \in \OO(q_i)$ such that $\tilde{d}_{i,j} = R_i d_{i,j}$ for $j \in [q_i]$. 
\end{proof}

The sensitivity analysis of the joint diagonalization problem~\eqref{opt:jointDapp} is studied in \cite{cardoso_perturbation_1994}, and we directly quote its main results below.
\begin{proposition}\label{prop:jointDsensitivity}
Suppose that $\{\hat{M}_k\}_{k=1}^K$ are generated as follows:
\begin{equation*}
	\hat{M}_k = M_k + \epsilon \tilde{M}_k,\quad M_k = D \Lambda_k D^T,
\end{equation*} 
where $D$ is unitary, $\epsilon$ is a real scalar, matrices $\tilde{M}_k$ are arbitrary and matrices $\Lambda_k$ are diagonal as in \eqref{eqn:jointM}. Suppose that the unicity assumption, i.e., $\vct{\lambda}_i \neq \vct{\lambda}_j$ for any $i \neq j$, holds true. Then any solution of the joint diagonalization problem~\eqref{opt:jointDapp} with the perturbed input $\{\hat{M}_k\}_{k=1}^K$, denoted by $\hat{D}$, is in the form
\begin{equation*}
\hat{D} = D(\mathbb{I} + \epsilon E +o(\epsilon)) J
\end{equation*}
where $J$ is the product of a permutation matrix with a diagonal matrix having only $\pm 1$ on its diagonal. Matrix $E$ has a null diagonal and is antisymmetric, i.e., $E + E^T = 0$. Its off-diagonal entries $E_{ij}$ are give by
\begin{equation*}
E_{ij} = \frac{1}{2}\sum_{k=1}^K f_{ij}(k) d^T_i(\tilde{M}_k+\tilde{M}_k^T)d_j\,,\quad\text{with}\quad f_{ij}(k) = \frac{\lambda_j(k) - \lambda_i(k)}{\sum_{l=1}^K(\lambda_j(l) - \lambda_i(l))^2}\,.
\end{equation*}
\end{proposition}

In this paper, we solve problem~\eqref{opt:jointDapp} using a Jacobi-like algorithm proposed in~\cite{cardoso_blind_1993,bunse-gerstner_numerical_1993}. The idea is to perform 2-dimensional rotation to reduce the amplitude of the off-diagonal pairs one by one. Denote by $R = R(p,q,c,s)$ the 2-dimensional rotation that deals with $(p,q)$ entries of $M_k$:
\begin{equation} \label{eqn:planerotation}
	R = R(p,q,c,s) = I + (c-1) \vct{e}_p \vct{e}_p^T - s \vct{e}_q \vct{e}_q^T + s \vct{e}_q \vct{e}_p^T + (c-1) \vct{e}_p \vct{e}_q^T,
\end{equation}
where $c^2 + s^2 = 1$ for unitarity. A simple calculation shows that
\begin{equation}\label{eqn:offchange}
\begin{split}
	\sum_{k=1}^{K} \sum_{i \neq j} |(R^T M_k R)_{i,j}|^2 = & \sum_{k=1}^{K} \sum_{i \neq j} |M_k(i,j)|^2 - \sum_{k=1}^{K} \left( |M_k(p,q)|^2 + |M_k(q,p)|^2 \right) \\
	   & + \sum_{k=1}^{K} \left( s c (M_k(q,q) - M_k(p,p)) + c^2 M_k(p,q) - s^2 M_k(q,p)\right)^2 \\
	   & + \sum_{k=1}^{K} \left( s c (M_k(q,q) - M_k(p,p)) - s^2 M_k(p,q) + c^2 M_k(q,p)\right)^2\,.
\end{split}
\end{equation}
It can be shown that the choice of $c$ and $s$ that minimizes~\eqref{eqn:offchange} also minimizes $\|L_{pq}z\|_2$ where $z = \left[c^2-s^2,2cs\right]^T$ is a $2\times 1$ vector, and
\begin{equation} \label{eqn:defineL}
	L_{pq} : = \begin{bmatrix} M_1(p,q) & \frac{M_1(q,q)-M_1(p,p)}{2} \\ \vdots & \vdots \\ M_K(p,q) & \frac{M_K(q,q)-M_K(p,p)}{2} \end{bmatrix}\,,
\end{equation}
is a $K\times 2$ matrix. It is apparent that the singular vector corresponding to the smallest singular value does the job. Denote this singular vector by $\vct{w}$ with $\vct{w}(1)\ge 0$. The optimizer of Eqn.~\eqref{eqn:offchange} is given by:
\begin{equation} \label{eqn:choiceCS}
	c = \sqrt{\frac{1+\vct{w}(1)}{2}}, \quad s = \frac{\vct{w}(2)}{2 c}.
\end{equation}

We perform such rotation for each pair of $(p,q)$ until the algorithm converges, as shown in Algorithm~\ref{alg:jointD}.
\begin{algorithm}
\caption{Jacobi-like Joint Diagonalization}\label{alg:jointD}
\begin{algorithmic}[1]
\Require $\epsilon > 0$; $\{M_k\}_{k=1}^K$, which are symmetric and jointly diagonalizable.
\Ensure $V\in \OO(n)$ such that $\sum_{k=1}^{K} \sum_{i \neq j} |(V^T M_k V)_{i,j}|^2 \le \epsilon \sum_{k=1}^K \|M_k\|_F^2$.
\State $V \leftarrow I$
\While{$\sum_{k=1}^{K} \sum_{i \neq j} |(V^T M_k V)_{i,j}|^2 > \epsilon \sum_{k=1}^K \|M_k\|_F^2$}
	\For{$p = 1, 2, \cdots, n$} 
		\For{$q = p+1, p+2, \cdots, n$} 
			\State define $L_{pq}$ as in~\eqref{eqn:defineL}
			\State compute $\vct{w}$, the normalized singular vector corresponding to the smallest singular value
			\State set $c = \sqrt{\frac{1+\vct{w}(1)}{2}}, \quad s = \frac{\vct{w}(2)}{2 c}$ and $R = R(p,q,c,s)$
			\State set $V \leftarrow V R$; $M_k \leftarrow V^T M_k V$ for $k=1, 2, \cdots, K$
		\EndFor
	\EndFor
\EndWhile
\end{algorithmic}
\end{algorithm}
The algorithm has been shown to have quadratic asymptotic convergence rate and is numerically stable, see~\cite{bunse-gerstner_numerical_1993}.

\section{Proof of Lemma~\ref{lem:Gmperturb}} \label{appendix:lemproof}
We point out that for the noiseless case, the ISMD in fact solves the following optimization problem to obtain $G_m$:
\begin{equation}\label{opt:jointGm}
\boxed{\begin{split}
	\min_{G_m \in \RR^{|P_m| \times K_m}} \quad & \sum_{n=1}^{M} \sum_{i \neq j} |B_{n;m}(i,j)|^2 \, \\
	\text{s.t.}   \quad & G_m G_m^T = A_{mm}\,, \\
			& B_{n;m} = G_m^{\dagger} A_{mn} A_{nn}^{\dagger} A_{mn}^T \left(G_m^{\dagger} \right)^T,
\end{split}}
\end{equation}
in which
\begin{equation}\label{eqn:Ainverse}
	G_m^{\dagger} = (G_m^T G_m)^{-1} G_m^T, \quad A_{nn}^{\dagger} = \sum_{i=1}^{K_n} \gamma_{n,i}^{-1} h_{n,i} h_{n,i}^T
\end{equation}
is the (Moore-Penrose) pseudo-inverse of $G_m$ and $A_{nn}$ respectively. The ISMD solves this optimization problem in two steps:
\begin{enumerate}
\item Perform eigen decomposition $A_{mm} = H_m H_m^T$. Then the feasible $G_m$ can be written as $H_m D_m$ with unitary matrix $D_m$.
\item Find the rotation $D_m$ which solves the joint diagonalization problem~\eqref{opt:jointDm}.
\end{enumerate}
Similarly, one can check that for the noisy case, the ISMD (with truncated eigen decomposition~\eqref{eqn:localeigtruncate}) solves the same optimization problem with perturbed input to obtain $\hat{G}_m$:
\begin{equation}\label{opt:jointGmhat}
\boxed{\begin{split}
	\min_{G_m \in \RR^{|P_m| \times K_m}} \quad & \sum_{n=1}^{M} \sum_{i \neq j} |B_{n;m}(i,j)|^2 \, \\
	\text{s.t.}   \quad & G_m G_m^T = \hat{A}_{mm}^{(t)}\,, \\
			& B_{n;m} = G_m^{\dagger} \hat{A}_{mn} \left(\hat{A}_{nn}^{(t)}\right)^{\dagger} \hat{A}_{mn}^T \left(G_m^{\dagger} \right)^T,
\end{split}}
\end{equation}
where, $\hat{A}_{nn}^{(t)}$ is the truncated $\hat{A}_{nn}$ defined in Eqn.~\eqref{eqn:truncatedAmm} and 
\begin{equation}\label{eqn:Ainversehat}
	\left(\hat{A}_{nn}^{(t)}\right)^{\dagger} = \sum_{i=1}^{K_n} \hat{\gamma}_{n,i}^{-1} \hat{h}_{n,i} \hat{h}_{n,i}^T
\end{equation}
is the pseudo-inverse of $\hat{A}_{nn}^{(t)}$.

Since $G_m$ is a minimizer of problem~\eqref{opt:jointGm}, the identity matrix $\mathbb{I}_{K_m}$ is one minimizer of the following joint diagonalization problem:
\begin{equation}\label{app:jointDm}
\boxed{\begin{split}
	\min_{V \in \OO(K_m)}\quad \sum_{n=1}^{M} \sum_{i \neq j} |(V^T B_{n;m} V)_{i,j}|^2 \,,
\end{split}}
\end{equation}
where 
\begin{equation}\label{app:B}
	B_{n;m} = G_m^{\dagger} A_{mn} A_{nn}^{\dagger} A_{mn}^T \left(G_m^{\dagger} \right)^T = D_m^T \Sigma_{n;m} D_m,
\end{equation}
where $D_m$ and $\Sigma_{n;m}$ are defined in the procedure of the ISMD. Let $\{\psi_k\}_{k=1}^K$ be a set of intrinsic sparse modes of $A$. Combining Lemma~\ref{lem:necessaryD} with Lemma~\ref{lem:sufficientD},  we get
\begin{equation}\label{app:Bdiagonal}
\begin{split}
	B_{n;m} &= D_m^T \Sigma_{n;m} D_m = \Pi_m V_m^T  \left(D^{(\psi)}\right)^T \Sigma_{n;m} D^{(\psi)} V_m \Pi_m = \Pi_m V_m^T B_{n;m}^{(\psi)} V_m \Pi_m = \Pi_m B_{n;m}^{(\psi)} \Pi_m.
\end{split}	
\end{equation}
The last equality is due to the fact that $V_m$ are diagonal matrices with diagonal entries either 1 or -1 in the identifiable case.\footnote{Readers can verify that Eqn.~\eqref{app:Bdiagonal} is still true in the non-identifiable case.} If $\Psi_m$ is reordered by $\Pi_m$, we simply have $B_{n;m} = B_{n;m}^{(\psi)}$ for all $n \in [M]$. Therefore, there exists such a set of intrinsic sparse modes $\{\psi_k\}_{k=1}^K$ that for all $n\in [M]$
\begin{equation}\label{app:Bdiagonal2} 
	B_{n;m} = B_{n;m}^{(\psi)}.
\end{equation}
One can easily verify that the unicity assumption holds true for the joint diagonalization problem~\eqref{app:jointDm} because the intrinsic sparse modes $\{\psi_k\}_{k=1}^K$ are pair-wisely identifiable. 

Combining the equality constraints in problem~\eqref{opt:jointGm} and problem~\eqref{opt:jointGmhat} and the assumption~\eqref{eqn:Annperturb}, we have
\begin{equation*}
	\hat{G}_m \hat{G}_m^T = \left((I + \epsilon E_m^{(eig)}) G_m\right) \left((I + \epsilon E_m^{(eig)}) G_m \right)^T.
\end{equation*}
Define 
\begin{equation}\label{app:defF}
	F_m \equiv (I + \epsilon E_m^{(eig)}) G_m.
\end{equation}
Then, there exists $U_m \in \OO(K_m)$ such that $\hat{G}_m = F_m U_m$. Since $\hat{G}_m$ is a minimizer of problem~\eqref{opt:jointGmhat}, $U_m$ is one minimizer of the following joint diagonalization problem:
\begin{equation}\label{app:jointDmhat}
\boxed{\begin{split}
	\min_{V \in \OO(K_m)}\quad \sum_{n=1}^{M} \sum_{i \neq j} |(V^T \hat{B}_{n;m} V)_{i,j}|^2 \,,
\end{split}}
\end{equation}
where
\begin{equation}\label{app:Bhat}
	\hat{B}_{n;m} = F_m^{\dagger} \hat{A}_{mn} \left(\hat{A}_{nn}^{(t)}\right)^{\dagger} \hat{A}_{mn}^T \left(F_m^{\dagger} \right)^T.
\end{equation}
From standard perturbation analysis of pseudo-inverse, for instance see Theorem 3.4 in~\cite{stewart1977perturbation}, we have
\begin{equation}\label{app:Ginvperturb}
	F_m^{\dagger} = G_m^{\dagger} + \epsilon E_m^{(ginv)}, \quad \|E_m^{(ginv)}\|_2 \le \mu \sigma_{min}^{-2}(G_m) \|E_m^{(eig)} G_m\|_2 \le \mu C_{eig} \sigma_{min}^{-2}(G_m) \|G_m\|_2
\end{equation}
and
\begin{equation*}
	\left(\hat{A}_{nn}^{(t)}\right)^{\dagger} = A_{nn}^{\dagger} + \epsilon E_n^{(ainv)}, \quad \|E_n^{(ainv)}\|_2 \le \mu \gamma_{n,K_n}^{-2} \|\hat{A}_{nn}^{(t)} - A_{nn}\|_2/\epsilon.
\end{equation*}
Here, $\sigma_{min}(G_m)$ is the smallest nonzero singular value of $G_m$ and $\gamma_{n,K_n}$ is the $K_n$-th eigenvalue of $A_{nn}$ as defined in~\eqref{eqn:localeig}. Denote the $(K_n+1)$-th eigenvalue of $\hat{A}_{nn}$ as $\hat{\gamma}_{n,K_n+1}$. From Corollary 8.1.6 in~\cite{golub_matrix_2012}, we have $\hat{\gamma}_{n,K_n+1} \le \epsilon \|\tilde{A}_{nn}\|_2$. Then, we get
\begin{equation*}
	\|\hat{A}_{nn}^{(t)} - A_{nn}\|_2 \le \|\hat{A}_{nn}^{(t)} - \hat{A}_{nn}\|_2 + \|\hat{A}_{nn} - A_{nn}\|_2 \le 2 \epsilon \|\tilde{A}_{nn}\|_2 \le 2 \epsilon,
\end{equation*}
where $\|\tilde{A}\|_2 \le 1$ has been used in the last inequality. Therefore, we obtain 
\begin{equation}\label{app:Ainvperturb}
	\left(\hat{A}_{nn}^{(t)}\right)^{\dagger} = A_{nn}^{\dagger} + \epsilon E_n^{(ainv)}, \quad \|E_n^{(ainv)}\|_2 \le 2 \mu \gamma_{n,K_n}^{-2}.
\end{equation}
When $\epsilon \ll 1$, the constant $\mu$ can be taken as 2 in both \eqref{app:Ginvperturb} and \eqref{app:Ainvperturb}. Combining \eqref{eqn:noisyCov}, \eqref{app:Ginvperturb} and \eqref{app:Ainvperturb}, we get
\begin{equation}\label{app:Bperturb}
\begin{split}
	\hat{B}_{n;m} =& B_{n;m} + \epsilon \tilde{B}_{n;m}\, , \\
	\tilde{B}_{n;m} =& E_m^{(ginv)} A_{mn} A_{nn}^{\dagger} A_{mn}^T \left(G_m^{\dagger} \right)^T + G_m^{\dagger} \tilde{A}_{mn} A_{nn}^{\dagger} A_{mn}^T \left(G_m^{\dagger}\right)^T + G_m^{\dagger} A_{mn} E_n^{(ainv)} A_{mn}^T \left(G_m^{\dagger} \right)^T \\
	& + G_m^{\dagger} A_{mn} A_{nn}^{\dagger} \tilde{A}_{mn}^T \left(G_m^{\dagger} \right)^T + G_m^{\dagger} A_{mn} A_{nn}^{\dagger} A_{mn}^T \left(E_m^{(ginv)} \right)^T.
\end{split}
\end{equation}
By Proposition~\ref{prop:jointDsensitivity}, there exists $E_m^{(jd)} \in \RR^{K_m \times K_m}$ such that 
\begin{equation*}
	U_m = (\mathbb{I}_{K_m} + \epsilon E_m^{(jd)} +o(\epsilon)) J_m,
\end{equation*}
where $J_m$ is the product of a permutation matrix with a diagonal matrix having only $\pm 1$ on its diagonal. Matrix $E_m^{(jd)}$ has a null diagonal and is antisymmetric, i.e., $E_m^{(jd)} + \left(E_m^{(jd)}\right)^T = 0$. Its off-diagonal entries $E_m^{(jd)}(i,j)$ are given by
\begin{equation*}
E_m^{(jd)}(i,j) = \sum_{n=1}^M f(n) \circ \tilde{B}_{n;m}\,,\quad\text{with}\quad f_{ij}(n) = \frac{B_{n;m}(j,j) - B_{n;m}(i,i)}{\sum_{n=1}^M(B_{n;m}(j,j) - B_{n;m}(i,i))^2}\,.
\end{equation*}
Here, $f(n)$ is the matrix with entries $f_{ij}(n)$ and $f(n) \circ \tilde{B}_{n;m}$ is the matrix point-wise product (also known as the Hadamard product). Notice that we take advantage of the fact that $\tilde{B}_{n;m}$ is symmetric to simplify $E_m^{(jd)}(i,j)$. Since $B_{n;m}(j,j) - B_{n;m}(i,i)$ is either $\pm 1$ or 0, $|f_{ij}(n)| \le 1$ for any $i$,$j$ and $n$, and thus we have $\|f(n)\|_F \le K_m$. Therefore, we conclude
\begin{equation}\label{app:Emestimate}
\|E_m^{(jd)}\|_F \le  \sum_{n=1}^M \| f(n) \circ \tilde{B}_{n;m} \|_F \le \sum_{n=1}^M \|f(n)\|_F \|\tilde{B}_{n;m}\|_F \le K_m^{3/2} \sum_{n=1}^M \|\tilde{B}_{n;m}\|_2,
\end{equation}
where we have used triangle inequality,  $\| f(n) \circ \tilde{B}_{n;m} \|_F \le \|f(n)\|_F \|\tilde{B}_{n;m}\|_F$ and $\|\tilde{B}_{n;m}\|_F \le K_m^{1/2} \|\tilde{B}_{n;m}\|_2$ in deriving the above inequalities. Combining \eqref{app:Bperturb}, \eqref{eqn:noisyCov}, \eqref{app:Ginvperturb} and \eqref{app:Ainvperturb}, we know that $\|\tilde{B}_{n;m}\|_2$ are bounded by a constant, denoted by $C_{jd}$, which only depends on $A$ and $C_{eig}$. From the assumption~\eqref{eqn:Annperturb}, $C_{eig}$ is a constant depending on $A$ but not on $\epsilon$ or $\tilde{A}$. Therefore, $C_{jd}$ depends only on $A$ but not on $\epsilon$ or $\tilde{A}$.

\section*{Acknowledgments}
This research was in part supported by Air Force MURI Grant FA9550-09-1-0613, DOE grant DE-FG02-06ER257, and NSF Grants No. DMS-1318377, DMS-1159138.

\bibliographystyle{plain}
\bibliography{ISMD}

\begin{thebibliography}{10}

\bibitem{agarwal2009regression}
Deepak Agarwal and Bee-Chung Chen.
\newblock Regression-based latent factor models.
\newblock In {\em Proceedings of the 15th ACM SIGKDD international conference
  on Knowledge discovery and data mining}, pages 19--28. ACM, 2009.

\bibitem{bunse-gerstner_numerical_1993}
Angelika Bunse-Gerstner, Ralph Byers, and Volker Mehrmann.
\newblock Numerical methods for simultaneous diagonalization.
\newblock {\em SIAM Journal on Matrix Analysis and Applications},
  14(4):927--949, 1993.

\bibitem{candes_robust_2011}
Emmanuel~J. Cand\`{e}s, Xiaodong Li, Yi~Ma, and John Wright.
\newblock Robust principal component analysis?
\newblock {\em Journal of the {ACM}}, 58(3):11:1--11:37, June 2011.

\bibitem{cardoso_blind_1993}
J.~F. Cardoso and A.~Souloumiac.
\newblock Blind beamforming for non-gaussian signals.
\newblock {\em IEE Proceedings F (Radar and Signal Processing)},
  140(6):362--370(8), December 1993.

\bibitem{cardoso_perturbation_1994}
Jean-Francois Cardoso.
\newblock Perturbation of joint diagonalizers.
\newblock Technical Report 94D023, Signal Department, Telecom Paris, Paris,
  1994.

\bibitem{chandrasekaran_rank-sparsity_2011}
Venkat Chandrasekaran, Sujay Sanghavi, Pablo~A. Parrilo, and Alan~S. Willsky.
\newblock Rank-sparsity incoherence for matrix decomposition.
\newblock {\em {SIAM} Journal on Optimization}, 21(2):572--596, April 2011.

\bibitem{xiu2015}
Yi~Chen, John Jakeman, Claude Gittelson, and Dongbin Xiu.
\newblock Local polynomial chaos expansion for linear differential equations
  with high dimensional random inputs.
\newblock {\em SIAM Journal on Scientific Computing}, 37(1):A79--A102, 2015.

\bibitem{dAspremont_sparsePCA}
A.~d'Aspremont, L.~El~Ghaoui, M.~Jordan, and G.~Lanckriet.
\newblock A direct formulation for sparse pca using semidefinite programming.
\newblock {\em SIAM Review}, 49(3):434--448, 2007.

\bibitem{efendiev2011multiscale}
Yalchin Efendiev, Juan Galvis, and Xiao-Hui Wu.
\newblock Multiscale finite element methods for high-contrast problems using
  local spectral basis functions.
\newblock {\em Journal of Computational Physics}, 230(4):937--955, 2011.

\bibitem{galvis_domain_2010}
Juan Galvis and Yalchin Efendiev.
\newblock Domain decomposition preconditioners for multiscale flows in high
  contrast media: Reduced dimension coarse spaces.
\newblock {\em Multiscale Modeling \& Simulation}, 8(5):1621--1644, January
  2010.

\bibitem{gao2012sparse}
Chuan Gao and Barbara~E Engelhardt.
\newblock A sparse factor analysis model for high dimensional latent spaces.
\newblock In {\em NIPS: Workshop on Analysis Operator Learning vs. Dictionary
  Learning: Fraternal Twins in Sparse Modeling}, 2012.

\bibitem{ghommem_mode_2013}
Mehdi Ghommem, Michael Presho, Victor~M. Calo, and Yalchin Efendiev.
\newblock Mode decomposition methods for flows in high-contrast porous media.
  global--local approach.
\newblock {\em Journal of Computational Physics}, 253:226--238, November 2013.

\bibitem{golub_matrix_2012}
G.H. Golub and C.F. Van~Loan.
\newblock {\em Matrix Computations}.
\newblock Matrix Computations. Johns Hopkins University Press, 2012.

\bibitem{hou_multiscale_1997}
Thomas~Y. Hou and Xiao-Hui Wu.
\newblock A multiscale finite element method for elliptic problems in composite
  materials and porous media.
\newblock {\em Journal of Computational Physics}, 134(1):169 -- 189, 1997.

\bibitem{hou_PetrovGalerkin_2004}
Thomas~Y. Hou, Xiao-Hui Wu, and Yu~Zhang.
\newblock Removing the cell resonance error in the multiscale finite element
  method via a petrov-galerkin formulation.
\newblock {\em Communications in Mathematical Sciences}, 2(2):185--205, 06
  2004.

\bibitem{hou_LocalgPC_2016}
Y.~Thomas Hou, Qin Li, and Pengchuan Zhang.
\newblock Exploring the locally low dimensional structure in solving random
  elliptic pdes.
\newblock {\em Multiscale Modeling \& Simulation}, 2016.

\bibitem{hou_sparseOC1_2016}
Y.~Thomas Hou and Pengchuan Zhang.
\newblock Sparse operator compression of elliptic operators -- part i : Second
  order elliptic operators.
\newblock {\em preprint}, 2016.

\bibitem{hou_sparseOC2_2016}
Y.~Thomas Hou and Pengchuan Zhang.
\newblock Sparse operator compression of elliptic operators -- part ii : High
  order elliptic operators.
\newblock {\em preprint}, 2016.

\bibitem{jenatton2010structured}
Rodolphe Jenatton, Guillaume Obozinski, and Francis~R Bach.
\newblock Structured sparse principal component analysis.
\newblock In {\em AISTATS}, pages 366--373, 2010.

\bibitem{Jolliffe_2003}
Ian~T. Jolliffe, Nickolay~T. Trendafilov, and Mudassir Uddin.
\newblock A modified principal component technique based on the lasso.
\newblock {\em Journal of Computational and Graphical Statistics}, 12(3):pp.
  531--547, 2003.

\bibitem{jolliffe2002simplified}
IT~Jolliffe, M~Uddin, and SK~Vines.
\newblock Simplified eofs three alternatives to rotation.
\newblock {\em Climate Research}, 20(3):271--279, 2002.

\bibitem{karhunen_uber_1947}
K.~Karhunen.
\newblock {\em \H{U}ber lineare Methoden in der Wahrscheinlichkeitsrechnung}.
\newblock Annales Academiae scientiarum Fennicae: Mathematica - Physica.
  Universitat Helsinki, 1947.

\bibitem{kohn1959image}
W~Kohn.
\newblock Image of the fermi surface in the vibration spectrum of a metal.
\newblock {\em Physical Review Letters}, 2(9):393, 1959.

\bibitem{krzanowski1995multivariate}
W.J. Krzanowski and F.H.C. Marriott.
\newblock {\em Multivariate Analysis: Kendall's Library of Statistics, Volume
  2}.
\newblock Kendall's advanced theory of statistics. Wiley, 1995.

\bibitem{LaiLuOsher_2014}
R.~{Lai}, J.~{Lu}, and S.~{Osher}.
\newblock {Density matrix minimization with $L_1$ regularization}.
\newblock {\em Communications in Mathematical Sciences}, to appear.

\bibitem{lee1999learning}
Daniel~D Lee and H~Sebastian Seung.
\newblock Learning the parts of objects by non-negative matrix factorization.
\newblock {\em Nature}, 401(6755):788--791, 1999.

\bibitem{lee2006efficient}
Honglak Lee, Alexis Battle, Rajat Raina, and Andrew~Y Ng.
\newblock Efficient sparse coding algorithms.
\newblock In {\em Advances in neural information processing systems}, pages
  801--808, 2006.

\bibitem{lehoucq1996deflation}
Richard~B Lehoucq and Danny~C Sorensen.
\newblock Deflation techniques for an implicitly restarted arnoldi iteration.
\newblock {\em SIAM Journal on Matrix Analysis and Applications},
  17(4):789--821, 1996.

\bibitem{loeve_probability_1977}
Michel Lo\`{e}ve.
\newblock {\em Probability Theory I}.
\newblock Comprehensive Manuals of Surgical Specialties. Springer, 1977.

\bibitem{lucas2004lapack}
Craig Lucas.
\newblock Lapack-style codes for level 2 and 3 pivoted cholesky factorizations.
\newblock {\em LAPACK Working}, 2004.

\bibitem{luo_high_2011}
Xi~Luo.
\newblock High dimensional low rank and sparse covariance matrix estimation via
  convex minimization.
\newblock {\em {arXiv} preprint {arXiv}:1111.1133}, 2011.

\bibitem{mairal2010online}
Julien Mairal, Francis Bach, Jean Ponce, and Guillermo Sapiro.
\newblock Online learning for matrix factorization and sparse coding.
\newblock {\em Journal of Machine Learning Research}, 11(Jan):19--60, 2010.

\bibitem{marzari2012maximally}
Nicola Marzari, Arash~A Mostofi, Jonathan~R Yates, Ivo Souza, and David
  Vanderbilt.
\newblock Maximally localized wannier functions: Theory and applications.
\newblock {\em Reviews of Modern Physics}, 84(4):1419, 2012.

\bibitem{marzari1997maximally}
Nicola Marzari and David Vanderbilt.
\newblock Maximally localized generalized wannier functions for composite
  energy bands.
\newblock {\em Physical review B}, 56(20):12847, 1997.

\bibitem{owhadi2015multi}
Houman Owhadi.
\newblock Multi-grid with rough coefficients and multiresolution operator
  decomposition from hierarchical information games.
\newblock {\em arXiv preprint arXiv:1503.03467}, 2015.

\bibitem{owhadi_polyharmonic_2014}
Houman Owhadi, Lei Zhang, and Leonid Berlyand.
\newblock Polyharmonic homogenization, rough polyharmonic splines and sparse
  super-localization.
\newblock {\em {ESAIM}: Mathematical Modelling and Numerical Analysis},
  48(02):517--552, 2014.

\bibitem{ozolins_compressed_2013}
Vidvuds Ozoli\c{n}\v{s}, Rongjie Lai, Russel Caflisch, and Stanley Osher.
\newblock Compressed modes for variational problems in mathematics and physics.
\newblock {\em Proceedings of the National Academy of Sciences},
  110(46):18368--18373, 2013.

\bibitem{sorensen1992implicit}
Danny~C Sorensen.
\newblock Implicit application of polynomial filters in ak-step arnoldi method.
\newblock {\em Siam journal on matrix analysis and applications},
  13(1):357--385, 1992.

\bibitem{stewart1977perturbation}
GW~Stewart.
\newblock On the perturbation of pseudo-inverses, projections and linear least
  squares problems.
\newblock {\em SIAM review}, 19(4):634--662, 1977.

\bibitem{vu2013fantope}
Vincent~Q Vu, Juhee Cho, Jing Lei, and Karl Rohe.
\newblock Fantope projection and selection: A near-optimal convex relaxation of
  sparse pca.
\newblock In {\em Advances in Neural Information Processing Systems}, pages
  2670--2678, 2013.

\bibitem{wagner2012toward}
Andrew Wagner, John Wright, Arvind Ganesh, Zihan Zhou, Hossein Mobahi, and
  Yi~Ma.
\newblock Toward a practical face recognition system: Robust alignment and
  illumination by sparse representation.
\newblock {\em Pattern Analysis and Machine Intelligence, IEEE Transactions
  on}, 34(2):372--386, 2012.

\bibitem{wannier1937structure}
Gregory~H Wannier.
\newblock The structure of electronic excitation levels in insulating crystals.
\newblock {\em Physical Review}, 52(3):191, 1937.

\bibitem{weinan2010localized}
E~Weinan, Tiejun Li, and Jianfeng Lu.
\newblock Localized bases of eigensubspaces and operator compression.
\newblock {\em Proceedings of the National Academy of Sciences},
  107(4):1273--1278, 2010.

\bibitem{witten2009penalized}
Daniela~M Witten, Robert Tibshirani, and Trevor Hastie.
\newblock A penalized matrix decomposition, with applications to sparse
  principal components and canonical correlation analysis.
\newblock {\em Biostatistics}, page kxp008, 2009.

\bibitem{zhang2012sparse}
Youwei Zhang, Alexandre d'Aspremont, and Laurent El~Ghaoui.
\newblock Sparse pca: Convex relaxations, algorithms and applications.
\newblock In {\em Handbook on Semidefinite, Conic and Polynomial Optimization},
  pages 915--940. Springer, 2012.

\bibitem{Zou_PCA_06}
Hui Zou, Trevor Hastie, and Robert Tibshirani.
\newblock Sparse principal component analysis.
\newblock {\em Journal of Computational and Graphical Statistics}, 15:265--286,
  2004.

\end{thebibliography}

\end{document}